\documentclass[reqno]{amsart}
\usepackage{amssymb, amsfonts, stmaryrd, color,fancyhdr, amsmath,amsthm,amsbsy,amsfonts, titletoc,  mathrsfs, mathabx, verbatim,extarrows, wasysym, euscript, enumerate, latexsym, mathtools, epsfig, subfigure, mathabx, graphicx, color, wrapfig, hyperref,graphicx,float, bm, setspace,stackengine}
\usepackage{CJKutf8}
\numberwithin{equation}{section}
\theoremstyle{plain}
\DeclareMathAlphabet{\mathpzc}{OT1}{pzc}{m}{it}

\def\mbU{\mathbb U}
\def\mcU{\mathcal U}
\def\mcV{\mathcal V}
\def\cO{\mathcal O}
\def\msA{\mathscr A}
\def\msB{\mathscr B}
\def\bB{\mathbf B}
\def\mbD{\mathbb D}
\def\mcC{\mathcal C}

\def\cD{\mathcal D}
\def\mbn{\mathbf n}
\def\mbI{\mathbb I}

\def\mrF{\mathrm F}

\def\bE{ {\bf E}}
\def\mcE{\mathcal E}
\def\mbE{\mathbb E}
\def\mbR{{\mathbb R}}
\def\mbQ{\mathbb Q}
\def\mcQ{\mathcal Q}
\def\oB{\overline{B}}
\def\mbM{\mathbb M}
\def\mcM{\mcM}
\def\mbT{\mathbb T}
\def\mrp{{\mathrm p}}
\def\mrq{{\mathrm q}}
\def\mrL{{\mathrm L}}

\def\mbp{{\mathbf p}}
\def\mbq{{\mathbf q}}
\def\msI{{\mathscr I}}
\def\mcI{{\mathcal I}}
\def\mbV{{\mathbb  V}}
\def\cM{\mathcal M}
\def\mbL{{\mathbb L}}
\def\cG{\mathcal G}
\def\mrd{{\mathrm d }}
\def\mcZ{{\mathcal Z}}

\def\msR{\mathscr R}

\def\Hin{H^2_{\mathrm{in}}}
\def\Rem{{\mathscr R}}

\def\muckymuck{distribution }
\def\muckymucks{distributions }
\def\kpp{}

\newcommand{\beq}{\begin{equation}}
\newcommand{\eeq}{\end{equation}}
\newcommand{\beqs}{\begin{equation*}}
\newcommand{\eeqs}{\end{equation*}}
\newcommand{\h}{\hspace}
\newcommand{\normmm}[1]{{\left\vert\kern-0.15ex\left\vert\kern-0.15ex\left\vert #1 
    \right\vert\kern-0.15ex\right\vert\kern-0.15ex\right\vert}}

\newcommand{\dd}{\, \mathrm{d}}

\newcommand{\varep}{\varepsilon}
\newcommand{\p}{\partial}
\newcommand{\la}{\lambda} 
\newcommand{\La}{\Lambda}

\newcommand{\diag}{\mathrm{diag}}
\newcommand{\hook}{\mathrel{\rotatebox[origin=c]{45}{$\hookrightarrow$}}}

\newtheorem{thm}{Theorem}[section]
\newtheorem{lemma}[thm]{Lemma}
\newtheorem{cor}[thm]{Corollary}
\newtheorem{remark}[thm]{Remark}
\newtheorem{prop}[thm]{Proposition}
\newtheorem{defn}[thm]{Definition}

\newtheorem{notation}{Notation}[section]

\begin{document}
\title{\textbf{Regularized Curve Lengthening from the Strong FCH Flow} }
\author{Yuan Chen}
\address{Department of Mathematics, Michigan State University, East Lansing, MI 48824, U.S.}
\curraddr{}
\email{chenyu60@msu.edu}

\author{Keith Promislow}
\address{Department of Mathematics, Michigan State University, East Lansing, MI 48824, U.S.}
\email{promislo@msu.edu}
\thanks{K.P. acknowledges support from NSF grants DMS-1409940 and DMS-1813203}


\date{\today}

\dedicatory{}


\begin{abstract} 
We present a rigorous analysis of the transient evolution of nearly circular bilayer interfaces evolving under the thin interface limit, $\varep\ll1$, of the mass preserving $L^2$-gradient flow of the strong scaling of the functionalized Cahn-Hilliard equation.  For a domain $\Omega\subset{\mbR}^2$ we construct a bilayer manifold with boundary comprised of quasi-equilibrium of the flow and a projection onto the manifold that associates functions $u$ in an $H^2$ tubular neighborhood of the manifold with an interface $\Gamma$ embedded in $\Omega$. These interfaces, and hence the bilayer manifold, are parameterized by a finite but asymptotically large number of degrees of freedom. The manifold contains a unique, up to translation and mass constraint, equilibrium of the gradient flow whose projected interface is circular up to exponentially small corrections. The thin tubular neighborhood is forward invariant under the flow with orbits that ultimately converge to the equilibrium. Projections of these orbits yield an interfacial evolution equivalent at leading order to the regularized curve-lengthening motion characterized by normal motion {\sl against} mean curvature, regularized by a higher order Willmore expression.  The curve lengthening is driven by absorption of excess mass from the regions of $\Omega$ away from the interface, generically leading to nontrivial dynamics that are ill-posed in the $\varep\to0$ limit. 
\end{abstract}
\maketitle


\section{Introduction}
Amphiphilic molecules are surfactants that form thin, bilayer interfaces comprised of two single-molecule layers. The functionalized Cahn-Hilliard(FCH) free energy was introduced in \cite{GHPY-11} to model the free energy of mixtures of amphiphilic molecules and solvent. It generalizes the model proposed by Gompper and Goos, \cite{GG-94}, that was motivated by earlier studies of scattering data. 
The FCH free energy is given in terms of the volume fraction $u-b_-$ of the amphiphilic molecule over a domain $\Omega$ as
\begin{equation}\label{FCH}
\mathcal F(u):=\int_\Omega \frac{\varep}{2} \left(\Delta u-\frac{1}{\varep^2}W'(u)\right)^2-\varep^{p-1}\left(\frac{\eta_1}{2}|\nabla u|^2+\frac{\eta_2}{\varep^2} W(u)\right) \mrd x,
\end{equation}
where $W:\mathbb R\mapsto\mathbb R$ is a smooth tilted double well potential with local minima at $u=b_{\pm}$ with $b_-<b_+$, $W(b_-)=0>W(b_+),$ and $W''(b_-)>0$.  The state $u\equiv b_-$ corresponds
to pure solvent, while $u\equiv b_+$ denotes a maximum packing of amphiphilic molecules. The system parameters $\eta_1>0$ and $\eta_2$ 
characterize key structural properties of the amphiphilic molecules. The small positive parameter $\varep\ll 1$ characterizes the ratio of the length of the molecule to the domain size. The term $\varep^{p-1}$ is a distinguished limit of a second small parameter with the weak scaling $p=2$ balancing the Willmore-type residual of the dominant squared term and the amphiphilic structure terms, while the strong scaling $p=1$ places these latter terms in a position of dominance.  We consider the strong scaling $p=1$, and refer the  interested reader to \cite{GJXCP-12} for a detailed discussion of physical meaning of the parameters. The FCH is known to be bounded from below over subsets of $H^2(\Omega)$ that incorporate a wide range of boundary conditions, \cite{PZ-13}. 

The goal of this work is to characterize the evolution of bilayer \muckymucks under a mass-preserving gradient flow of the FCH energy.  More specifically, to any smooth, embedded curve $\Gamma\subset\Omega$ we may associate a bilayer \muckymuck $\Phi=\Phi_\Gamma\in H^2(\Omega)$ which is an approximate local minimizer of the FCH energy. In a neighborhood of $\Gamma$ the bilayer \muckymuck can be expressed as
\beq\label{form-u}
\Phi_\Gamma (x;\sigma)=\phi_0\left(z \right)+ \frac{\varep \sigma}{(W''(b_-))^2}, 
\eeq
where $z=z(x)$ is $\varep$-scaled signed distance to $\Gamma$,   $\phi_0$ is  the leading order bilayer profile defined as the unique non-trivial solution of the ODE
\begin{equation}
\label{phi0-def}
\p_z^2\phi_0=W'(\phi_0), 
\end{equation} 
that is homoclinic to $b_-$ as $z\to\pm\infty,$
and $\sigma$ is a spatial constant that determines the value of $\Phi_\Gamma$ away from $\Gamma.$ The gradient flow preserves the system mass which is set by the initial data, $u_0$. We scale the system mass with $\varep$ 
\beq
\label{def-Mass}
\int_\Omega (u-b_-)\dd x =\varep M_0,
\eeq
and constrain the length of the interface $\Gamma$ to be $O(1)$ with respect to $\varep,$ so that mass of surfactant within an $\varep$-distance of $\Gamma$ has the same scaling as the system mass. For a bilayer \muckymuck $\Phi_\Gamma$ the mass constraint relates $\sigma$ and the length of $\Gamma$  through the leading order relation
\beq\label{|gamma|-sigma}
|\Gamma| =\frac{M_0}{m_0}-\frac{|\Omega|}{(W''(b_-))^2}\cdot \frac{\sigma}{m_0}.
\eeq
where $m_0$ denotes the bilayer mass-per-unit-length,
\beq
\label{def-m0}
m_0:=\int_{\mathbb R} (\phi_0-b_-)\dd z.
\eeq

It is instructive to examine the leading order reduction of the FCH energy at $\Phi_\Gamma (\cdot; \sigma)$. Accounting for the mass-dependent slaving \eqref{|gamma|-sigma} the strong scaling of the FCH reduces to a Canham-Helfrich type energy  \cite{Ca-70, H-73} 
\beq\label{energy-W-M}
\mathcal E(\Gamma,\sigma):=
\mathcal F(\Phi_\Gamma(\cdot;\sigma))=
\frac{\nu_s}{2}\int_{\Gamma} |\kappa|^2 \dd s+\frac{\nu_b}{2\varep}\left(\sigma- \sigma_1^*\right)^2\!\!\!,
\eeq
where $\sigma_1^*$ is the leading order equilibrium value of the bulk parameter $\sigma$. The second term, scaled by $\varep^{-1}$ can be viewed as enforcing an equilibrium length on $\Gamma$ through \eqref{|gamma|-sigma}. The surface and bulk coefficients $\nu_s,\nu_b>0$ only depend upon the system parameters, the well convexity at $b_-$, and the domain size $|\Omega|,$
\begin{equation}
\label{def-hatla1*}
\sigma_1^*:=-\frac{\eta_1+\eta_2}{2m_0}m_1^2, \quad \nu_s= m_1^2, \quad \nu_b= \frac{|\Omega|}{(W''(b_-))^2},\quad
m_1=\|\phi_0'\|_{L^2(\mathbb R)}.
\end{equation}
The equilibrium value $\sigma_1^*$ can be interpreted as the value of $\sigma$ at which absorption of surfactant into a bilayer balances with ejection of surfactant out of the bilayer.  The reduced energy reflects a strong preference for a bulk density near the equilibrium value $\sigma_1^*$. This preference corresponds to a high energy of dispersed surfactant and a correspondingly high rate of absorption. Bilayer \muckymucks for which $|\sigma-\sigma_1^*|=O(\varep)$ stike a balance between interfacial curvature, curve length, and bulk density that drives the complex interfacial dynamics.


The nature of the interfacial evolution arising from gradient flows of the reduced energy \eqref{energy-W-M} is better understood through the normal velocity it induces on the interface $\Gamma.$ Accounting for the slaving $\sigma=\sigma(|\Gamma|)$ from \eqref{|gamma|-sigma}, formal arguments, \cite{CKP-18}, show that the full energy \eqref{FCH} and the reduced energy \eqref{energy-W-M} both drive a geometric gradient flow that dissipates the energy \eqref{energy-W-M}, at leading order through the regularized curve lengthening normal velocity 
\beq\label{V-normal}
V_{\textrm{RCL}}= -\varep^{-1}  m_0(\sigma -\sigma_1^*) \kappa - \nu_s \left(\Delta_s +\frac{\kappa^2}{2}\right) \kappa.  
\eeq  
If the bulk density is lower than the equilibrium value, 
$\sigma-\sigma_1^*<0$, then the surface term induces a mean curvature flow (MCF), which shortens the curve.  If the bulk density exceeds the equilibrium value, $\sigma-\sigma_1^*>0$, then the system  dissipates total energy in a curve-lengthening motion \emph{against} curvature, absorbing amphiphilic material from the bulk. We call this regime regularized curve lengthening, with the higher order Willmore term serves as a singular perturbation that regularizes the ill-posed motion against curvature. 

We address the mass-preserving $L^2(\Omega)$ gradient flow of the FCH energy \eqref{FCH}, written in terms of the chemical potential  $\mathrm F=\mrF(u)$, associated to $\mathcal F$ through the rescaled variational derivative
\beq
\label{eq-FCH-L2-p}
\mathrm F(u):= \varep^3\frac{\delta \mathcal F}{\delta u}=(\varep^2\Delta -W''(u))(\varep^2\Delta u-W'(u))+\varep^p (\eta_1\varep^2\Delta u-\eta_2 W'(u)).
\eeq
The mass-preserving FCH $L^2$-gradient flow takes the form
 \begin{equation}\label{eq-FCH-L2}
\p_t u=-\Pi_0  \mathrm F(u),
\end{equation}
subject to periodic boundary conditions on $\Omega\subset \mathbb R^2$. Here $\Pi_0$ is the zero-mass projection given by
\begin{equation}
 \Pi_0 f:=f-\left<f\right>_{L^2}, 
\end{equation}
in terms of the averaging operator
\beq
\label{def-massfunc}
\left<f\right>_{L^2}:=\frac{1}{|\Omega|}\int_\Omega f\dd x.
\eeq 
We provide a rigorous justification of this regularized curve lengthening flow via an asymptotically large dimensional center-stable manifold reduction in a vicinity of the equilibrium arising from the bilayer distribution with a circular interface $\Gamma_0.$

\begin{center}
\begin{tabular}{cccc}
\includegraphics[height=1.in]{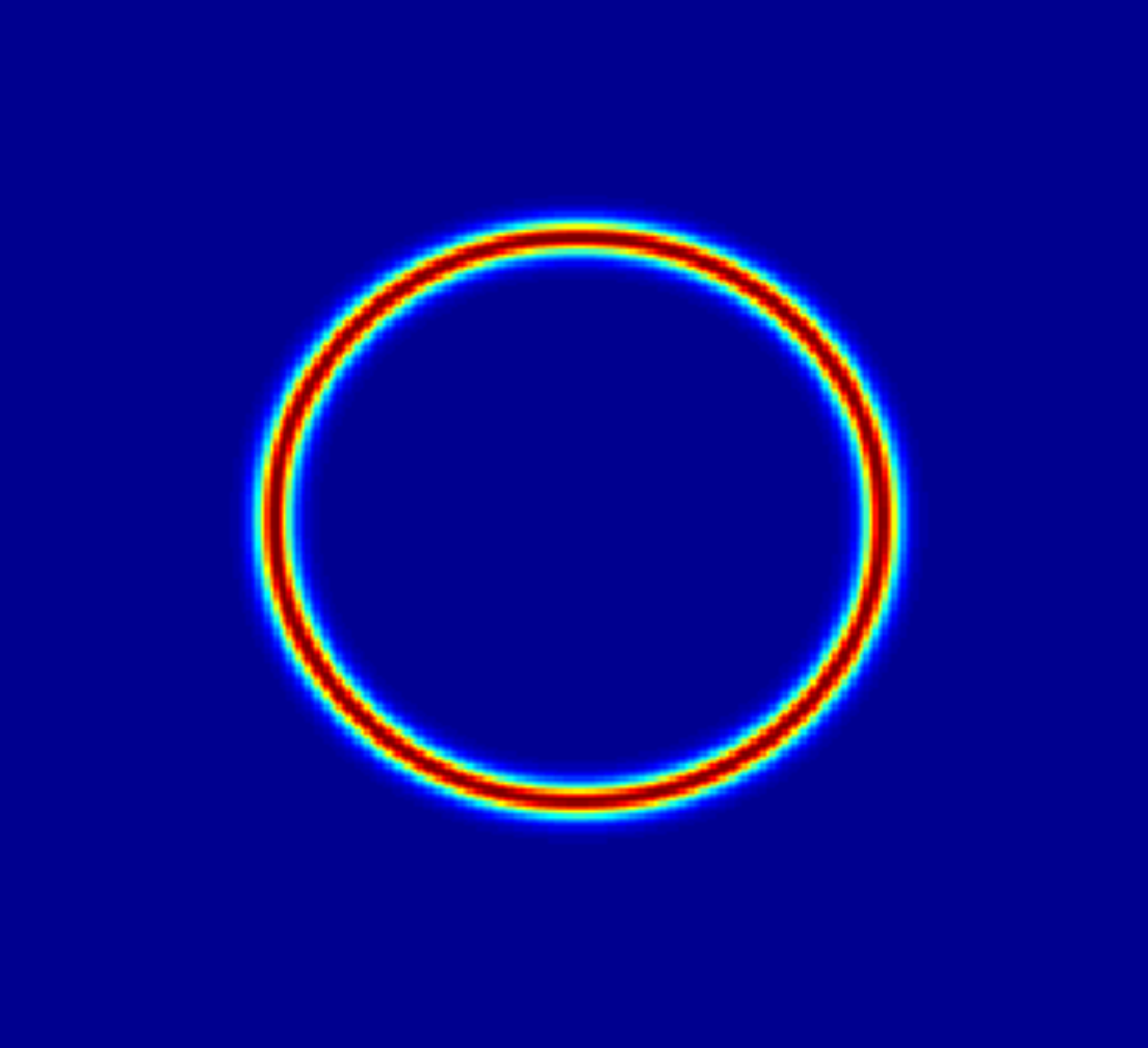} &
\includegraphics[height=1.in]{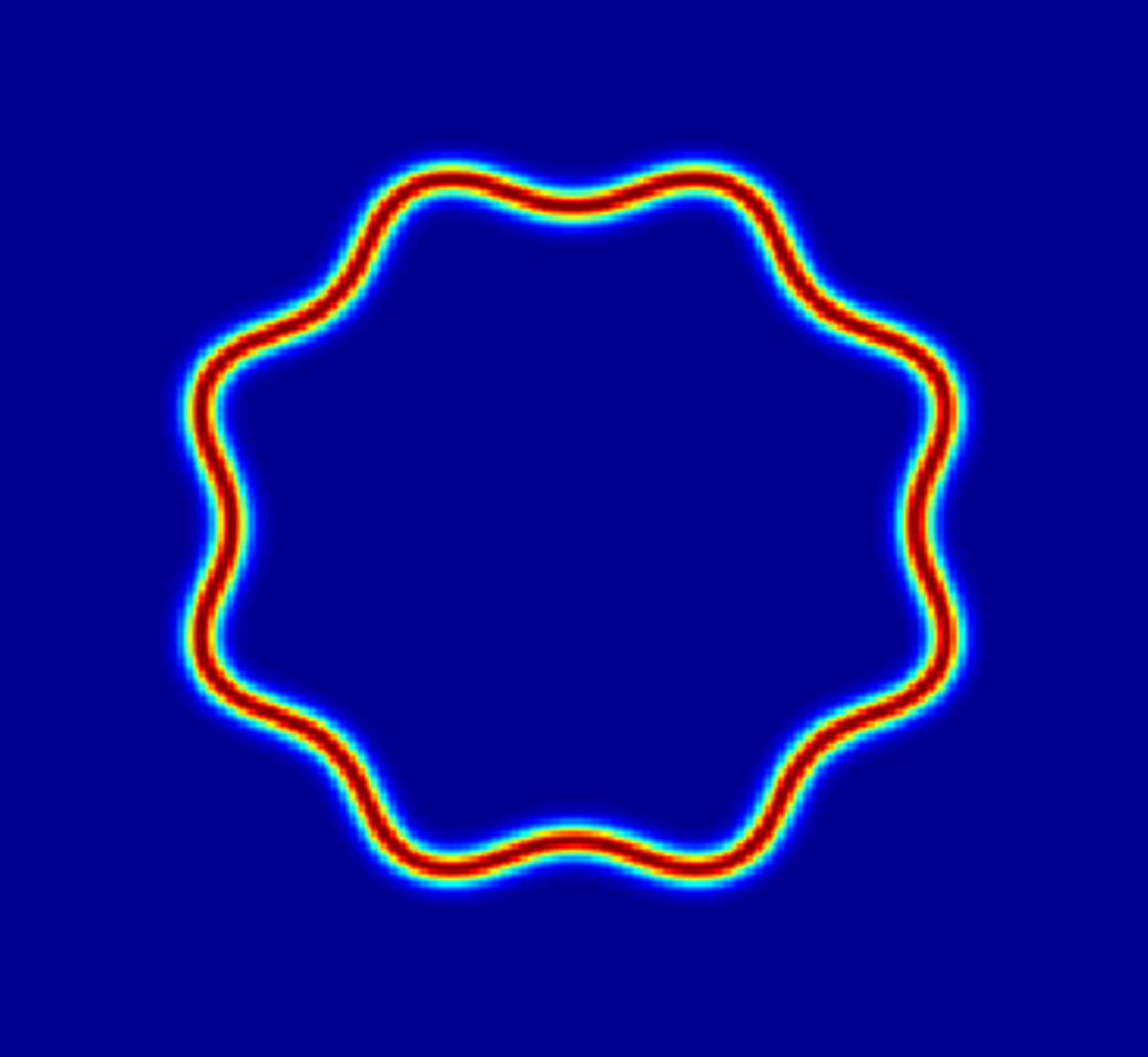}&
\includegraphics[height=1.in]{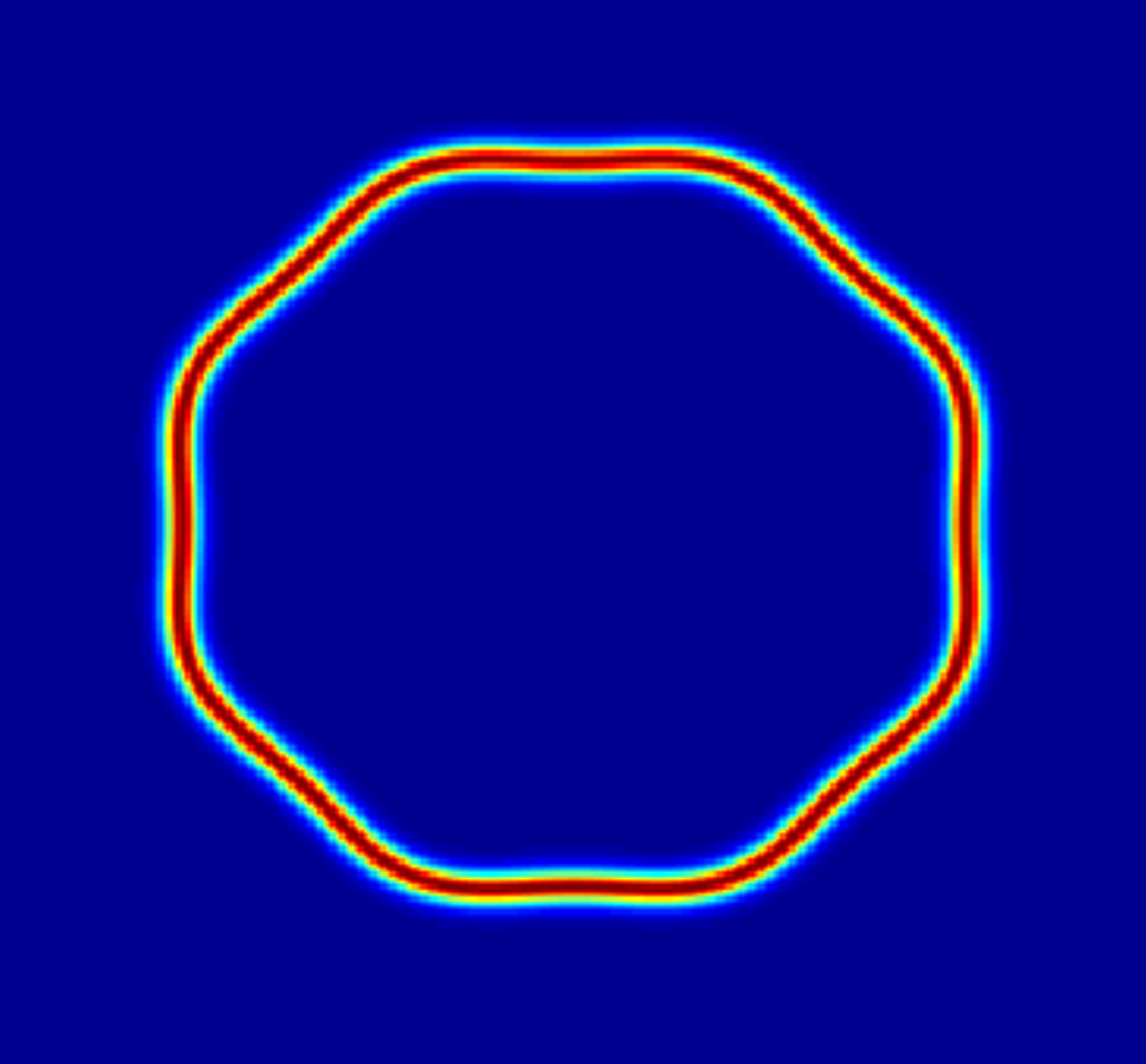}&
\includegraphics[height=1.in]{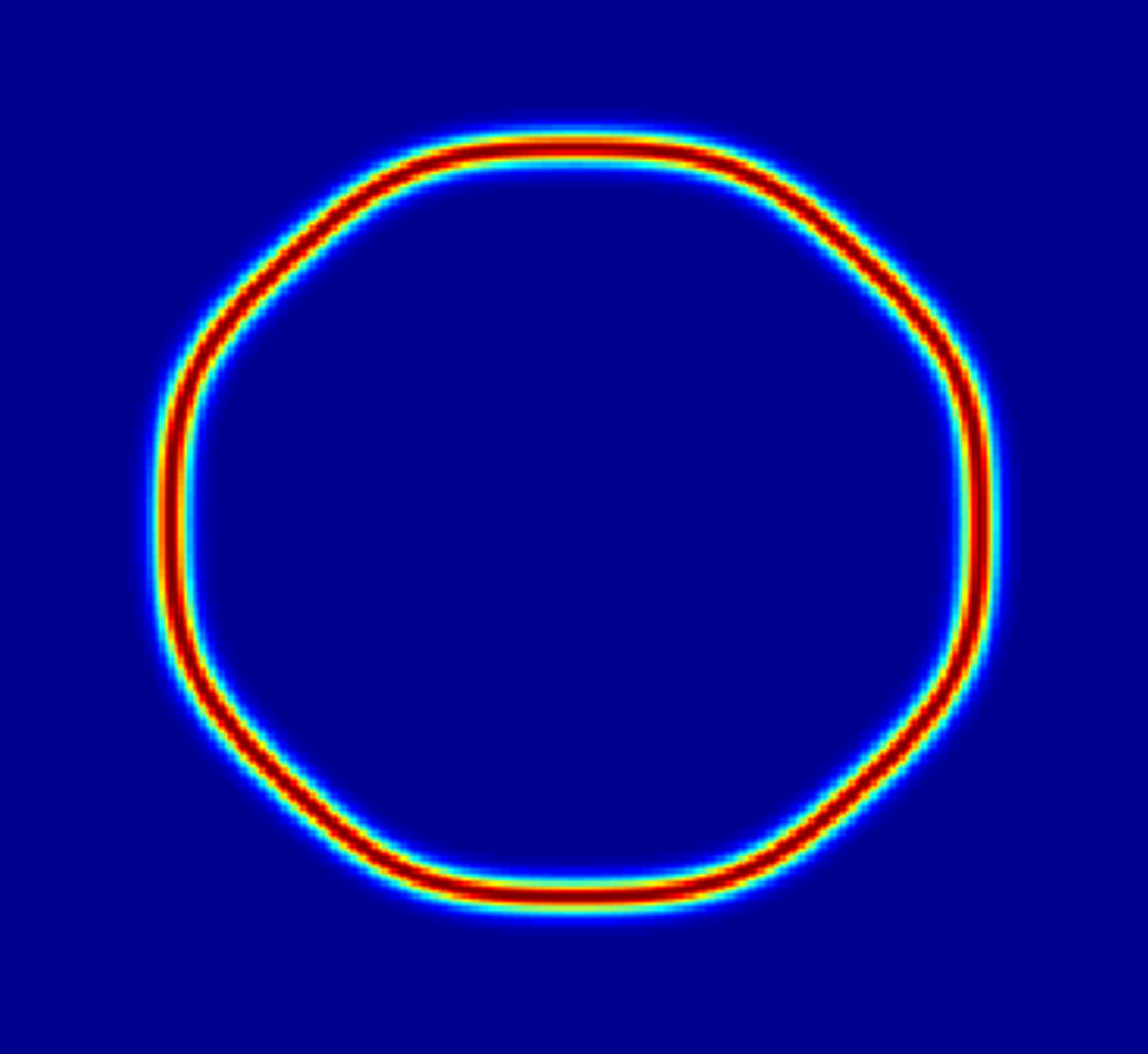}\\
(a) $t=0$ & (b) $t=367$ & (c) $t=930$ & (d) $t=1175$
\end{tabular}
\vskip 0.05in
\begin{tabular}{p{4.8in}}
Figure 1) Numerical simulation of the strong FCH mass preserving $L^2$ gradient flow on $\Omega=[-2\pi,2\pi]^2$ from initial data $u_0=\Phi_\Gamma(x;\sigma)$ with $\Gamma$ a circle and $\sigma= 2\sigma_1^*$, double the equilibrium value. Left to right, color coded contours of the evolving solution $u=u(t)$ at indicated times, show a meandering transient followed by relaxation to a circular equilibrium with larger radius. System parameters are $\varep=0.20$, $\eta_1=1.45$ and $\eta_2=2.0.$
\end{tabular}
\vskip 0.15in
\end{center}


We construct a manifold with boundary contained in $H^2(\Omega)$ whose constituent points are refinements of the  bilayer \muckymucks that are
quasi-equilibrium of the system \eqref{eq-FCH-L2}. This is the bilayer manifold, $\cM_b$, and it possesses a nonlinear projection that uniquely decomposes functions $u$ from an open neighborhood of $\cM_b$ into a point on the manifold (a bilayer \muckymuck\!\!) and a perturbation that is orthogonal to the tangent plane. The interfaces $\Gamma=\Gamma_\mbp$ that define the bilayer \muckymucks\!\!, and hence the bilayer manifold, are parameterized by an asymptotically large but finite set of meander parameters $\mbp=(\mrp_0, \cdots, \mrp_{N_1-1})$,  which  characterize the shape of $\Gamma_\mbp$ as a perturbation of a fixed base interface $\Gamma_0$. The bilayer manifold contains an equilibrium of the flow whose associated interface is, up to exponentially small terms, circular.  Initial data from a thin tubular neighborhood of the manifold remain inside the tubular neighborhood for all time, and converge to the nearly circular equilibrium on an $O(\varep^{-4})$ time scale. During this time the interfaces $\Gamma_{\mbp(t)}$ are defined from the projection of $u(t)$ onto the manifold and they evolution has a normal velocity that, at leading order, reduces to  the regularized curve lengthening flow \eqref{V-normal}.

The refined bilayer \muckymuck\!\!, denoted $\Phi_\mbp$, is introduced in Lemma\,\ref{lem-def-Phi-p}. We identify two frequency separated slow spaces, the pearling and meandering spaces, as small eigenvalue spaces of the second variation of the FCH energy at $\Phi_\mbp$. The meandering slow space approximates the tangent space of a bilayer manifold at $\Phi_\mbp$ and is connected to the parameterization of the associated interface $\Gamma_\mbp$.
 In particular, $\mrp_0$ scales the length of $\Gamma_\mbp$,  $\mrp_1, \mrp_2$ are translation parameters, and $\hat\mbp=(\mrp_3,\cdots, \mrp_{N_1-1})$ characterizes the shape deviation from the base interface $\Gamma_0$.
The pearling slow space is associated to weakly damped modes that encode spatially periodic modulations of the width of the bilayer distribution near the interface $\Gamma.$ 

The construction of the bilayer manifold requires that the base interface $\Gamma_0$ and the scaled system mass $M_0$ be mass-compatible, in the sense that
\beq
  |\sigma_0(\Gamma_0, M_0) - \sigma_1^*| \leq  \delta,
\eeq 
where $\sigma_0$ is slaved to $\Gamma_0$ and $M_0$ through \eqref{|gamma|-sigma}, and $\delta>0$ is a small parameter independent of $\varep.$ 
We define $\cM_b$  as the graph of $\Phi_\mbp$ for $\mbp$ in the domain $\cO_\delta$, defined in \eqref{def-cO-p}, viz.
\beq
\cM_b(\Gamma_0,M_0) 
:= \left\{\Phi_\mbp(\sigma)\,\Bigl | \,\,  \mbp \in \cO_\delta,\, \left<\Phi_\mbp-b_-\right>_{L^2}=\frac{\varep M_0}{|\Omega|}\right\}.
\eeq 
Here $\sigma=\sigma(\mbp)$ the bulk density parameter is slaved to $\mbp$ to guarantee that $\Phi_\mbp$ has the system mass prescribed in \eqref{def-Mass}. 
In a companion paper, \cite{CP-linear}, we establish the stability of the bilayer manifold up to its boundary.
This requires a pearling stability condition
\begin{equation}
\label{PSC*}
(\mathbf{PSC}_*) \qquad \sigma_1^* S_1+(\eta_1-\eta_2) \la_0>0,
\end{equation}
on the system parameters $\eta_1>0$ and $\eta_2\in\mathbb R$, that renders the pearling modes linearly stable. Here $\la_0<0$ is the ground state eigenvalue of the  linearization of (\ref{phi0-def}) about $\phi_0$, and $S_1\in\mathbb R$ depends only upon the form of the double well, $W$.  A detailed investigation of the onset of the pearling instability was conducted in \cite{DHPW-14, NP-17}. 
Working in the $H^2$ inner norm,
\beq\label{def-Hin} \|u\|_{\Hin} := \sqrt{\|u\|_{L^2(\Omega)}^2+\varep^4 \|u\|^2_{H^2(\Omega)}},
\eeq 
the companion paper considers a bilayer manifold constructed from a smooth, embedded base interface $\Gamma_0$ and shows that solutions of the gradient flow arising from initial data within $O(\varep^{5/2})$ of $\cM_b$ remain within twice that distance of the manifold so long as their projected parameters $\mbp(t)$ reside in $\cO_\delta.$ That is, the only way to leave the tubular neighborhood of $\cM_b$ is through its boundary. This exit corresponds to generation of large curvatures in the associated interface $\Gamma_\mbp$ or in its proximity to self-intersection. A key obstacle is to control the coupling between the  evolution of the meander modes and the pearling modes, as the pearling modes are only weakly damped by the gradient flow.

In the current work we consider a bilayer manifold, $\cM_b$ build around a \emph{circular} base interface $\Gamma_0$. We assume the pearling stability condition holds and restrict the bilayer manifold to a smaller base domain $\cO_\delta^\circ\subset\cO_\delta$, see \eqref{def:cO-circ}, 
and establish that solutions $u(t)$ arising from initial data from an $O(\varep^{5/2})$ tubular neighborhood of the bilayer manifold, remain within twice this distance to the manifold for all time and ultimately converge to the unique, up to translation, equilibrium corresponding to the circular interface with prescribed system mass. Moreover we show that at leading order the normal velocity of the  interfaces $\Gamma_{\mbp(t)}$ arising as projections of $u(t)$ is given by the regularized-curve lengthening flow \eqref{V-normal}.  The most significant impact of the restriction to $\cO_\delta^\circ$ is the constraint on the bulk density of the initial data. This must satisfy 
\beq
|\sigma_0-\sigma_1^*|\leq C\varep^{1/2}\delta
\eeq 
for some positive constant $C$. This restricts the interface length, controlled by $\mrp_0$ to lie within $O(\varep^{1/2}\delta)$ of the equilibrium interface length, $\mrp_0^*$ which is set by the system mass. This condition also restricts the initial value of the reduced energy $\mcE(\Gamma,\sigma)$ to be $O(\delta^2)$ and leads the requirement that the system mass $M_0$ and base interface $\Gamma_0$ form an admissible pair, see \eqref{adm-pair}.

The analysis uses three small parameters: $\varep\in (0, \varep_0), \rho, \delta$. The parameter $\varep_0$ establishes the upper bound for the values of $\varep$ for which the principle results hold. The spectral parameter $\rho>0$ controls the size of the meander and pearling subspaces scaling their dimension $N\sim \varep^{-1}\rho$. The parameter $\delta>0$ predominantly controls the size of the domain $\cO_{\delta}^\circ$ of the meander parameters. We restrict $\delta>0$ sufficiently small to prevent self-intersection of the curves, to define the nonlinear projection, and to close certain nonlinear estimates, as in Lemma\,\ref{lem-est-p-2}. We choose $\rho$ and $\varep_0$ sufficiently small to separate the wave numbers of the pearling and meander modes that reduces the coupling between the meander and pearling modes in Lemma\,\ref{thm-coupling est}. Finally in Step 1 and 2 of Theorem\,\ref{thm-main-2} we select $\varep_0$ sufficiently small in terms of $\delta$ and $\rho$ to close the nonlinear estimates for all $\varep\in(0,\varep_0).$ In Proposition\,\ref{cor-Vp} we show that the formal normal velocity $V_\mbp$, \eqref{normal-velocity}, which slightly refines $V_{\mathrm{RCL}}$, captures the leading order normal motion as described by the evolution of the meander parameters. The meander evolution is driven at leading order by the projection of the formal normal velocity onto the Galerkin space of Laplace-Beltrami modes that characterize the interfacial shapes, \eqref{def-PiG}. However the Galerkin space only has dimension $O(\varep^{-1}\rho)$, and consequently the projection induces additional truncation errors. Indeed the largest error in the discrepancy between the normal velocity and $V_\mbp$ comes from the truncation of the surface diffusion term, $\Delta_s \kappa$, under the Galerkin projection.

The reduced energy \eqref{energy-W-M} and the regularized curve lengthening flow \eqref{V-normal} have a connection to the Euler-Bernoulli energy for an elastica and its gradient flow 
subject to a constant length constraint. The reduced energy can be viewed as a penalty method that enforces the length constraint through \eqref{|gamma|-sigma}, with the flow modeling the over-damped relaxation of a slightly compressible elastica whose initial configuration can admit some pre-compression or pre-extension of its relaxed length.  There is a substantial literature,  \cite{LS-85, Wen-95}, on the Euler-Bernoulli energy that includes analysis of equilibrium and asymptotic behaviour. In particular, circles, self-intersected figure eights, and their multiple folds are the only equilibrium solutions, and any initial smooth curve without self-intersection relaxes to a circle.

It is natural to compare the result for bilayer interface dynamics with those derived for the front solutions of the Cahn-Hilliard equation.  For the Cahn-Hilliard, much of the initial work, notably \cite{Pego-89} and \cite{ABC-94}, 
focused on formal and rigorous derivations of the Mullins-Sekerka flow in the $\varep\to 0$ limit.  Quasi-stationary dynamics based upon a radial scaling and translation parameters were 
derived in \cite{AF-03, AFK-04} in 2D and 3D.
 The FCH gradient flows differ from Cahn-Hilliard flows in that its sharp interface limit, $\varep\to0$, is ill posed.
 The $\Gamma$-limits constructed in \cite{RS-06} consider the case $\eta_2=\eta_1<0$, for which the Willmore and functionalization terms act in concert. This is fundamentally different than the competition between these terms expressed in the the strong FCH. This competition leads to a wide variety of minimizing sequences from $H^2(\Omega)$ whose energy are bounded as $\varep\to 0^+$ but are not readily associated to bilayers. These include the pearled interfaces constructed in \cite{PWu15}, as well the cylindrical filaments studied in \cite{DP15, CKP-18}. Pearling can lower the free energy of a bilayer distribution by modulating the width of the level sets of $u$ near an interface. The pearling modes manifest themselves in the analysis herein as they are at best weakly damped under the flow. As outlined in \cite{CP-linear}, this requires a separation of the frequencies of the the meandering and pearling modes so as to more tightly bound the projection of the interface motion residual onto the pearling modes. This key step in the analysis of the FCH bilayers has no analogy within the study of the fronts solutions in the Cahn-Hilliard model.  

 Nevertheless the Cahn-Hilliard equation and the FCH gradient flow both enjoy structure-structure interactions compatible with the Lifschitz-Slyosov-Wagner (LSW) theory of coarsening introduced in \cite{LS-61} and \cite{W-61} which reduces the interactions of finite numbers of near-circular interfaces to systems of ODEs for the radii and center positions. Formal reductions similar to LSW theory for the FCH, including interactions of constant curvature codimension one and two structures were derived in \cite{DP15}. 

The remainder  of this article is organized as follows. In Section \ref{sec-pre}, we present the local coordinates and finite dimensional  variation of the interface. In Section \ref{sec-profile} we construct
the bilayer \muckymucks $\Phi_\mbp$ whose image under the map
$\mbp\mapsto\Phi_\mbp$ defines the bilayer manifold $\cM_b$. We  establish the coercivity of the linearization  $\Pi_0\mbL_{\mathbf p}$ of the gradient flow about $\Phi_\mbp$ when the operator is restricted to act on the orthogonal compliment to the slow space. In section \ref{sec-p} we derive the dynamics of meandering parameters $\mbp$ by projecting the residual $\mrF(\Phi_\mbp)$ onto the tangent plane of the manifold. In section \ref{sec-main-thm} close the nonlinear estimates which reduce the nonlinear flow to an ODE on the tangent plane of $\cM_b$, this concludes with the statement of asymptotic stability of the circular equilibrium and the rigorous reduction of the finite dimensional meander evolution to the curvature driven evolution of the regularized curve-lengthening flow.

\subsection{Notation}\label{ssec-Notation} We present some general notation. 
\begin{enumerate}
\item  The symbol $C$ generically denotes a positive constant whose value depends only on the system parameters $\eta_1, \eta_2$, mass parameter $M_0$, the domain $\Omega$, and 
radius of initial circle $R_0$. In particular its value is independent of $\varep$ and $\rho$, so long as they are sufficiently small. 
The value of $C$  may vary line to line without remark.  In addition, $A\lesssim B$ indicates that quantity $A$ is less than quantity $B$ up 
to a multiplicative constant $C$ as above, and $A\sim B$ if $A\lesssim B$ and $B\lesssim A$.  
The expression $f=O(a)$  indicates the existence of a constant $C$, as above,
and a norm $|\cdot|$ for which
\begin{equation*}
| f| \leq C |a|.
\end{equation*}
\item The quantity $\nu$ is a positive number, independent of $\varep$, that denote an exponential decay rate. It may vary from line to line.
\item If a function space $X(\Omega)$ is comprised of functions defined on the whole spatial domain $\Omega$, we will drop the symbol $\Omega$. 
\item We use $\bm 1_{E}$ as the characteristic function of an index set $E\subset \mathbb N$,  i.e. $\bm 1_{E}(x)=1$ if $x\in E$;  $\bm 1_{E}(x)=0$ if $x\notin E$.
We denote the usual Kronecker delta by 
\beqs
\delta_{ij}=\left\{\begin{array}{ll} 1,& i=j; \\
                                                       0, & i\neq j.
                                                       \end{array} \right.
                                                       \eeqs
\item For a finite vector $\mathbf q=(\mathrm q_j)_j$,  we denote the norms
\begin{equation*}
\|\mathbf q\|_{l^k}=\left(\sum_j  |\mathrm q_j|^k\right)^{1/k}, \qquad 
 \hbox{for $k\in \mathbb N^+$}, 
 \end{equation*}
 and  $\|\mathbf q\|_{l^\infty}=\max_j |\mathrm q_j|.$ For a matrix $\mbQ=(\mbQ_{ij})_{ij}$ as a map from $l^2$ to $l^2$ has operator norm $l^2_*$ defined by
 \begin{equation*}
 \|\mbQ\|_{l^2_*}=\sup_{\{\|\mbq\|_{l^2=1}\}} \|\mathbb Q \mbq\|_{l^2}.
 \end{equation*}
We write
\beqs \mrq_j = O(a)e_j,\quad \mbQ_{ij}=O(a) \mbE_{ij},\eeqs
where ${\bm e}=(e_j)_j$ is a vector with $\|{\bm e}\|_{l^2}=1$ or ${\mbE}$ is a matrix with operator norm $\|\mbE\|_{l^2_*}=1$ 
to imply that $\|\mbq\|_{l^2} = O(a)$ or $\|\mbQ\|_{l^2_*}=O(a)$ respectively. See \eqref{def-e-i}-\eqref{est-e-i,j} of Notation\,\ref{Notation-e_i,j} for usage.
\item  The matrix $e^{\theta\mathcal R}$ denotes rotation through the  angle $\theta$ with the generator $\mathcal R$.  More explicitly, 
\begin{equation*}\mathcal R=\left(\begin{array}{cc}0& -1\\1&0\end{array}\right), \quad e^{\theta\mathcal R}= \left(\begin{array}{cc}\cos \theta& -\sin \theta\\\sin \theta &\cos \theta \end{array}\right).
\end{equation*}
\end{enumerate}

\section{Perturbed interfaces and local coordinates}\label{sec-pre}
We consider smooth closed interfaces $\Gamma$ immersed in $\Omega$ and parameterized by $\bm\gamma:\mathscr I\mapsto\Omega$, where $\mathscr I$ is the periodic interval of length $2\pi R_0.$ The following class of curves parameterized are smooth and do not self-intersect.

\begin{defn}
\label{def-GKl}
Given $K, \ell>0$ and an integer $k>0$ the class $\cG_{K,\ell}^k$ consists of closed curves $\Gamma$ embedded in $\Omega$ whose parameterization $\bm\gamma$ has the properties (a) $\|\bm\gamma\|_{W^{k,\infty}(\mathscr I)}\leq K$ and (b) for any points two points on $\mathscr I$ that satisfy $|s_1-s_2|_{\mathscr I}>1/(2K)$ then $|\bm \gamma(s_1)-\bm\gamma(s_2)|>\ell$. Here $|\cdot|_{\mathscr I}$ denotes the periodic distance 
$|s|_\msI=\min\bigl\{\big|s-2\pi R_0 k\big|\, : k\in \mathbb Z\bigr\}.$
\end{defn}

For each $\Gamma\in \cG_{K,\ell}^2$ with $\ell\leq \sqrt{2\pi}/K$ there exists a tubular neighborhood, $\Gamma^\ell$ of $\Gamma$ with thickness $\ell$, such that  the change of coordinates $x\mapsto (s,r)$ through
\beq\label{coord-wh}
x=\bm \gamma(s)+r\mbn(s), 
\eeq 
is well posed. Here $\mbn=e^{-\pi\mathcal R/2}\bm \gamma'/|\bm \gamma'|$ the outer normal of $\Gamma$ and $r=r(x)$ the signed distance of $x$ to the curve $\Gamma$. Introducing the scaled distance $z=r/\varep\in[-\ell/\varep, \ell/\varep]$, we refer to $(z,s)$ as the local coordinates near $\Gamma$. 

In the sequel we fix a base interface $\Gamma_0$  which is a circle with radius $R_0>0$ and constant curvature
\begin{equation}
\kappa_0=-1/R_0,
\end{equation}
and the local coordinates are defined on all of $\Omega$ minus the center of the circle. Let $\bm\gamma_0=\bm\gamma_0(s)$ be the parameterizetion by arc-length of $\Gamma_0$, that is, $|\bm \gamma_0'(s)|=1$  for all $s$ lies on the periodic domain $\msI$.  
The Laplace-Beltrami operator is denoted by $-\Delta_s: H^2(\msI)\rightarrow L^2(\msI)$ and its scaled eignevalues $\{\beta_k^2\}_{k=0}^\infty$ and normalized eigenfunctions $\{\Theta_k\}_{k=0}^\infty$ satisfy,
\begin{equation}\label{def-beta}
-\Delta_s  \Theta_k =\beta_k^2 \Theta_k\big/R_0^2.
\end{equation}
The ground state eigenmode is spatially constant: 
\beq
\label{def-Theta0}
\Theta_0=1/\sqrt{2\pi R_0}, \qquad \beta_0=0;
\eeq 
and for $k= 1,2\cdots$, 
\begin{equation}\label{def-Theta-beta}
\Theta_{2k-1}=\frac{1}{\sqrt{2\pi R_0}}\cos\left(\frac{ks}{R_0}\right),\quad \Theta_{2k}=\frac{1}{\sqrt{2\pi R_0}}\sin \left(\frac{ks}{R_0}\right); \quad \hbox{and}\;\;  \beta_{2k-1}=\beta_{2k}=k.
\end{equation}

We construct interfaces that are perturbations of $\Gamma_0$ parameterized by an $N_1$-vector $\mbp$. The construction cuts off high-frequency modes.  We split the parameter vector
\beq\label{def-hatp}
\mbp = (\mrp_0, \mrp_1, \mrp_2, \hat\mbp), \qquad \hat\mbp =  (\mrp_3, \mrp_4, \cdots, \mrp_{N_1-1}), 
\eeq
in which $\mrp_0$ scales the length of the perturbed interface, $\mrp_1, \mrp_2$ translate it, and $\hat \mbp$ deviates the perturbed interface from circularity. The following weighted spaces control $\hat\mbp$.

\begin{defn}[Weighted perturbation space]
 Let $\mathbb D$ be the $(N_1-3) \times (N_1-3)$ diagonal matrix 
\beq\label{def-mbD}
\mathbb D=\diag\{\beta_3^2, \beta_4^2, \cdots \beta_{N_1-1}^2 \}.
\eeq
We say $\hat \mbp$ lies in $\mathbb V_r^k$ if  $\|\hat\mbp\|_{\mbV_r^k}:= \| \mathbb D^{r/2} \hat\mbp\|_{l^k}<\infty$, or more precisely,
\beq
\label{def-mbV}
\|\hat\mbp\|_{\mbV_r^k}= \left(\sum_{j=0}^{{N_1}-1} \beta_j^{kr} |\mathrm p_j|^k\right)^{1/k}<\infty.
\eeq
When $k=1$, we omit the superscript $k$ and denote the space by $\mbV_r$.
\end{defn}
The $\mbp$-variation of $\Gamma_0$ is given below.
\begin{defn}[Perturbed interfaces]\label{def-P-interface} Fix $C\lesssim 1$ and let $\delta>0$ be a small parameter, independent of $\varep.$ For each $\mbp$ from the set
\beq\label{A-00}
\cD_\delta :=\left\{\mbp\in \mbR^{N_1} \; \Big|\; \mrp_0 > -1/2,  \quad |\mrp_1|+|\mrp_2|+ \|\hat\mbp\|_{\mbV_2} \leq C, \quad \|\hat \mbp\|_{\mbV_1}\leq C\delta  \right\},
\eeq  
we define the $\mathbf p$-variation of $\Gamma_0$, denoted by $\Gamma_{\mathbf p}$, via the parameterization 
\begin{equation}\label{def-gamma-p}
\bm\gamma_{\mathbf p}(s):=(1+\mrp_0)\bm\gamma_{\bar p}(s)/A(\mbp) +\mrp_1\Theta_0 \bE_1 +\mrp_2\Theta_0 \bE_2,  \quad \hbox{for} \quad s\in \msI.
\end{equation}
Here $A(\mbp)>0$ normalizes the length of the perturbed parameterization $\bm \gamma_{\bar p}$, 
\beq
A(\mbp)=|\Gamma_0|^{-1} \int_{\msI} |\bm\gamma_{\bar p}'(s)|\dd s, \quad |\Gamma_0| = 2\pi R_0,
\eeq
given by
\beq
\bm\gamma_{\bar p}(s)= \bm \gamma_0(s) + \bar p(\tilde s) \mbn_0(s)
\eeq
where $\mathbf n_0(s)$ is the outer normal vector of the circle $\Gamma_0$ parameterized by its arc-length parameter $s$  while
\begin{equation}\label{def-barp}
 \bar p(\tilde s):= \sum_{i=3}^{{N_1}-1} \mathrm p_i \tilde \Theta_i(\tilde s), \qquad \tilde \Theta_i(\tilde s):=\Theta_i\left(\frac{2\pi R_0 \tilde s}{|\Gamma_\mbp|}\right),\end{equation}   
are parameterized by arc length of $\bm \gamma_\mbp$, denoted $\tilde s=\tilde s(s)\in \msI_\mbp=[0,|\Gamma_\mbp|]$, obtained as the solution of
\beq \label{def-ts}
\frac{\mrd \tilde s}{\mrd s}=|\bm \gamma_\mbp'|,\qquad \tilde s(0)=0.
\eeq
\end{defn}

\begin{remark}
\begin{enumerate}
\item The definition of $\Gamma_\mbp$ is implicit in $\mbp$ through \eqref{def-gamma-p} and \eqref{def-barp}-\eqref{def-ts}. This implicit definition is well posed for $\mbp$ satisfying \eqref{A-00}, see Lemma 2.10 of \cite{CP-linear}. The small assumption on $\|\hat \mbp\|_{\mbV_1}$ prevents self intersection of the perturbed curve, see Lemma 2.11 of \cite{CP-linear}. 
The implicit definition insures the orthogonality \eqref{ortho-tTheta} of $\tilde \Theta_i$ in $L^2(\msI_\mbp)$,
\beq\label{ortho-tTheta}
\int_{\msI_{\mbp}}\tilde \Theta_j\tilde\Theta_k\dd \tilde s=(1+\mrp_0)\delta_{jk}
, \quad j,k=0, 1, \ldots, N_1-1.
\eeq
\item  The $\mathrm p_0$-term scales the length of the curve from $2\pi R_0$ to $2\pi(1+\mrp_0)R_0$. The parameters $\mathrm p_1$ and $\mathrm p_2$ serve to  translate the interface. The perturbation they induce is not normal to the original interface, however its projection onto $\mathbf n_0$ satisfies
\begin{equation}\label{proj-n0-E-basis}
\Theta_0 \bE_1\cdot \mathbf n_0=\frac{1}{\sqrt{2\pi R_0}}\cos{\frac{s}{R_0}}=\Theta_1, \quad \Theta_0\bE_2\cdot \mathbf n_0=\frac{1}{\sqrt{2\pi R_0}}\sin\frac{s}{R_0}=\Theta_2.
\end{equation}
\end{enumerate}
\end{remark}

From the definition, we have $|\Gamma_\mbp|=2\pi R_0 (1+\mrp_0)$.  It follows from \eqref{def-beta} and \eqref{def-barp} that the scaled $\tilde\Theta_j=\tilde\Theta_j(\tilde s)$ satisfy
\beq \label{tTheta''}
-\tilde \Theta_j''(\tilde s)=\beta_{\mbp,j}^2 \tilde \Theta_j(\tilde s), \qquad \beta_{\mbp,j}= \frac{\beta_j}{R_0(1+\mrp_0)}. 
\eeq 
Here and below, primes of $\tilde \Theta_j$ denote derivatives with respect to $\tilde s$. 
We remark that  $\beta_{\mbp,j}$ equals $\beta_j/R_0$ when $\mbp=\bm 0$, rather than $\beta_j$. The orthogonality \eqref{ortho-tTheta} implies  
\beq\label{re-barp-hatp}
\|\hat{\mathbf p}\|_{\mathbb V_k^2} \sim \|\bar p\|_{H^k(\mathscr I_\mbp)}, \quad \|\bar p^{(k)}\|_{L^\infty(\mathscr I_\mbp)}\lesssim \|\hat{\mathbf p}\|_{\mathbb V_k},
\eeq
for all $\mbp\in\cD_\delta$ where the implied constants depend only upon $R_0$.

The following embeddings are direct results of H\"older's inequality and the asymptotic form of  $\beta_j$ introduced in \eqref{def-Theta-beta}, details are omitted. 
\begin{lemma}\label{lem-est-V}
Suppose that $\hat{\mathbf p}\in l^\infty(\mathbb R^{N_1})$, 
then  $\|\hat\mbp\|_{\mbV^2_0}=\|\hat\mbp\|_{l^2}$ and
\begin{equation*}
\begin{aligned}
\|\hat{\mathbf p}\|_{\mbV_k}\lesssim  \|\hat{\mathbf p}\|_{\mbV_{k+1}^2},\quad \|\hat{\mathbf p}\|_{\mbV_{k+1}^r}\lesssim N_1\|\hat{\mathbf p}\|_{\mbV_k^r},  \quad \|\hat\mbp\|_{\mbV_k}\lesssim  N_1^{k+1/2}\|\hat\mbp\|_{l^2}.
\end{aligned}
\end{equation*}
In addition, for any vector $\mathbf a \in l^2(\mathbb R^{m})$ we have the dimension dependent bound
\begin{equation}\label{est-l1-l2}
\|\mathbf a\|_{l^1}\leq {m}^{1/2}\|\mathbf a\|_{l^2}.
\end{equation}
\end{lemma} 
When developing expansions of the interface $\Gamma_\mbp$ it is convenient to build in the uniform rescaling and translations associated to $(\mrp_0,\mrp_1, \mrp_2)$ so that the expansions are controlled by $\hat{\mbp}.$ To this end we define the leading order perturbed interfacial map
\beq\label{def-gamma-p0}
\bm \gamma_{\mbp,0}=(1+\mrp_0)\bm \gamma_0(s) +\mrp_1\Theta_0 \bE_1 +\mrp_2 \Theta_0\bE_2,
\eeq 
whose interface $\Gamma_{\mbp,0}$ is a translated, scaled circle with constant  curvature $\kappa_{\mbp,0}$. 
 We assume $\mbp\in \cD_\delta$, \eqref{A-00}, holds  in the sequel. 
\begin{lemma}[Geometric quantities of $\Gamma_\mbp$]\label{lem-Gamma-p} 
The length normalization $A(\mbp)$ depends quadradically upon $\hat\mbp$, and the length of $\Gamma_\mbp$ depends  only on $\mrp_0$,
\beq\label{est-A-|Gamma-p|}
A(\mbp)=1+O(\|\hat\mbp\|_{\mbV_1}^2);\qquad |\Gamma_\mbp|=(1+\mrp_0) |\Gamma_0|.
\eeq
The gradient of $A$ with respect to $\mbp$ satisfies
\beqs
\|\nabla_\mbp A\|_{l^2}\lesssim \|\hat\mbp\|_{\mbV_2^2}.
\eeqs
The curvature of  $\Gamma_{\mathbf p}$, defined by 
\begin{equation}\label{def-kappa}
\kappa_{\mathbf p}:=\bm \gamma_{\mathbf p}''\cdot \mathbf n_{\mathbf p}/|\bm\gamma'_{\mathbf p}|^2, \quad\hbox{with}\quad \mathbf n_{\mbp}=e^{-\pi \mathcal R/2}\bm \gamma'_{\mathbf p}\big/|\bm \gamma'_{\mathbf p}|
\end{equation}
admits the expansion
\begin{equation}\label{def-kappa-p,0}
\kappa_{\mathbf p}(s)
=\kappa_{\mathbf p,0}+\mathcal Q_1 +\mathcal Q_2, \qquad \kappa_{\mathbf p,0}=-\frac{1}{R_0(1+\mathrm p_0)}.
\end{equation}
where the linear approximation is given by
\beqs
\mathcal Q_1=\frac{1}{(1+\mrp_0)R_0^2}\sum_{j= 3}^{{N_1}-1}(1-\beta_j^2)\mathrm p_j\tilde \Theta_j,
\eeqs
and the quadratic remainder $\mathcal Q_2$ satisfies
\beqs
\begin{aligned}
&\|\mathcal Q_2\|_{L^2(\msI_\mbp)}\lesssim  \|\hat\mbp\|_{\mbV_2^2}^2, \qquad  \|\mathcal Q_2\|_{H^1(\msI_\mbp)} \lesssim \|\hat\mbp\|_{\mbV_3^2}\|\hat\mbp\|_{\mbV_2^2}, \\
& \| \mathcal Q_2\|_{H^2(\msI_\mbp)} \lesssim \|\hat\mbp\|_{\mbV_4^2}\|\hat\mbp\|_{\mbV_2^2}; \qquad
 \| \mathcal Q_2\|_{H^3(\msI_\mbp)} \lesssim \|\hat\mbp\|_{\mbV_5^2}\|\hat\mbp\|_{\mbV_2^2}+\|\hat\mbp\|_{\mbV_3^2}\|\hat\mbp\|_{\mbV_4^2}.
\end{aligned}
\eeqs 
The curvature  $\kappa_\mbp$ and normal $ \mbn_\mbp$ depend only on $\mathrm p_0$ and $\hat{\mathbf p}$,  and satisfy  the following bounds
\begin{equation}\label{bound-kappa-p}
\|\kappa_{\mathbf p}\|_{L^\infty}   +  \|\varep^2\Delta_{s_{\mathbf p}}\kappa_{\mathbf p}\|_{L^\infty} \lesssim 1+\|\hat{\mathbf p}\|_{\mbV_{2}};
\quad |\mathbf n_{\mathbf p}-\mathbf n_0|\lesssim \|\hat{\mathbf p}\|_{\mbV_1}, 
\end{equation} 
Moreover, the perturbed and original normal satisfy the relation
\beq\label{est-bn-p0}
\mbn_\mbp\cdot \mbn_0= 1+O(\|\hat\mbp\|_{\mbV_1}^2).
\eeq
\end{lemma}
\begin{proof}
The length of $\Gamma_\mbp$ follows from its definition, and the approximation of $A(\mbp)$ and its gradient estimate are given in Lemma 2.11 of \cite{CP-linear}. 
Taking the derivative of $\bm \gamma_\mbp$ in \eqref{def-gamma-p} and using $\mbn_0'=-\kappa_0\bm \gamma_0'$ we find
\beq \label{est-gamma'}
\bm \gamma_\mbp'=\frac{1+\mrp_0}{A(\mbp)} \Big[\big(1-\kappa_0\bar p(\tilde s)\big)\bm\gamma_0'+\bar p'(\tilde s) |\bm\gamma_\mbp'| \mbn_0(s) \Big],
\eeq
and hence for $\hat \mbp \in \mbV_2$ we have the approximations
\beq\label{est-|gamma'|}
 |\bm \gamma'_\mbp|= \frac{1+\mrp_0}{A(\mbp)} \left(1-\kappa_0 \bar p\right) +O(\|\hat\mbp\|_{\mbV_1}^2), \qquad |\bm \gamma'_\mbp|'= O(\|\hat\mbp\|_{\mbV_1}).
 \eeq
 To obtain an approximation of the curvature $\kappa_\mbp$ we take an additional $s$ derivative of \eqref{est-gamma'}. Using  the relation between the tangent and normal  we find the equality
 \beqs
 \begin{aligned}
 \bm \gamma_{\mathbf p}'' 
 &=\frac{1+\mrp_0}{A(\mbp)}\left[\Big(\kappa_0 +\frac{(1+\mrp_0)^2}{A^2(\mbp)}\bar p'' (\tilde s)-\kappa_0^2 \bar p +\mathcal Q_{2,0}\Big)\mbn_0- 2\kappa_0 \bar p'(\tilde s) |\bm \gamma_\mbp'| \bm \gamma_0' \right].
 \end{aligned}
 \eeqs
 Here $\mathcal Q_{2,0}(\bm \gamma_\mbp'')$  takes the explicit form
 \beqs
 \mathcal Q_{2,0}= \left(|\bm \gamma_\mbp'|^2-\frac{(1+\mrp_0)^2}{A(\mbp)} \right) \bar p'' + \bar p'|\bm \gamma_\mbp'|'
 \eeqs
 and is an intermediate quadratic remainder from \eqref{est-|gamma'|}.
  From the tangent and normal relation: $\mbn=e^{-\pi \mathcal R/2}\bm \gamma'$, we deduce from \eqref{est-gamma'} that
 \beq\label{gamma'-bot}
 e^{-\pi\mathcal R/2} \bm\gamma_\mbp'= \frac{1+\mrp_0}{A(\mbp)} \Big[(1-\kappa_0\bar p(\tilde s))\mbn_0-\bar p'(\tilde s) |\bm \gamma_\mbp'|\bm \gamma_0'\Big], 
 \eeq
 which when dotted with the approximation for $\bm \gamma_{\mbp}^{\prime\prime}$ implies
 \beqs
 e^{-\pi\mathcal R/2} \bm\gamma_\mbp'\cdot  \bm \gamma_\mbp'' =\left( \frac{1+\mrp_0}{A(\mbp)}\right)^2  \left[ \kappa_0 +\frac{(1+\mrp_0)^2}{A^2(\mbp)} \bar p'' -2\kappa_0^2 \bar p +\mathcal Q_{2,1}\right] , 
 \eeqs
 in which $\mathcal Q_{2,1}$ is a quadratic term given by
 \beqs
 \mathcal Q_{2,1}=  \mathcal Q_{2,0}(1-\kappa_0\bar p) + \left[ -\kappa_0\bar p\left(\frac{(1+\mrp_0)^2}{A^2(\mbp)} \bar p'' -\kappa_0 \bar p\right) +\left(2\kappa_0 \bar p' |\bm \gamma_\mbp'| +\kappa_0' \bar p \right)\bar p'|\bm \gamma_\mbp'|\right].
 \eeqs
 Finally, in light of \eqref{est-|gamma'|} we rewrite
 \beqs
 \frac{1}{|\bm\gamma_\mbp'|^3}=\left(\frac{A(\mbp)}{1+\mrp_0}\right)^3 \left(1+3\kappa_0 \bar p +\mathcal Q_{2,2} \right),\qquad \mathcal Q_{2,2
}:= \frac{1}{|\bm \gamma_\mbp'|^3}-1-3\kappa_0\bar p, 
\eeqs
and substituting these expressions in \eqref{def-kappa} we obtain the curvature expansion,
 \beqs
 \kappa_\mbp= \frac{A(\mbp)}{1+\mrp_0} \left[\kappa_0 +\kappa_0^2 \bar p +\frac{(1+\mrp_0)^2}{A^2(\mbp)} \bar p''(\tilde s) +\mathcal Q_{2,3} \right]
 \eeqs 
 where $A(\mbp)=1+O(\|\hat\mbp\|_{\mbV_1}^2)$ and the final quadratic remainder takes the form 
 \beqs
 \mathcal Q_{2,3}:=\frac{A(\mbp)}{(1+\mrp_0)^2}e^{-\pi \mathcal R/2}\bm \gamma_\mbp'\cdot \bm \gamma_\mbp''\;\mathcal Q_{2,2}   + \frac{(1+\mrp_0)^3}{A^3(\mbp)}\frac{1}{|\bm \gamma_\mbp'|^3} \; \mathcal Q_{2,1} +  3\kappa_0 \bar p\left[\frac{(1+\mrp_0)^2}{A^2(\mbp)}\bar p'' -2\kappa_0\bar p \right]. 
 \eeqs
 The form of the expansion \eqref{def-kappa-p,0} follows from the definition of $\bar p$, \eqref{est-A-|Gamma-p|}, and $(A(\mbp)-1)$ is quadratic. The $H^k$ estimates of $\mathcal Q_2$ follows directly from the formula for the quadratic terms, the independence of $A(\mbp)$ from $\tilde s_\mbp$,  and the embedding estimates in Lemma \ref{lem-est-V}.  
 
 The curvature bounds in \eqref{bound-kappa-p} follow directly from these expansions and the embedding estimate of Lemma \ref{lem-est-V}. To establish that the normals are nearly parallel, from the definition \eqref{def-kappa} of $\mbn_\mbp$ we have
 \beqs
 \mbn_\mbp\cdot \mbn_0= \frac{e^{-\pi\mathcal R/2}\bm \gamma_\mbp'}{|\bm \gamma_\mbp'|} \cdot \mbn_0. 
 \eeqs
 The estimate \eqref{est-bn-p0} follows directly by  \eqref{gamma'-bot} and \eqref{est-|gamma'|}. This completes the proof.  
\end{proof}

In Lemma 2.11 of \cite{CP-linear} it is established that there exists $K,\ell,$ and $\delta>0$ depending only upon the choice of $R_0$ such that for each $\mbp\in\cD_\delta$ the curve  $\Gamma_\mbp\in \cG_{K,2\ell}^2$. In particular $\Gamma_\mbp$  has a tubular neighborhood of width $2\ell$, denoted $\Gamma_\mbp^{2\ell}$ and called the reach of $\Gamma_\mbp$, in which the local coordinates $(s_{\mathbf p}, z_{\mathbf p})$ are well defined. On the interface $\Gamma_\mbp$, where $z_\mbp=0$, we have the relation $s_\mbp=s$.   The coordinate systems are summarized in the following table. 
\begin{table}[h!]
\centering
\begin{tabular}{c|c|c||c|c|c}
\hline
    Curve & Parameterization& Arc-length & Neighborhood & Local coordinates & Arc-length extension \\[2pt]
    \hline 
    $ \Gamma_0$ & $\bm \gamma_0(s), s\in \msI$ & $s\in \msI$ & $\Gamma_0^{2\ell}$ & $(s,z), s\in \msI$ & $s\in \msI$ \\[2pt]
    \hline 
    $\Gamma_\mbp$ & $\bm \gamma_\mbp(s), s\in \msI$ & $\tilde s \in \msI_\mbp$ & $\Gamma_\mbp^{2\ell}$ & $(s_\mbp, z_\mbp), s_\mbp \in \msI$ & $\tilde s_\mbp\in\msI_\mbp$\\[3pt]
    \hline
\end{tabular}
\vskip 0.1in
\caption{Table of Local Coordinates: $\msI=[0, |\Gamma_0|]$ and $ \msI_\mbp=[0, |\Gamma_\mbp|]$. The geometric quantities $\mbn_\mbp, \bm\gamma_\mbp$ and $\kappa_\mbp$ on $\Gamma_\mbp$ have natural extensions to $\Gamma_\mbp^{2\ell}$. }
\label{table:1}
\end{table}

We say a function $f=f(s_\mbp)$ lies in $L^2(\msI_\mbp)$ if \beqs 
\|f\|_{L^2(\msI_\mbp)}^2:=\int_{\msI} f^2(s_\mbp) |\bm \gamma_\mbp'|\dd s_\mbp <\infty   \quad \hbox{or equivalently} \quad \int_{\msI_\mbp} f^2(s_\mbp) \dd \tilde s_\mbp <\infty.  
\eeqs 

\begin{notation}\label{Notation-h}
To simplify the presentation of the subsequent calculations, we will use the blanket notation $h(z_\mbp, \bm \gamma_{\mathbf p}^{(k)})$ for any smooth function defined in $\Gamma_\mbp^{2\ell}$ that decays exponentially close  to a constant as $|z_\mbp|\to 2\ell/\varep$  and depends upon $s_{\mbp}$ only through $|\bm \gamma_\mbp'|,  \kappa_\mbp,  \mbn_\mbp\cdot \mbn_0$  and their scaled derivatives $\varep^k\nabla_{s_\mbp}^k$ of order up to $k.$ In particular, if $h$ is independent of $z_\mbp$, then we denote it by $h(\bm \gamma_\mbp'')$.   

Similarly, $h(\bm \gamma_{\mathbf p,0}'')$ denotes  a smooth function whose $s_{\mbp}$ dependence arises from
the leading order perturbed curve $\bm\gamma_{\mbp,0}$ through its geometric quantities $|\bm \gamma_{\mbp,0}'|,  \kappa_{\mbp,0},  \mbn_{\mbp,0}\cdot \mbn_0$ and $\varep^k\p_{s_\mbp}^k$ their scaled derivatives up to order $2$.
\end{notation}
 
 We will frequently decompose these quantities in the form
 \beq\label{decomp-h}
 h(z_\mbp, \bm \gamma_\mbp'')=h(z_\mbp, \bm \gamma_{\mbp,0}'') +\left(h(z_\mbp, \bm \gamma_{\mbp}'')-h(z_\mbp, \bm \gamma_{\mbp,0}'')\right),
 \eeq
in order to obtain Lipschitz type estimates on the difference in terms of  $\hat\mbp$. This is formalized in the following Lemma. 
\begin{lemma}\label{lem-h}
Suppose $h=h(\bm \gamma_\mbp'')$ enjoys the properties of Notation\,\ref{Notation-h}, and is decomposed as in \eqref{decomp-h}. Then the leading order term $h(\bm \gamma_{\mbp,0}'')$ is a constant independent of $\tilde s_\mbp$ and if  $\hat\mbp \in \mbV_2$, then
\begin{equation}\label{est-h-p-V-2+2}
\| h(\bm \gamma_{\mbp}'')- h(\bm \gamma_{\mbp,0}'')\|_{L^2(\msI_{\mbp})}\lesssim \|\hat{\mbp}\|_{\mbV_2^2};
\end{equation}
and if $\hat\mbp \in \mbV_3^2$, then for $l\geq 1$, 
\begin{equation}\label{est-h-gamma''-de}
\left\|\varep^{l-1}\p_{s_{\mbp}}^{l}\left(h( \bm \gamma_{\mbp}'')-h(\bm \gamma_{\mbp,0}''\right)\right\|_{L^2(\msI_{\mbp})}\lesssim \|\hat{\mbp}\|_{\mbV_{3}^2}.
\end{equation}
\end{lemma}
\begin{proof} Estimates in  \eqref{est-h-p-V-2+2}-\eqref{est-h-gamma''-de}, can be derived directly by the approximations of $|\bm \gamma_\mbp'|, \kappa_\mbp$ and $\mbn_\mbp$ in \eqref{est-|gamma'|}, \eqref{def-kappa-p,0}. We only need to verify  $h(z_\mbp, \bm \gamma_{\mbp,0}'')$ is independent of $s_\mbp$. This is true since $|\bm \gamma_{\mbp,0}'|=1+\mrp_0$ by \eqref{est-|gamma'|},  $\kappa_{\mbp,0}$ admits form in \eqref{def-kappa-p,0} and  \beq \mbn_0\cdot \mbn_{\mbp,0}=1, \quad\hbox{for}\quad  \mbn_{\mbp,0}= \frac{e^{-\pi \mathcal R/2}\bm \gamma_{\mbp,0}'}{|\bm \gamma_{\mbp,0}'|}.\eeq 
Here we used \eqref{gamma'-bot} with $\hat\mbp=0.$
\end{proof}

\begin{lemma}\label{Notation-e_i,j}
Recalling the notation of section\,\ref{ssec-Notation}, if $f\in L^2(\msI_\mbp)$, then there exists a unit vector $\bm e=(e_i)$ such that
\beq\label{def-e-i}
\int_{\msI_\mbp}f \tilde \Theta_i \dd \tilde s_\mbp=O(\|f\|_{L^2(\msI_\mbp)})e_i.
\eeq
If  in addition $f\in L^\infty$ on $\msI_\mbp$, then for any vector $\mathbf a=(\mathrm a_j)\in l^2$, we have
\beq\label{est-e-i,j}
\left|\sum_{j}\int_{\msI_\mbp} f\tilde \Theta_i \mathrm a_j\tilde \Theta_j\dd \tilde s_\mbp\right|\lesssim \|\mathbf  a\|_{l^2}\|f\|_{L^\infty}e_i,
\eeq
and there exists a matrix $\mbE=(\mbE_{ij})$ with $l^2_*$ norm one,  such that
\beq\label{def-e-i,j}
\int_{\msI_\mbp}f\tilde \Theta_i \tilde \Theta_j \dd \tilde s_\mbp =O(\|f\|_{L^\infty}) \mbE_{ij}.
\eeq
\end{lemma}
\begin{proof} The estimates follow from Plancherel and classic applications of Fourier theory.
\end{proof}

The following Lemma estimates the $\mbp$-variation of the  local coordinates  associated to  $\Gamma_{\mathbf p}$ which is quoted from Lemma 6.2 of \cite{CP-linear}. In particular it estimates the difference between  
$(s_{\mathbf p}, z_{\mathbf p})$ and $(s,z)$  in terms of $\mathbf p$.  
\begin{lemma}
\label{lem-change-of-coord}
Let $(s_{\mathbf p}, z_{\mathbf p})$ be the local coordinates subject to $\Gamma_{\mathbf p}$ on $\Gamma_\mbp^{2\ell}$. Assuming \eqref{A-00} the tangent coordinate $s_\mbp$ satisfies
\beqs
\|\nabla_\mbp s_{\mbp}\|_{L^2(\msI_\mbp)}\lesssim 1;
\eeqs 
while $z_\mbp$ have the $\mbp$-gradient
\beqs
\frac{\p z_\mbp}{\p \mrp_j}= \varep^{-1} \xi_j(s_\mbp),
\eeqs 
where $\xi_j$ is a function of $s_\mbp$ given explicitly by 
\beqs
\xi_j(s_\mbp)=\left\{\begin{aligned}
&-\left(\frac{R_0+\bar p }{A}\left(1-(1+\mrp_0)\p_{\mrp_0}\ln A  \right)- \frac{\tilde s_\mbp \bar p'}{A} \right)\mathbf n_0\cdot \mathbf n_{\mathbf p}, \qquad & j=0\\[3pt]
&\h{50pt}-\Theta_0\bE_j\cdot \mathbf n_{\mathbf p}, & j=1,2;\\[3pt]
& - \left(\tilde \Theta_j-\frac{(1+\mrp_0)\p_{\mrp_j}\ln A}{A} (R_0+\bar p) \right)\mathbf n_0\cdot \mathbf n_{\mathbf p}& j\geq 3.
\end{aligned}
\right.
\eeqs 
Moreover, 
we have the estimate
\begin{equation}\label{est-diff-s-z-p}
|s_{\mathbf p}- s|\lesssim \|\mathbf p\|_{l^1},\quad |z_{\mathbf p}-z|\leq \varep^{-1}\|\mathbf p\|_{l^1}.
\end{equation}
\end{lemma}

We conclude this section by introducing the $L^2$-inner product and Cartesian Laplacian in the local coordinates of $\Gamma_\mbp$.  For any $a>0$ we define the interval 
\beq
\mbR_{a}=[-a/\varep,  a/\varep]. 
\eeq 
For each $f, g\in L^2(\Omega)$ with support in $\Gamma_\mbp^{2\ell}$, their $L^2$-inner product can be written in local coordinates  $(s_\mbp, z_\mbp)$ as
\beq
\left<f, g\right>_{L^2}=\int_{\mathbb R_{2\ell}}\int_{\msI}  f(s_\mbp, z_\mbp) g(s_\mbp, z_\mbp) \varep(1-\varep z_\mbp \kappa_\mbp) |\bm \gamma_\mbp'| \, \mrd s_\mbp\mrd z_\mbp; 
\eeq 
or equivalently in terms of $(\tilde s_\mbp, z_\mbp)$ with $\dd\tilde s_\mbp=|\bm \gamma_\mbp'|\dd s_\mbp$ and $\tilde s_\mbp\in\msI_\mbp$ as
\beq
\left<f, g\right>_{L^2}=\int_{\mathbb R_{2\ell}}\int_{\msI_\mbp}  f(s_\mbp, z_\mbp) g(s_\mbp, z_\mbp) \varep(1-\varep z_\mbp \kappa_\mbp) \, \mrd \tilde s_\mbp\mrd z_\mbp.
\eeq 
 
The $\varep$-scaled Laplacian admits the local expansion
\begin{equation}\label{eq-Lap-induced}
\varep^2 \Delta_{  x}= \p_{z_\mbp}^2+\varep \mathrm H_\mbp \p_{z_\mbp}+\varep^2\Delta_g, \qquad \Delta_{g}:= \Delta_{s_\mbp}+\varep z_\mbp  D_{s_\mbp,2},
\end{equation}
where $\mathrm H_\mbp$  is the extended curvature
\begin{equation}\label{def-H}
\mathrm H_\mbp(s_\mbp,z_\mbp):= -\frac{\kappa(s_\mbp)}{1-\varep z_\mbp\kappa(s_\mbp)},
\end{equation}
 $-\Delta_{s_\mbp}=-\p_{\tilde s_\mbp}^2$ is the Laplace-Beltrami operator on  the surface $\Gamma_\mbp$ and $D_{s_\mbp,2}$ is a relatively bounded perturbation of $\Delta_{s_\mbp}$. In particular, 
\begin{equation}\label{def-Ds}
  D_{s_\mbp,2} =a(s_\mbp, z_\mbp)\Delta_{s_\mbp}  + b(s_\mbp, z_\mbp) \p_{s_\mbp},
\end{equation}
where the smooth coefficients $a, b$ are given explicitly by
\begin{equation}\label{def-a,b}
a(s_\mbp, z_\mbp)=(\varep z_\mbp)^{-1} \left( \frac{1}{|1-\varep z_\mbp\kappa_\mbp|^2}-1\right), \qquad b(s_\mbp, z_\mbp)=\frac{(\varep z_\mbp)^{-1}}{2|\bm \gamma_\mbp'|^2}\p_{s_\mbp} a(s_\mbp, z_\mbp).
\end{equation}


\section{Bilayer manifold and linear stability}\label{sec-profile}
The bilayer manifold $\cM_b$ arises as the graph of the bilayer \muckymucks $\Phi_\mbp$ over its domain $\cD_\delta$. The $\mbp$-varied interfaces $\Gamma_\mbp$ are integral in the construction of of $\Phi_\mbp.$ We establish the coercivity of the FCH energy on the tangent plane to $\cM_b.$

\subsection{Bilayer manifold}

We refine the construction of the bilayer \muckymuck $\Phi_\mbp$,  developed in \cite{CP-linear} to accommodate the equilibrium $\Phi_{\mbp^*}$ of the FCH flow, and induce Lipschitz estimates on $\Phi_\mbp$ for $\mbp$ near $\mbp^*.$  

The bilayer \muckymuck is constructed through a matching of an inner description in the reach $\Gamma_\mbp^{2\ell}$ to an outer distribution on the remainder of $\Omega$.  The inner construction begins with $\phi_0$ defined on $L^2(\mathbb R)$ as the nontrivial homoclinic solution of 
\begin{equation}\label{def-phi0n2}
\p_z^2 \phi_{0}-W'(\phi_{0})=0, \qquad \lim_{|z|\to \infty} \phi_0(z)=b_-. 
\end{equation} 
The orbit $\phi_0$ is unique up to translation, even about $z=0$, and converges to $b_-$ as $z$ tends 
to $\pm \infty$ at the exponential rate $\sqrt{W''(b_-)}>0.$  The linearization  $\mathrm L_{\kpp 0}$ of  \eqref{def-phi0n2} about $\phi_0$, 
\begin{equation}\label{def-rL-p0}
\mathrm L_{\kpp 0}:=-\p_{z_{\kpp}}^2+W''(\phi_{0}(z_{\kpp})),
\end{equation}
is a Sturm-Liouville operator on the real line whose coefficients decay exponentially fast to constants at $z_{\kpp}=\infty$.  The following Lemma follows from classic results and direct calculations, see for example Chapter 2.3.2 of \cite{KP-13}. 
\begin{lemma}\label{lem-L0}
The spectrum of \,$\mathrm L_{\kpp 0}$ is real, and uniformly positive except for two point spectra: $\la_0<0$ and $\la_1=0$. The ground state eigenfunction $\psi_0$ of $\mathrm L_{\kpp 0}$ is even and positive, with ground state eigenvalue $\la_0<0$. The operator $\mrL_{\kpp 0}$ has an inverse that
is well defined on the $L^2$ perp of its kernel, span$\{\phi_0^\prime\},$ and both $\mrL_{\kpp 0}$ and its inverse preserve evenness and oddness. 
\end{lemma}

The base profile $\phi_0$ is a function of one variable, we extend it to $\Gamma_\mbp^{2\ell}$ and $\Omega$ through the dressing process.
\begin{defn}[Dressing]\label{def-dressing}
Given a function $f(z): \mathbb R\rightarrow \mathbb R$ which tends to a constant $f^\infty$ and whose
derivatives tend to zero at an $\varep$ independent exponential rate as $z\rightarrow \pm \infty$, we define the dressed function, $f^d\in L^2(\Omega)$, of $f$ with respect to $\Gamma_\mbp$ as
\begin{equation*}
f^d (x)=  f(z_\mbp(  x)) \chi(\varep |z_\mbp(  x)|/\ell)+f^\infty(1-\chi(\varep |z_\mbp(x)|/\ell)), \qquad \forall x\in\Omega.
\end{equation*} 
Here $\chi:\mathbb R\rightarrow \mathbb R$ is a fixed smooth cut-off function satisfying: $\chi(r)=1$ if $r\leq 1$ and $\chi(r)=0$ if $r\geq 2$. Where there is no ambiguity we abuse notation and use  $f(z_\mbp)$ to denote the dressed function $f^d$.   
\end{defn}

A function $f=f(x)\in L^1$ is said to be {\it localized} near the interface $\Gamma_\mbp$ if there exists $\nu>0$ such that for all $  x\in \Gamma_\mbp^{2\ell}$, 
\begin{equation*}
|f(  x(s_\mbp, z_\mbp))| \lesssim e^{-\nu|z_\mbp|}.
\end{equation*}
We introduce the dressed operator as follows.  
\begin{defn}[Dressed operator]\label{def-Ld}
Let $\mrL: {\mathrm D}\subset L^2(\mbR)\mapsto L^2(\mbR)$  be a self-adjoint differential operator with smooth coefficients whose derivatives of all order decay to zero at an exponential rate
at $\infty.$ We define the space $\mathcal S $ to consist of the functions $f$ as in Definition \ref{def-dressing}. The to each $\mbp\in\cD_\delta$ the 
dressed operator $\mrL_\mbp:\mathrm D\cap\mathcal S \mapsto L^2(\Omega)$ and its $r$'th power, $r\in\mathbb N$ are given by
\beq
\mrL_{\mbp}^r f := (\mrL^r f)^d .
\eeq
If $r<0$ then we assume that $f\in\mathcal R(\mrL)$ and the inverse $\mrL_\mbp^{-1} f$ decays exponentially to a constant at $\pm\infty.$
\end{defn}
Since $1$ is orthogonal to $\phi_0^\prime$ we may define
$\mathrm L_{\kpp 0}^{-k} 1$  on the real line $\mathbb R$ and its dressing subject to $\Gamma_\mbp$, denoted by $B_{\mbp,k}$  as a function of $x$, that is, 
\beq \label{def-B+dp,k}
B_{\mathbf p,k}(x):=\mrL_{\mbp,0}^{-k}1=(\mrL_0^{-k}1)^d , \qquad x\in \Omega.
\eeq 
 When $\mbp=\bm 0$, we drop the subscript $\mbp$ and denote $B_{\mbp,k}$ as $B_{k}$.   Recalling the averaging operator, \eqref{def-massfunc} we introduce
\begin{equation}
\label{def-Bdp-k-mass}
\oB _{\mathbf p, k}:={|\Omega|}\left<B_{\mathbf p,k} \right>_{L^2}=\int_\Omega B_{\mbp, k} \dd x.
\end{equation} 
With this notation we define the first correction $\phi_1$ to the pulse profile
\begin{equation}\label{def-phi1}
\begin{aligned}
\phi_1(\sigma)=\phi_{1}(z_{\mathbf p};\sigma)&:=\sigma B_{\mbp,2}+\frac{\eta_d}{2} \mathrm L_{\mathbf p,0}^{-1}\left( z_{\mathbf p}\phi_{0}'(z_\mbp) \right),
\end{aligned}
\end{equation}
which depends upon the bulk density and meander parameters, $\sigma$ and $\mbp$. 
As a function on $\mbR$,  $\phi_1$ is smooth and is even about $z=0$.  

The following Lemma establishes a family of bilayer \muckymucks that share a common leading order residual. In Definition\,\ref{def-Phi-p-1} we establish a particular choice of bilayer \muckymuck that supports the equilibrium bilayer \muckymuck with a nearly circular interface. 

\begin{lemma}\label{lem-def-Phi-p}
Let $\phi_0$, $\phi_1,$ be as defined in \eqref{def-phi0n2}, \eqref{def-phi1} respectively.  Then for any smooth function  $\phi_{\geq 3}=\phi_{\geq 3}(z_\mbp)$  that, up to terms exponentially small in $\varep$, decays exponentially to a constant $\phi_{\geq 3}^\infty$  as $|z|\to \infty$ 
and exponential term $\phi_{\mbp, e}=\phi_{\mbp, e}(x)$  defined on $\Omega$, 
 there exists $\nu>0$ and a function $\phi_2=\phi_2(\bm\gamma_\mbp'',z_\mbp;\sigma)$ that satisfies the assumptions of Notation\,\ref{Notation-h} such that the bilayer \muckymuck
\begin{equation}
\label{def-Phi-p}
\Phi_{\mathbf p}(x; \sigma):=\phi_{0}(z_{\mathbf p})+\varep \phi_{1}( z_{\mathbf p}; \sigma)+\varep^2\phi_{\geq 2}(\bm \gamma_\mbp'', z_{\mathbf p};\sigma)+e^{-\ell \nu/\varep}\phi_{\mbp, e}(x),
\end{equation}
with $\phi_{\geq 2}:=\phi_{2}+\varep \phi_{\geq 3}(z_\mbp)$
has the  residual 
\begin{equation}\label{exp-mF}
\begin{aligned}
\mathrm F(\Phi_{\mathbf p})=\mrF_m(s_\mbp, z_\mbp)+e^{-\ell\nu/\varep}\mrF_e(x) \quad \hbox{with}\quad  \mrF_m=\varep\sigma+\varep^2\mathrm F_2+\varep^3\mathrm F_3+\varep^4\mathrm F_{\geq 4},
\end{aligned}
\end{equation}
whose expansion terms take the form
\begin{equation}\label{F-234}
\begin{aligned}
\mathrm F_2&=\kappa_{\mathbf p}(\sigma-\sigma_1^*) f_2(z_{\mathbf p});\qquad \qquad 
\mathrm F_3=-\phi_0'\Delta_{s_{\mathbf p}}\kappa_{\mathbf p} +f_3(z_{\mathbf p}, \bm \gamma_{\mbp}''),\\
&\mathrm F_{\geq 4}= \Delta_{g}  f_{4,1}(z_{\mathbf p}, \bm \gamma_{\mbp}'')+f_{4,2}(z_{\mathbf p}, \bm \gamma_{\mbp}'').
\end{aligned}
\end{equation}
Here $f_2$ is localized near $\Gamma_\mbp$ and odd, and $f_3, f_{4,1}, f_{4,2}$ are smooth functions which decay exponentially fast to a constant as $|z|\to \infty$ and satisfy the assumptions of Notation\,\ref{Notation-h}. In addition,  the projections of $\mathrm F_2, \mathrm F_3$ satisfy
\begin{equation}\label{est-proj-rF2,3}
\begin{aligned}
&\int_{\mathbb R_{2\ell}} \mathrm F_2\,\phi_0'\dd z_{\mathbf p}=m_0(\sigma_1^*-\sigma)\kappa_{\mathbf p}+O(e^{-\ell\nu/\varep}); 
\\
&\int_{\mathbb R_{2\ell}} \mathrm F_3 \,\phi_0'\, \mrd z_{\mathbf p}=m_1^2\left(-\Delta_{s_{\mathbf p}}\kappa_{\mathbf p} -\frac{\kappa_{\mathbf p}^3}{2} +\alpha \kappa_{\mathbf p}\right)+O(e^{-\ell\nu/\varep}).
\end{aligned}
\end{equation}
Here  $\alpha =\alpha (\sigma; \eta_1, \eta_2)$  depends smoothly on $\sigma$.
\end{lemma}
\begin{proof}
This follows from Lemma 3.2 of \cite{CP-linear}. While we have incorporated lower order terms terms in $\Phi_\mbp$, they do not affect the form of $\mrF_{2,3,4}$. 
\end{proof}
\begin{remark}
Explicit forms of  $\phi_{2}$ and $\alpha$ are given in \cite{CP-linear}.  These forms are omitted as they do not impact the central result of this work.
\end{remark}

To recover the relaxation to equilibrium requires a particular choice for the lowest order term $\phi_{\geq 3}$ in Lemma\, \ref{lem-def-Phi-p}. Fixing a system mass $M_0$ determines an equilibrium of prescribed radius $\mrp_0^*$ and bulk density $\sigma^*$. This leaves only the translational parameters $\mrp_1$ and $\mrp_2$ to characterize the available equilibrium states. Under the FCH flow, we anticipate that the meander and bulk density parameters converge  to one of these equilibrium,
\beq
  \mbp(t)\to \mbp^*=(\mrp_0^*,\mrp_1^*,\mrp_2^*, 0, \cdots, 0), \quad \sigma(t)\to\sigma^*=\sigma_1^*+O(\varep).
\eeq  
The existence of equilibrium bilayer \muckymucks associated to circular interfaces, up to exponentially small terms, was established in \cite{DHPW-14}. The result is quoted here.
\begin{lemma}\label{lem-exist-equi}
Let $\varep>0$ be sufficiently small and let $\Gamma_*$ be a circle  with curvature $\kappa_*$, centered at the origin, and strictly contained within the periodic domain $\Omega$. Let  $z_*$ denote $\varep$-scaled  distance to $\Gamma_*$. Then for each $\varep$ sufficiently small there exists a unique constant  $\sigma^*=\sigma_1^*+\varep\sigma_{\geq2}^*(\kappa_*,\varep)$, a uniformly (in $\varep$) bounded 
function  $\phi_{\geq 3}^*=\phi_{\geq 3}^*(z_*;\kappa_*,\varep)$ which decays exponentially fast to a constant as $z\to\infty$,
and a uniformly (in $\varep$) smooth function $\phi_e=\phi_e(x;\kappa_*,\varep, \Omega)$ and a $\nu>0$ such that,
\begin{equation*}
\begin{aligned}
 \Phi_*(x):=&\phi_{0}(z_*(x))+\varep\phi_1(z_*(x);\sigma^*) +\varep^2\phi_2(z_*(x); \sigma^*,\kappa_*)+\varep^3\phi_{\geq 3}^*(z_*(x);\sigma^*, \kappa_*) \\
 &+ e^{-\ell\nu/\varep}\phi_{e}(x;\sigma^*, \kappa_*),
 \end{aligned}
\end{equation*}
are equilibrium of \eqref{eq-FCH-L2} subject to periodic boundary conditions on $\Omega.$ Translates of periodic extensions of $\Phi_*$ are also exact equilibrium.
\end{lemma}
The exponential correction $\phi_e$ is nontrivial outside the reach $\Gamma_*^{2\ell}$, and it arises from the interaction of the radial equilibrium inside the reach with the period box $\Omega=[-L,L]^2.$

\begin{defn}\label{def-Phi-p-1} The bilayer \muckymuck $\Phi_\mbp$ is defined through \eqref{def-Phi-p} with $\phi_{\geq3}$ chosen as a dress of  $\phi^*_{\geq3}$ around $\Gamma_\mbp$, and $\phi_{\mbp,e}(x)$ as a linear translation of $\phi_e$ from Lemma\,\ref{lem-exist-equi}. Specifically we take
 \beqs
 \phi_{\geq 3}(z_\mbp):=\phi_{\geq 3}^*(z_\mbp; \sigma^*, \kappa_*), \qquad   \phi_{\mbp,e}
=\phi_e(x-\mrp_1\Theta_0\bE_1-\mrp_2\Theta_0\bE_2; \kappa_*, \sigma^*).
 \eeqs
\end{defn}
The FCH gradient flow preserves system mass, which is set by the initial data. We adjust the value of the bulk density parameter $\sigma$ so that the mass of  $\Phi_{\mathbf p}$ equals the system mass. From the form \eqref{def-Phi-p} of $\Phi_{\mbp}$ with $\phi_1=\phi_1(\sigma)$ given by \eqref{def-phi1}  we deduce that 
\beq
\label{eqn:SysMass}
\langle \Phi_\mbp(x)-b_-\rangle_{L^2} = \varep\frac{M_0}{|\Omega|},
\eeq
precisely if 
\begin{equation}\label{def-sigma}
\begin{aligned}
\sigma(\mathbf p)= \frac{1}{\oB_{\mathbf p,2} } \Bigg\{ M_0- \int_\Omega \bigg[&\frac{1}{\varep}\Big(  \phi_0(z_{\mathbf p}) -b_-  +\varep^2\phi_{\geq 2} (\bm \gamma''_{\mathbf p}, z_{\mathbf p})+e^{-\ell\nu/\varep}\phi_{\mbp, e}\Big) +\frac{\eta_d}{2} \left(\mathrm L_{\mathbf p, 0}^{-1}(z_{\mathbf p}\phi_0')\right) \bigg]\, \mrd   x\Bigg\}.
\end{aligned}
\end{equation}
The following result approximates the dependence of $\sigma$ on $(\mrp_0, \hat\mbp)$. 
\begin{lemma}\label{lem-sigma}
Let $\varep\in(0,\varep_0)$ with $\varep_0$ small enough, and consider system mass of the form \eqref{eqn:SysMass} with $M_0$ independent of $\varep$.  
The dependence of the bulk density $\sigma=\sigma(\mbp)$ on $\mbp\in\cD_\delta$ satisfies
\begin{equation*}
\sigma(\mbp)=\sigma_0-\frac{c_0 m_1^2}{m_0}R_0 \mathrm p_0+\varep  \mathcal C(\mrp_0)\mrp_0+O\left(\varep^2 \|\hat{\mathbf p}\|_{\mbV_2^2}\right)
\end{equation*}
for some smooth function $\mathcal C(\mrp_0)$  of $\mrp_0$ that is bounded uniformly, independent of $\varep$. The base bulk density $\sigma_0=\sigma_0(\Gamma_0, M_0)$ is independent of $\mbp$, and  $c_0$ is a fixed positive constant; indeed,
\beq\label{def-c0}
\sigma_0(\Gamma_0, M_0)=\frac{M_0  -m_0|\Gamma_0|}{B_2^\infty|\Omega|}  +O(\varep^2),  \qquad c_0=\frac{2\pi m_0^2}{\oB_2 m_1^2}>0.
\eeq
\end{lemma}
\begin{remark}
The base interface and system mass $(\Gamma_0,M_0)$ form an admissible pair if and only if $\sigma_0(\Gamma_0,M_0)$ is bounded uniformly in $\varep.$
\end{remark}
\begin{proof} We address the terms on the right-hand side of \eqref{def-sigma} one by one. First,  under the local coordinates near $\Gamma_\mbp$  we rewrite
\beqs
\int_\Omega (\phi_0(z_\mbp)-b_-) \dd x= \varep \int_{\msI}\int_{\mbR_{2\ell}} (\phi_0(z_\mbp)-b_-) |\bm \gamma_\mbp'| (1-\varep z_\mbp\kappa_\mbp) \dd s_\mbp\dd z_\mbp.
\eeqs 
Since $\phi_0(z_\mbp)-b_-=\phi_0^d-b_-$ is localized near $\Gamma_\mbp$ and  even with respect to $z_\mbp$, we have 
\beqs
\int_\Omega (\phi_0(z_\mbp)-b_-) \dd x= \varep |\Gamma_\mbp| \int_{\mbR_{2\ell}} (\phi_0(z) -b_-)\chi(\varep z/\ell) \dd z.
\eeqs 
With $m_0$ defined in \eqref{def-m0}, there exists some constant $C_1$ independent of $\mbp$ such that
\beq\label{sigma-est-1}
\int_\Omega (\phi_0(z_\mbp)-b_-) \dd x= \varep |\Gamma_\mbp| \left(m_0 +C_1e^{-\ell\nu/\varep}\right). 
\eeq
The remaining leading order term depends on $z_\mbp$ only and is localized. We deduce that
\beq\label{sigma-est-2}
\int_\Omega  \mrL_{\mbp, 0}^{-1}(z_\mbp \phi_0')\dd x =  C_2\varep |\Gamma_\mbp| , \qquad C_2:= \int_{\mbR_{2\ell}} \mrL_0^{-1} (z\phi_0')\chi(\varep z/\ell) \dd z.    
\eeq
The constant $\oB_{\mbp, 2}$ defined in \eqref{def-sigma} is the mass of the dressed function $B_{\mbp, 2}$ introduced in \eqref{def-B+dp,k}-\eqref{def-Bdp-k-mass}. Since $B_{\mbp, 2}$ approaches $B_{\mbp, 2}^\infty=B_2^\infty$ as $|z|\to \infty$, we may rewrite 
\beqs
\oB_{\mbp,2}=\int_\Omega \left(B_{\mbp, 2}-B_{2}^\infty\right) \dd x +B_2^\infty |\Omega|  
\eeqs
The first term is localized and even with respect to $z_\mbp$ from which we deduce
\beq\label{sigma-est-3}
\oB_{\mbp,2}=\varep C_3  |\Gamma_\mbp| +B_2^\infty |\Omega|, \qquad  C_3:= \int_{\mbR_{2\ell}} (B_2 -B_2^\infty)\chi(\varep z/\ell) \dd z. 
\eeq
Finally, the term  $\phi_{\geq 2}=\phi_{\geq 2}(z_\mbp, \bm \gamma_\mbp'')$ has the properties of Notation \ref{Notation-h}. Subtracting the far field value $\phi_{\geq 2}^\infty$, integrating out $z_\mbp$ and applying Lemma \ref{lem-est-h-proj} from the Appendix for $j=0, k=0$  we find  
\beq\label{sigma-est-4}
\int_\Omega \phi_{\geq 2}(z_\mbp, s_\mbp)\dd x=\phi_{\geq 2}^\infty |\Omega| +\varep f(\mrp_0) +O(\varep \|\hat\mbp\|_{\mbV_2^2}),  
\eeq
for some smooth function $f=f(\mrp_0)$. Combining \eqref{sigma-est-1}-\eqref{sigma-est-2} and \eqref{sigma-est-3}-\eqref{sigma-est-4} with \eqref{def-sigma} yields
\beqs
\sigma(\mbp)= \frac{1}{B_2^\infty|\Omega| +\varep C_3|\Gamma_\mbp|} \left\{M_0 -\frac{e^{-\ell\nu/\varep} }{\varep}  M_e -\varep^2 \phi_{\geq 2}^\infty |\Omega| - (m_0 +C\varep)|\Gamma_\mbp|  +\varep^2 f(\mrp_0) +O(\varep^2 \|\hat\mbp\|_{\mbV_2^2}) \right\}
\eeqs
where $C=C_1\varep^{-1}e^{-\ell\nu/\varep}+C_2 \lesssim 1$  and 
\beqs
M_e
=\int_\Omega \phi_{\mbp,e}(x)\dd x = \int_\Omega \phi_{e}(x)\dd x
\eeqs
is a mass correction arising from the exponentially small term  $\phi_{\mbp, e}$ from Definition\,\ref{def-Phi-p-1}.
The Lemma follows from \eqref{est-A-|Gamma-p|} by extracting the leading order terms, introducing the $\mbp$ independent constants $c_0$,
\beqs
\sigma_0:= \frac{1}{B_2^\infty|\Omega|} \left\{M_0 -\frac{e^{-\ell\nu/\varep} }{\varep}  M_e -\varep^2 \phi_{\geq 2}^\infty |\Omega| -m_0|\Gamma_0|\right\} \lesssim 1,
\eeqs
and taking $\varep_0$ small enough. 
\end{proof}

We introduce the equilibrium meander parameters $\mbp^*=(\mrp_0^*,\mrp_1^*, \mrp_2^*, \bm 0)$ for which $\sigma(\mbp^*)=\sigma^*.$  With this notation we have the following result.

\begin{cor}\label{cor-def-hatla} 
Suppose $\mbp^*=(\mrp_0^*, \mrp_1^*, \mrp_2^*, \bm 0)\in \cD_\delta$ the bulk density parameter $\sigma=\sigma(\mbp)$ satisfies the relation  
\begin{equation}\label{re-hatla*-p0*}
\sigma^*-\sigma(\mathbf{p})= \frac{c_0m_1^2}{m_0} R_0(\mathrm p_0-\mathrm p_0^*)+O\left(\varep|\mathrm p_0-\mathrm p_0^*|, \varep^2\|\hat{\mathbf p}\|_{\mbV_2^2}\right).
\end{equation} 

\end{cor}

To establish the long-time convergence to equilibrium requires Lipschitz estimates on the residual of the bilayer \muckymucks $\Phi_\mbp$ for $\mbp$ near $\mbp^*.$   
\begin{lemma}\label{lem-est-L2-F}
For $\mbp\in \cD_\delta$ 
the components of the residual $\mrF(\Phi_{\mathbf p})$ given in \eqref{exp-mF} satisfy
\begin{equation*}
\|\mathrm F_2\|_{L^2}\lesssim \varep^{1/2}|\sigma-\sigma^*| + \varep^{3/2}; \;\; \|\mathrm F_3-\mrF_3^\infty\|_{L^2}+\|\mathrm F_{\geq 4}-\mrF_{\geq 4}^\infty\|_{L^2}\lesssim \varep^{1/2}\left(\|\hat{\mathbf p}\|_{\mbV_4^2}+1\right),
\end{equation*}
and the exponential residual satisfies $\|\mrF_e\|_{L^2}\lesssim 1$. If $\Phi_{\mbp}$ is chosen as in Definition\,\ref{def-Phi-p-1}   then the overall residual satisfies the Lipschitz estimate
\begin{equation*}
\|\Pi_0\mathrm F(\Phi_{\mathbf p})\|_{L^2}\lesssim \varep^{5/2}|\mrp_0-\mrp_0^*|+\varep^{5/2}\|\hat{\mathbf p}\|_{\mbV_2^2}+\varep^{7/2}\|\hat{\mathbf p}\|_{\mbV_4^2}.
\end{equation*}
\end{lemma}
\begin{proof} 
The $L^2$-bounds  of the difference between  $\mathrm F_k$ and its bulk value for $k=2, 3, 4$ follow from the expressions for $\mathrm F_2, \mathrm F_3, \mathrm F_4$ given 
in Lemma \ref{lem-def-Phi-p}, see \cite{CP-linear} for details. 
The $L^2$-estimate of $\Pi_0\mathrm F(\Phi_{\mathbf p})$ follows from comparing it to the zero residual of $\Phi_*$. Indeed we write
\begin{equation}\label{est-L2-Pi0F-1}
\|\Pi_0\mathrm F(\Phi_{\mathbf p})\|_{L^2}=\|\Pi_0\mathrm F(\Phi_{\mathbf p})-\Pi_0\mathrm F(\Phi_*)\|_{L^2}
\end{equation}
where $\Phi_*$ is the equilibrium solution associated with bulk density state $\sigma^*$ and interface  $\Gamma_*$ with parameterization
\beq\label{def-Gamma*}
\bm \gamma_*:=\bm \gamma_{\mathbf p^*}+\Theta_0\Big((\mathrm p_1-\mathrm p_1^*)\bE_1+(\mathrm p_2-\mathrm p_2^*)\bE_2\Big).
\eeq 
obtained by translating $\Gamma_{\mathbf p^*}$ to place its center at $(\mrp_1,\mrp_2)$.
The triangle inequality and the expansion of $\mathrm F=\mathrm F(\Phi_{\mathbf p})$ from Lemma \ref{lem-def-Phi-p}, yield the estimate
\begin{equation}\label{est-L2-Pi0F-2}
\begin{aligned}
\|\Pi_0(\mathrm F(\Phi_{\mathbf p})- \mathrm F(\Phi_*))\|_{L^2} \leq& \varep^2\|\Pi_0(\mathrm F_2-\mathrm F_2(\Phi_*))\|_{L^2}+\varep^3\|\Pi_0(\mathrm F_3-\mathrm F_3(\Phi_{*}))\|_{L^2}\\
      &+\varep^4\| \Pi_0(\mathrm F_{\geq 4}- \mathrm F_{\geq 4}(\Phi_*))\|_{L^2} +e^{-\ell\nu/\varep}\|\mrF_e-\mrF_e(\Phi_*)\|_{L^2}.
\end{aligned}
\end{equation}
We use the form of the $\mrF_k(\Phi_*)$ residuals to establish that they are Lipschitz in $(\mrp_0-\mrp_0^*, \hat\mbp)$.
We observe from Lemma \ref{lem-def-Phi-p} that $\mathrm F_2$ admits the general form 
$\mathrm F_2=\kappa_{\mathbf p} f_2(z_{\mathbf p})(\sigma-\sigma_1^*)$, while $\mathrm F_2(\Phi_*)= \kappa_*f_2(z_*)(\sigma^*-\sigma_1^*)$.  We deduce that
\begin{equation*}
\|\mathrm F_2-\mathrm F_2(\Phi_*)\|_{L^2}^2\leq |\sigma-\sigma^*|^2 \int_\Omega \kappa_{\mathbf p}^2f_2^2(z_{\mathbf p})\dd   x 
+|\sigma^*-\sigma_1^*|^2\int_\Omega |\kappa_{\mathbf p} f_2(z_{\mathbf p}) -\kappa_*f_2(z_*)|^2\dd   x.
\end{equation*}
The integrals contribute a factor of $\varep$ since the integrands are bounded and their support is localized near the interfaces $\Gamma_{\mathbf p}$ and $\Gamma_*$, respectively.  We decompose the second integrand as $\kappa_{\mathbf p}\left(f(z_{\mathbf p})-f(z_*)\right)+\left(\kappa_{\mathbf p}-\kappa_*\right)f(z_*)$ which we bound by 
$|z_{\mathbf p}-z_*|+|\kappa_{\mathbf p}-\kappa_*|$ in its support set. Using the estimates of Lemmas \ref{lem-change-of-coord}, \ref{lem-Gamma-p} and \ref{lem-est-V},  recalling
 $|\sigma^*-\sigma_1^*|\lesssim \varep$ we arrive at the bound
\begin{equation}\label{est-F2}
\|\Pi_0\mathrm F_2-\Pi_0\mathrm F_2(\Phi_*)\|_{L^2}^2\lesssim \varep|\sigma-\sigma^*|^2+\varep\left(\,|\mathrm p_0-\mathrm p_0^*|^2+\|\hat{\mathbf p}\|_{\mbV_2^2}^2\right).
\end{equation}
 The $L^2$ bounds of $\mathrm F_3-\mathrm F_3(\Phi_*)$ and $\mathrm F_{\geq 4}-\mathrm F_{\geq 4}(\Phi_*)$ involve higher derivatives of the perturbed curve $\bm \gamma_\mbp$ from \eqref{def-gamma-p} which are controlled with through  \eqref{ortho-tTheta}, specifically 
\begin{equation*}
\|\varep^2\Delta_{s_\mbp}\bm \gamma_\mbp^{(k)}\|_{L^2(\msI_\mbp)}+ \|\bm \gamma_\mbp^{(k)}-\bm \gamma_{*}^{(k)}\|_{L^2(\msI)}\lesssim \|\hat{\mathbf p}\|_{\mbV_{k}^2} +|\mrp_0-\mrp_0^*|.
\end{equation*}
And with a use of the above bound and form of $\mrF_3, \mrF_4$ in Lemma \ref{lem-def-Phi-p}, one can establish similarly 
\beq\label{est-F34}
\|\Pi_0\mrF_3-\Pi_0\mrF_3(\Phi_*)\|_{L^2}^2 +\|\Pi_0\mrF_4-\Pi_0\mrF_4(\Phi_*)\|_{L^2}^2 \lesssim  \|\hat\mbp\|_{\mbV_4^2}^2 +\varep^{-1} |\mrp_0-\mrp_0^*|^2.
\eeq 
The term $\mrF_e$ incorporates residual from $\phi_e$ and from the dressing process. However $\phi_e$ in Lemma \ref{lem-def-Phi-p} cancels with the corresponding term in $\Phi_*$, and the
residual is due solely to the dressing which contributes
\beq\label{est-Fe}
\|\mrF_e-\mrF_e(\Phi_*)\|_{L^2} 
\lesssim  \|\hat\mbp\|_{\mbV_4^2}.
\eeq
Reporting the term-wise bounds, \eqref{est-F2}, \eqref{est-F34} and \eqref{est-Fe}, into \eqref{est-L2-Pi0F-2}, and applying Corollary \eqref{cor-def-hatla} completes the proof.
\end{proof}

\subsection{Linear coercivity}\label{sec-linear}
The  nonlinear stability analysis hinges upon the properties of the  linearization of the flow \eqref{eq-FCH-L2} about the  bilayer \muckymucks $\Phi_\mbp$ introduced in Definition \ref{def-Phi-p-1}.   The linearization takes the form $\Pi_0\mbL_{\mathbf p}$ where
\begin{equation}\label{def-bLp}
\begin{aligned}
\mbL_{\mathbf p} :=\frac{\delta^2\mathcal F}{\delta u^2}\Big|_{u=\Phi_{\mathbf p} }=&(\varep^2\Delta-W''(\Phi_{\mathbf p} )+\varep \eta_1)(\varep^2\Delta -W''(\Phi_{\mathbf p} )) \\
&- (\varep^2 \Delta \Phi_{\mathbf p} - W'(\Phi_{\mathbf p} ))W'''(\Phi_{\mathbf p} )+\varep\eta_dW''(\Phi_{\mathbf p} ),
\end{aligned}
\end{equation}
denotes the second variational derivative of the free energy $\mathcal F$ at $\Phi_{\mathbf p} $ and $\eta_d:=\eta_1-\eta_2$. 
When restricted to functions with support within the $\Gamma_\mbp^{2\ell}$,  the Cartesian Laplacian admits the expansion \eqref{eq-Lap-induced} and  the leading order operator takes the form
\begin{equation}
\begin{aligned}
\label{def-cLp0}
\mbL_{\mbp, 0} := \mathcal L_\mbp^2, \qquad \hbox{where} \quad \mathcal L_\mbp:=\mrL_{\mbp,0}-\varep^2\Delta_{s_{\mbp}}. 
\end{aligned}
\end{equation}
 Spectral analysis of the leading order operator  $\mbL_{\mbp,0}$ leads to definitions of slow spaces $\mcZ$, for which the operator is uniformly coercive on the orthogonal complement,
  see \cite{DHPW-14, HP-15, NP-17}. However these spaces are only invariant under the linear operator up to order of $\varep.$ 
This is not sufficient to close the nonlinear energy estimates required to establish stability. Consequently the modified space slow spaces were introduced in \cite{CP-linear}. We summarize these results below.
 


Let $\psi_{k}(z)$ denote the normalized  eigenfunctions of the operator $\mathrm L_{0}$ defined in Lemma\,\ref{lem-L0}. The dressed  and scaled version $\tilde\psi _k$ is defined as
\begin{equation*}
\tilde \psi_{ k} (z_{\mathbf p}):=\varep^{-1/2}\psi _{ k}(z_{\mathbf p}).
\end{equation*}
\begin{defn}[Slow spaces]\label{def-slow-space}
For $k=0, 1$, fixed $\rho>0$, and all $\mbp\in\cD_\delta$ we introduce the disjoint index sets:
\begin{equation}\label{def-Sigma}
\Sigma_k=\Sigma_k(\mathbf p, \rho)= \{j \bigl| \,\La_{ki}^2:=(\la_k+\varep^2\beta_{\mathbf p,i}^2)^2 \leq \rho\}, \quad \hbox{and}\quad \Sigma:=\Sigma_0\cup \Sigma_1,
\end{equation}
and the  slow space $\mathcal Z_*=\mathcal Z_*(\mathbf p, \rho)\subset L^2$, as the union of the pearling and meander spaces, $\mcZ_*^0$ and $\mcZ_*^1$,
\begin{equation}\label{def-Z*}
\mathcal Z_*:=\mathcal Z_*^0\cup \mathcal Z_*^1 \qquad  \hbox{with}\qquad \mathcal Z_*^k=\mathrm{span}\left\{Z_{\mathbf p, *}^{I(i)i}, i\in \Sigma_k\right\}.
\end{equation} 
Here the modified basis functions take the form
\begin{equation}\label{eq-def-Z*}
\begin{aligned}
Z_{\mathbf p,*}^{I(i)i}&:=\left(\tilde \psi_{I(i)} +\varep \tilde \varphi_{1,i} \right)\tilde \Theta_i+\varep\tilde\varphi_{2,i}  \varep \tilde \Theta_i', 
\end{aligned}
\end{equation}
where $\tilde \varphi_{l,i}=\varep^{-1/2} \varphi_{l,i}(z_{\mathbf p}, \bm \gamma''_{\mathbf p})$, enjoys the properties of Notation\,\ref{Notation-h},  localized near $\Gamma_\mbp$   and satisfying
\begin{equation}\label{ortho-varphi}
\int_{\mathbb R}\varphi_{l,i}(z,\bm \gamma_{\mbp}'') \psi_{I(i)}(z)\dd z=0, \qquad l=1,2, \quad i\in \Sigma.
\end{equation}
 The leading order term of $Z_{\mbp, *}^{I(i)i}$, obtained by setting  $\varphi_{l,i}=0$,  is denoted $Z_{\mbp}^{I(i)i}$, that is 
\beq
\label{def-Zp}
Z_{\mbp}^{I(i)i}:=\tilde \psi_{I(i)}\tilde \Theta_i. 
\eeq
\end{defn}
The Laplace-Beltrami eigenvalues $\beta_{\mbp, j}$ are defined in \eqref{tTheta''}. They satisfy the classical Weyl's Law asymptotics, \cite{W-1910},
\beq\label{est-N01}
N_0:=|\Sigma_0| \sim \varep^{-1}\rho^{1/2}, \qquad N_1:=|\Sigma_1|\sim \varep^{-1}\rho^{1/4}.
\eeq
In addition, for $\rho$ suitably small, the pearling index set $\Sigma_0$ is far away from the meandering index set $\Sigma_1$. In particular, 
\beq\label{MP-gap}
|\beta_k-\beta_j| \geq C\varep^{-1}, \qquad k\in \Sigma_0, j\in \Sigma_1. 
\eeq
 Thanks to the pearling-meandering gap \eqref{MP-gap}, orthogonality \eqref{ortho-tTheta} and \eqref{ortho-varphi}, the basis functions of $\mcZ_*$ are almost orthogonal  as
 \begin{equation}\label{ortho-Z*}
\begin{aligned}
&\left<Z_{\mathbf p, *}^{I(i)i}, Z_{\mathbf p, *}^{I(j)j}\right>_{L^2}= \left\{ \displaystyle\begin{array}{cc}\left(1+\mrp_0\right)\delta_{ij}+O\left(\varep^2, \varep^{2}\|\hat{\mathbf p}\|_{\mbV_2^2}\right)\mbE_{ij},  &I(i)=I(j);\\[2pt]
 \qquad O\left(\varep^2,  \varep^2\|\hat{\mathbf p}\|_{\mbV_4^2}\right)\mbE_{ij}, \qquad\qquad \qquad  &I(i)\neq I(j).
  \end{array}\right.
  \end{aligned}
\end{equation}  
We define the $L^2$ linear projection on the subspace $\mathcal Z_*^k$  by $\Pi_{\mathcal Z_*^k}$ for $k=0, 1$. In particular, for any $u\in L^2$ there exists a unique constant vector $(u_i)_{i\in \Sigma_k}\in l^2(\mbR^{N_k})$ such that $\Pi_{\mcZ_*^k}u=\sum_{i\in \Sigma_k}u_iZ_{\mbp,*}^{ki}$.
When restricted to $\mathcal Z_*$ the bilinear form of the full linearized operator $\Pi_0\mbL_{\mathbf p} \big|_{\mathcal Z_*}$, induces an $N\times N$ matrix $\mathbb M^*$ with entries
 \begin{equation}\label{def-M+*}
 \mathbb M_{ij}^*=\left<\Pi_0 \mbL_{\mathbf p}  Z_{\mathbf p, *}^{I(i)i}, Z_{\mathbf p, *}^{I(j)j}\right>_{L^2}.
 \end{equation}
We decompose $\mathbb M^*$ into a block structure corresponding to the pearling and meandering spaces,
\begin{equation}
\mathbb M^*=\left(\begin{array}{cc}
\mathbb M^*(0,0) & \mathbb M^*(0,1)\\
\mathbb M^*(1,0) &\mathbb M^*(1,1)
\end{array}\right),\qquad \mathbb M_{ij}^*(k,l)=\mathbb M^*_{ij} \quad \hbox{for}\;\;\,i\in \Sigma_k, j\in \Sigma_l.
\end{equation}
A detailed analysis of $\mbM^*$ is given in \cite{CP-linear}. In particular, under the {\sl pearling stability condition} 
\beq\label{cond-P-stab}
(\mathbf{PSC}) \qquad \sigma S_1+\eta_d \la_0>0,
\eeq
the pearling sub-block $\mbM^*(0,0)$ is positive definite, that is, \begin{equation*} \mbq^T \mbM^*(0,0) \mbq \geq  \frac{\varep}{2}(1+\mrp_0)(\sigma S_1+\eta_d \la_0) \|\mbq\|_{l^2}^2, \quad \textrm {for\, all}\, \mbq\in\mbR^{N_0}.\end{equation*}

 \begin{lemma}\label{thm-coupling est}
Let $\varep\in(0,\varep_0)$ with $\varep_0$ and $\rho$ suitably small.  
For any  function $v^\bot\in (\mcZ^1_*)^\bot$ and $\mbp\in \cD_\delta$, we can bound the meandering coupling as following: 
\beqs
 \|\Pi_{\mathcal Z^1_*}\mbL_{\mathbf p}   v^\bot\|_{L^2}\lesssim \left(\varep^2+ \varep^{2}\|\hat\mbp\|_{\mbV_4^2}\right)\|v\|_{L^2}.
\eeqs
\end{lemma}
\begin{proof}
This is a direct consequence of Theorem 4.11 of \cite{CP-linear}. 
\end{proof}

Perpendicular to the slow space $\mcZ_*$,  the linearized operator $\Pi_0\mbL_{\mathbf p} $ is coercive. 
\begin{lemma}\label{lem-coer}
Fix $\rho>0$ and  $\varep_0>0$ sufficiently small. Then there exists $C>0$ such that for all $\varep\in(0, \varep_0)$ and all $w\in \mathcal Z_*^\bot$  we have the coercivity estimates
\begin{equation}\label{coer-bLp}
\left< \mbL_{\mathbf p} w,w\right>_{L^2}\geq C\rho^{2} \left(\varep^4\|w\|_{H^2}^2+\|w\|_{L^2}^2\right)\qquad\hbox{and}\quad \|\mbL_{\mathbf p}  w\|_{L^2}^2\geq  C \rho^{2}\left< \mbL_{\mathbf p} w,w\right>_{L^2}.
\end{equation}
\end{lemma}

 We say that the base interface $\Gamma_0$ and the scaled system mass $M_0$ introduced in \eqref{def-Mass} are an \textit{admissible pair} if the system mass balances with the length of $\Gamma_0$ in the sense that 
 \beq \label{adm-pair}
 \left|M_0- m_0|\Gamma_0| \right|\lesssim  1.
 \eeq
The bilayer manifold is defined as the graph of the bilayer \muckymucks over the domain $\cD_\delta.$ 
\begin{defn}[Bilayer Manifold]
\label{def-bM0}
  Given a  circle $\Gamma_0$ with radius $R_0>0$ and system mass $M_0$ which form an admissible pair, \eqref{adm-pair}. We define the $N_1$-dimensional  bilayer manifold $\cM_b(\Gamma_0,M_0;\rho)$  
  to be the graph of the map $\mbp\mapsto \Phi_\mbp(\sigma(\mbp))$ over the domain $\cD_\delta$, where the bilayer \muckymuck $\Phi_\mbp(\sigma)$ is introduced in Definition \ref{def-Phi-p-1}  with bulk density $\sigma=\sigma(\mbp)$ prescribed through  \eqref{eqn:SysMass}, or alternatively \eqref{def-sigma}.  The spectral parameter $\rho$ controls the dimension $N_1$, and is introduced in Definition \ref{def-slow-space}.
\end{defn}


\section{Dynamics of meandering modulation parameters}\label{sec-p}
We derive the projection of the mass preserving gradient flow \eqref{eq-FCH-L2-p}-\eqref{eq-FCH-L2} onto the approximate tangent space of the bilayer manifold. We restrict attention to the affine space defined by the mass constraint \ref{def-Mass}. We denote by $\mcV_R(\cM_b, \cO_{m,\delta})$ the tubular neighborhood of thickness $R$ in the $H^2$ inner norm, \eqref{def-Hin}, that surrounds the restriction of the bilayer manifold $\cM_b$ restricted to domain $\cO_{m, \delta}\subset\cD_\delta$,
\beq
\label{def-tubular}
\mcV_R(\cM_b, \cO_{m,\delta}):= \left\{u\in H^2 \; \big| \; \min_{\mbp\in \cO_{m,\delta}}\|u-\Phi_\mbp\|_{\Hin}<R, \quad \left<u-b_-\right>_{L^2}=\frac{\varep M_0}{|\Omega|}\right\}.
\eeq 
where $\cO_{m,\delta}$ as a subset of $\cD_\delta$, defined in \eqref{A-00}, is given by
\beq\label{def-cO-p}
 \mathcal O_{m,\delta}  := \left\{ \mbp \in \mbR^{N_1}\; \big| \; |\mrp_0|+ \|\hat\mbp\|_{\mbV_3^2}\leq m \delta \right\}\subset \cD_\delta.
\eeq 

\begin{defn}[Manifold Projection] 
Let $\mcU$ be a neighborhood of $\cM_b.$ We say $\Pi_{\cM_b}u:=\Phi_\mbp(\sigma)$ is a projection of $\mcU$ onto $\cM_b$ and $\Pi_{\cM_b}^\bot u:= v^\bot $ is its compliment if for each $u\in\mcU$ there exist a unique  $\mbp\in \cD_\delta$ and mass-free orthogonal perturbation $v^\bot\in (\mcZ_{*}^1(\mbp))^\bot$ such that
\beq\label{decomp-u-Mb}
u=\Phi_\mbp+v^\bot.
\eeq
We call $\mbp$ the meander parameters associated to $u.$
\end{defn}
The following, from Lemma 5.2 of \cite{CP-linear}, establishes the existence of a manifold projection.
\begin{lemma}\label{lem-Manifold-Projection}
Let $\cM_b=\cM_b(\Gamma_0,M_0)$ be the bilayer manifold of Definition  \ref{def-bM0}. Let $\delta,\varep_0>0$ be sufficiently small, then for all $\varep\in(0,\varep_0)$ there exists a manifold projection $\Pi_{\cM_b}$ defined on the tubular set $\mcU=\mcV_{\delta\varep}(\cM_b, \cO_{2,\delta}).$  Moreover, for $u\in\mcU$ of the form $u=\Phi_{\mbp_0}+v$ with $\mbp_0\in\cD_\delta$ and massless perturbation $v\in H^2$ satisfying $\|v\|_{\Hin}\leq \delta \varep$, then $u$'s meander parameters $\mbp$ and orthogonal perturbation, $v^\bot$ satisfy
\begin{equation*}
\begin{aligned}
\|\mbp-\mbp_0\|_{l^2} \lesssim \varep^{1/2}\|v\|_{L^2}; \qquad \|v^\bot\|_{\Hin}\lesssim \|v\|_{\Hin}. 
\end{aligned}
\end{equation*}
\end{lemma}


Assume that a solution $u=u(t)$ of the FCH gradient flow satisfies $u\in \mcV_{C\varep^{5/2}}(\cM_b, \cO_{2,\delta})$ on the interval $[0,T]$. Then for $\varep_0$ sufficiently small, depending on $\delta,$ we have $u\in\mcU$ and we may decompose $u$ as
\begin{equation}\label{decomp-u}
u(  x,t) = \Phi_{\mathbf p} (  x;  \sigma) +v^\bot( x, t),  \quad v^\bot\in (\mathcal Z_*^1)^\bot ,\quad \int_\Omega v^\bot\dd   x=0. 
\end{equation} 
The FCH gradient flow  \eqref{eq-FCH-L2} can be written in terms of the pair $(\mbp,v^\bot)=(\mbp(t),v^\bot(t))$ with the bulk density parameter $\sigma=\sigma(\mbp(t))$ through \eqref{def-sigma}.  Substituting the decomposition \eqref{decomp-u} into the FCH gradient flow leads to an equation for $\Phi_\mbp$ and $v^\bot$:
\begin{equation}\label{eq-v-Phi}
\p_t \Phi_{\mathbf p} =-\Pi_0 \mathrm F(\Phi_{\mathbf p} ) +\Rem [v^\bot],
\end{equation}
where  $\Rem[v^\bot]$ is the orthogonal remainder contributed by orthogonal perturbation $v^\bot$. Specifically
\beq\label{def-Rv}
\Rem[v^\bot]:= - \p_t v^\bot -\Pi_0 \mbL_{\mathbf p}  v^\bot-\Pi_0  \mathrm N( v^\bot),
\eeq
 where $\mathrm N(v^\bot)$ is the genuinely nonlinear term defined by
\begin{equation}\label{def-N}
\mathrm N( v^\bot):= \mathrm F(\Phi_{\mathbf p} + v^\bot)-\mathrm F(\Phi_{\mathbf p} ) - \mbL_{\mathbf p}  v^\bot.
\end{equation}
As substantiated in Remark \ref{remark-timede-proj} and Corollary \ref{cor-residual-proj}, projecting the flow \eqref{eq-v-Phi} onto the meander space $\mcZ_*^1$ yields the dynamics for $\mbp$, which at leading order is equivalent to the finite dimensional geometric flow defined by the system
\beq
\label{def-Vp}
\Pi_{G_1}\left(\p_t \bm \gamma_\mbp \cdot \mbn_\mbp - V_\mbp\right)=0, 
\eeq 
where the curvature induced velocity is given by 
\beq\label{normal-velocity}
V_\mbp= \varep^3 \frac{m_0}{m_1^2}(\sigma_1^*-\sigma)\kappa_\mbp -\varep^4 \left( \Delta_{s_\mbp}\kappa_\mbp +\frac{\kappa_\mbp^3}{2}+\alpha \kappa_{\mbp}\right), 
\eeq 
and $\Pi_{G_1}:L^2(\msI_\mbp)\mapsto L^2(\msI_\mbp)$ is the projection onto the space $G_1\subset L^2(\msI_\mbp)$ spanned by the first $N_1$ Laplace-Beltrami modes
\beq \label{def-PiG}
 \Pi_{G_1} f := \frac{1}{1+\mrp_0}\sum\limits_{j=0}^{N_1-1} 
 \tilde\Theta_j  \int_{\msI_\mbp}  f(\tilde s_\mbp)\tilde\Theta_j(\tilde s_\mbp) \, \mrd \tilde s_\mbp .
\eeq

The dependence of the bilayer \muckymuck, the slow space $\mcZ_*$, and the local coordinates $(z_\mbp,s_\mbp)$ on $\mbp$ plays a key role in the analysis. In the next three subsections we analyse the slow projection of $\p_t \Phi_\mbp$, the residual $\Pi_0 \mrF(\Phi_\mbp)$, and the remainder $\Rem[v^\bot]$. This is conducted under the a priori assumptions 
\begin{equation}\label{A-0}
\mbp \in \cO_{2, \delta}\subset \cD_\delta \qquad\hbox{and}\qquad \|\dot \mbp\|_{l^2}\lesssim \varep^3.
\end{equation}
In section\,\ref{sec-main-thm} the assumption on the rate $\dot\mbp$ will be folded into the norm bounds on $\mbp.$



\subsection{Projection of $\p_t\Phi_\mbp$}
The projection of $\p_t \Phi_\mbp$ onto the meandering slow space $\mcZ_*^1$, is determined by the $N_1\times N_1$ matrix $\mathbb T$ whose $(k,j)$-th component is defined by
\begin{equation}\label{def-bT}
\mathbb T_{kj}:=\left<\frac{\p \Phi_{\mathbf p} }{\p\h{0.5pt} \mathrm p_j}, Z_{\mathbf p, *}^{1k}\right>_{L^2}, \qquad \qquad \hbox{for $k,j\in \Sigma_1$}. 
\end{equation}
This matrix approximates the first fundamental form of the bilayer manifold induced by graph of $\Phi_\mbp$ over the approximate tangent space $\mathcal Z_*^1.$
 The asymptotic form  of $\p_{\mrp_j} \Phi_\mbp$ follows from Lemmas \ref{lem-change-of-coord} and \ref{def-Phi-p}. This is presented below, with full justification found in Lemma $6.4$ of  \cite{CP-linear}. 
\begin{lemma}\label{lem-Phi_t-p} The bilayer \muckymuck $\Phi_\mbp$  given in Definition \ref{def-Phi-p-1} satisfies
\beqs
\frac{\p \Phi_\mbp}{\p \mrp_j} = \frac{1}{\varep} (\phi_0'+\varep \phi_1')\xi_j(s_\mbp) +\varep \mathrm R_j, 
\eeqs
where $\xi_j(s_\mbp)=\varep \frac{\p z_\mbp}{\p \mrp_j}$ depending on parameter $\mbp$ are given in Lemma \ref{lem-change-of-coord}; and the remainder $\mathrm R=(\mathrm R_j)$ lies in $L^2$ and it's projection to the meandering slow space satisfies the estimate
\beqs
\|\Pi_{\mathcal Z_*^1}\mathrm R\|_{L^2}\lesssim \varep^{1/2}.
\eeqs
 \end{lemma}
Introducing the canonical unit basis $\{\bB_k\}_{k\in \Sigma_1}$  of $\mathbb R^{N_1}$, from chain rule and \eqref{def-bT} we have
\begin{equation}\label{def-bTp'-ek}
\left<\mathbb T\dot{\mathbf p}, \bB_k\right>=\left<\p_t \Phi_\mbp, Z_{\mathbf p, *}^{1k}\right>_{L^2}\qquad\qquad \hbox{for} \;  k\in \Sigma_1.
\end{equation}
To motivate the next result we observe that, up to a multiplicative constant, the inner product on the right-hand side has the leading order
\beqs
\int_{\msI_\mbp}\p_t \bm \gamma_\mbp\cdot \mbn_\mbp \tilde \Theta_k \dd \tilde s_\mbp.
\eeqs 
From Lemma \ref{lem-change-of-coord} and chain rule we have the relation
\beq\label{GM-timede-proj}
\p_t \bm \gamma_\mbp \cdot \mbn_\mbp =  -\varep  \frac{\p z_\mbp}{\p t} = -\sum_{j\in\Sigma_1}\dot\mrp_j \xi_j(s_\mbp).  
\eeq 
The projection of $\xi_j(s_\mbp)$ to the Garlekin space $G_1$ involve the $(N_1-3) \times N_1$ matrix $\mbU$ 
with $(j,k)$-th component 
\beq\label{def-bU}
\mbU_{jk}:=\frac{1}{1+\mrp_0}\int_{\msI_\mbp} \tilde s_\mbp \tilde \Theta_j' \tilde \Theta_k \dd \tilde s_\mbp, \qquad j=3, 4, \cdots, N_1-1, \quad k\in \Sigma_1.
\eeq
In fact, we have the following approximation.
\begin{lemma}\label{lem-xi}
With $\xi_j(s_\mbp)$ defined in Lemma \ref{lem-change-of-coord}, it holds that for $j,k\in \Sigma_1$
\beqs 
-\int_{\msI_\mbp} \xi_j(s_\mbp)\tilde \Theta_k\dd \tilde s_\mbp=(1+\mrp_0)\left\{\begin{aligned}
& R_0/\Theta_0   +O(\|\hat\mbp\|_{\mbV_2^2})\mathbb E_{00}, & j=0, k=0;\ 
\\
& \mrp_k \bm 1_{\{k\geq 3\}}-\hat\mbp^T\mbU\mathbf B_k +O(\|\hat\mbp\|_{\mbV_2^2}^2)\mathbb E_{k0} &j=0,k\geq 1;\\
& \delta_{jk} + O(\|\hat\mbp\|_{\mbV_2^2}) \mbE_{kj} & j=1,2, k\in \Sigma_1;\\
& \delta_{jk} +O(\|\hat\mbp\|_{\mbV_2^2}^2) \mbE_{kj}, & j\geq 3, k\in \Sigma_1. 
\end{aligned}\right. 
\eeqs
Moreover when $j\in \Sigma_1, k\geq 3$, we have the following approximation for the weighted projection
\beqs 
-\int_{\msI_\mbp} \xi_j(s_\mbp)\beta_k\tilde \Theta_k\dd \tilde s_\mbp=(1+\mrp_0)\left\{\begin{aligned}
& \beta_k \mrp_k  -\hat\mbp^T\mbU\mathbf B_k +O(\|\hat\mbp\|_{\mbV_2^2}^2)\mathbb E_{k0} &j=0;\\
& \beta_k \delta_{jk} + O(\|\hat\mbp\|_{\mbV_2^2}) \mbE_{kj} & j=1,2;\\
&\beta_k \delta_{jk} +O(\|\hat\mbp\|_{\mbV_2^2}\|\hat\mbp\|_{\mbV_3^2}) \mbE_{kj}+O(\|\hat\mbp\|_{\mbV_2^2}^2)\beta_j, & j\geq 3. 
\end{aligned}\right. 
\eeqs
\end{lemma} 
The proof of the Lemma is technical and is postponed to the Appendix.

We introduce the translational vector
\begin{equation}\label{def-p'}
\mrp_{1,2}=\sum\limits_{k=1}^{2}\mrp_k\bB_k,
\end{equation}
the constants $m_2, m_3$
\beq
\label{def-m34}
m_2=\frac{1}{2} \int_{\mathbb R} \mrL_0^{-1}(z\phi_0') \dd z; \qquad m_3= \frac{1}{2} \int_{\mathbb R} |z\phi_0'|^2 \dd z,
\eeq
and the scalar function $\mu_0=\mu_0(\sigma, \mrp_0)$,
\begin{equation}\label{def-mu0}
\mu_0(\sigma, \mrp_0):=\frac{m_1^2}{m_1^2+\varep(\sigma m_2+\eta_d m_3^2)}\frac{1}{1+\mrp_0}=1+O(\mrp_0,\varep).
\end{equation}

\begin{lemma}\label{lem-def-bT} 
Under the a priori  assumptions \eqref{A-0},  there exists a unit vector $\bm e=(e_k)_{k\in \Sigma_1}$ for which  \eqref{def-bTp'-ek} admits the expansion: 
\begin{enumerate}
\item   if $k=0$, 
\begin{equation}
\begin{aligned}
-\frac{\varep^{1/2}\mu_0}{m_1} \left<\mathbb T \dot{\mbp}, \mathbf B_0\right>=&\frac{R_0}{\Theta_0}\dot{\mrp}_0+O\left((\varep^2+\|\hat\mbp\|_{\mbV_2^2})\|\dot{\mbp}\|_{l^2}\right).
\end{aligned}
\end{equation}
\item if $k\geq 1$,
\begin{equation}\label{est-bTp'-ek}
\begin{aligned}
-\frac{\varep^{1/2}\mu_0}{m_1} \left<\mathbb T \dot{\mbp}, \bB_k\right>=\dot{\mrp}_k 
+ \dot{\mrp}_0  \left(\mrp_k \bm 1_{\{k\geq 3\}}- \hat \mbp^T\mathbb U\bB_k\right) +O\left(\|\hat{\mbp}\|_{\mbV_2^2}|\dot \mrp_{1,2}|+(\varep^2+\|\hat\mbp\|_{\mbV_2^2}^2)\|\dot{\mbp}\|_{l^2}\right) e_k,
\end{aligned}
\end{equation}
where $\mathbb U$ is given in \eqref{def-bU} and $\mathbf 1_{\{k\geq j\}}(k)$ returns $1$ for $k\geq j$ and $0$ for $k<j.$
\end{enumerate}
\end{lemma}
\begin{proof} 
We have the algebraic relation
\beq\label{est-bT-0}
-\left<\mbT \dot\mbp, \mathbf B_k\right>_{l^2}=-\sum_{j\in \Sigma_1}\mbT_{kj}\dot\mrp_j.
\eeq 
We calculate $\mbT_{kj}$ for $k,j\in \Sigma$ and then complete the sum over $j.$ From the definition of $\phi_0, \phi_1$ in \eqref{def-phi0n2}, \eqref{def-phi1},  we have the following identity
\beq\label{est-bT-1}
\int_{\mathbb R} (\phi_0'+\varep\phi_1')\phi_0' \dd z
= m_1^2+\varep(\sigma m_2+\eta_d m_3^2), 
\eeq
 where $m_1, m_2, m_3$ are defined in \eqref{def-hatla1*} and \eqref{def-m34}.  Using the expressions for $\p_{\mrp_j}\Phi_\mbp$ from Lemma \ref{lem-Phi_t-p} and for $Z_{\mbp, *}^{1k}$ with $k\in \Sigma_1$ from \eqref{eq-def-Z*},
 the equality $\psi_1=\phi_0'/m_1$, the localization of the dressing of the eigenfunctions of $\mrL_0$, and the orthogonality \eqref{ortho-varphi} we find
 \beqs
 -\mathbb T_{kj}=-\frac{1}{\varep^{1/2}} \int_{\mbR_{2\ell}} \int_{\msI_\mbp} (\phi_0'(z_\mbp) +\varep\phi_1'(z_\mbp)) \frac{\phi_0'}{m_1}  \xi_j(s_\mbp) \tilde \Theta_k  \dd \tilde s_\mbp \dd z_\mbp +O(\varep^{3/2} \mbE_{kj}),
 \eeqs
 where the remainder matrix $\mathbb E_{kj}$ has bounded 
 $l^2_*$ norm by Lemma\,\ref{Notation-e_i,j}. We use \eqref{est-bT-1} to simplify 
  \beq\label{est-bT-2}
 -\mathbb T_{kj}=-\frac{m_1^2 +\varep(\sigma m_2+\eta_d m_3^2)}{\varep^{1/2}}  \int_{\msI_\mbp}  \xi_j(s_\mbp) \tilde \Theta_k   \dd \tilde s_\mbp +O(\varep^{3/2} \mbE_{kj}).
 \eeq
To complete the proof we assume $k\geq 1$ use the three cases on $j$ in Lemma \ref{lem-xi}  and return to the summation \eqref{est-bT-0}. The second identity of the Lemma follows after introducing $\mu_0$ as in \eqref{def-mu0}. The first identity, for $k=0$, follows similarly except that the terms involving $\mbU$ 
are small and placed in the remainder. 
\end{proof}

\begin{remark}\label{remark-timede-proj}
By the definition of $\mbT_{kj}$, the identities \eqref{est-bT-2} and \eqref{est-bT-0} in the proof of Lemma \ref{lem-def-bT} imply 
\beqs
\left<\p_t \Phi_\mbp, Z_{\mbp,*}^{1k}\right>_{L^2} = \frac{m_1}{\varep^{1/2}}\sum_{j\in \Sigma_1}\dot\mrp_j\int_{\msI_\mbp}\xi_j(s_\mbp)\tilde \Theta_k \dd \tilde s_\mbp+O(\varep^{1/2}\|\dot\mbp\|_{l^2})e_k.
\eeqs 
When combined with identity \eqref{GM-timede-proj} this implies
\beq
\label{NV-LB-proj}
 \left<\p_t \Phi_\mbp, Z_{\mbp,*}^{1k}\right>_{L^2} =- \frac{ m_1}{\varep^{1/2}} \int_{\msI_\mbp} \p_t\bm \gamma_\mbp \cdot \mbn_\mbp  \tilde \Theta_k\dd \tilde s_\mbp+O(\varep^{1/2}\|\dot\mbp\|_{l^2})e_k.
\eeq 
\end{remark}

\begin{cor}\label{cor-proj-Phi_t} With the same assumptions as  Lemma\,\ref{lem-def-bT}, we have the approximation under weighted projection
\beqs
\begin{aligned}
-\frac{\varep^{1/2} \mu_0}{m_1} \left<\mathbb T\dot \mbp, \beta_k\mathbf B_k\right>=&\beta_k \dot\mrp_k  +\dot\mrp_0 \left(\beta_k \mrp_k - \hat\mbp^T \mbU  \beta_k \mathbf B_k \right) +O(\|\hat\mbp\|_{\mbV_2^2}  \|\dot \mrp_{1,2}\|_{l^2})e_k \\
& O((\varep^2\beta_k+\|\hat\mbp\|_{\mbV_2^2}\|\hat\mbp\|_{\mbV_3^2}) \|\dot \mbp\|_{l^2}) e_k+O(\|\hat\mbp\|_{\mbV_2^2}^2 \|\dot\mbp\|_{\mbV_1^2})e_k.
\end{aligned}
\eeqs
for $k\geq 3$.
\end{cor}
\begin{proof}
Following the proof of Lemma\,\ref{lem-def-bT}, we estimate $\beta_k \mbT_{kj}$. From  \eqref{est-bT-2}, we have 
\beqs
\beta_k \mathbb T_{kj}=-\frac{m_1^2 +\varep(\sigma m_2+\eta_d m_3^2)}{\varep^{1/2}}  \int_{\msI_\mbp}  \beta_k \xi_j(s_\mbp) \tilde \Theta_k   \dd \tilde s_\mbp +O(\varep^{3/2} \beta_k \mbE_{kj}).
\eeqs
The corollary follows from Lemma \ref{lem-xi}.
\end{proof}

\begin{lemma}\label{cor-bU} The matrix  $\mathbb U=(\mathbb U_{lk})$  with components $l,k = 3, \cdots N_1-1$ defined in \eqref{def-bU} satisfies 
\beq\label{est-bU-1}
\begin{aligned}
  \|\mbU^T\|_{l_2^*} \lesssim  \|\mbD^{1/2}\|_{l_2^*}.
\end{aligned}
\eeq

\end{lemma}
\begin{proof}
With $\hat\mbp=(\mrp_3, \cdots, \mrp_{N_1-1})\in l^2$, $\bar p=\sum_{l=3}^{N_1-1} \mrp_l \tilde \Theta(\tilde s_\mbp)$, 
and $\mbU_{lk}$ as defined we have
\beqs
\left<\mbU^T\hat\mbp, \bB_k\right>_{l^2}=\sum_{k=3}^{N_1-1} \mbU_{lk} \mrp_l= \frac{1}{1+\mrp_0} \int_{\msI_\mbp} \bar p'(\tilde s_\mbp) \tilde \Theta_k \dd \tilde s_\mbp.
\eeqs
From \eqref{re-barp-hatp} $\|\bar p'\|_{L^2(\msI_\mbp)}\lesssim\|\hat\mbp\|_{\mbV_1^2}= \|\mbD^{1/2}\hat\mbp\|_{l^2}$, and the result follows from Lemma \ref{Notation-e_i,j}.
\end{proof}

\subsection{Projection of the residual}
 The projection of the bilayer \muckymuck residual $\mrF(\Phi_\mbp)$, given in Lemma \ref{lem-def-Phi-p}, to the meandering slow space $\mcZ_*^1$ drives the dynamics of meandering parameters $\mathbf p$. The following Lemma breaks down this projection. At leading order the projection arises from the normal velocity $V_\mbp$ from \eqref{normal-velocity} onto the cut-off Garlekin space $G_1$. To this end we introduce
 \beq\label{def-sAB}
 \begin{aligned}
 \mathscr A_k(\mbp)&:= \int_{\msI_\mbp} \kappa_\mbp \tilde \Theta_k \dd \tilde s_\mbp,\\
 \mathscr B_k(\mbp)&:=\int_{\msI_\mbp} \left(-\Delta_{s_\mbp} \kappa_\mbp -\frac{\kappa_\mbp^3}{2}+\alpha \kappa_\mbp \right)\tilde \Theta_k  \dd \tilde s_\mbp,
 \end{aligned}
 \eeq
 which give the dominant components of this projection.  
\begin{lemma}\label{lem-proj-F} Under the a priori assumptions \eqref{A-0}, there exists a unit vector  $\bm  e=(e_k)_{k\in \Sigma_1}$  such that 
\begin{enumerate}
    \item for $k=0$, 
    \beqs
\begin{aligned}
\left<\Pi_0 \mathrm F(\Phi_{\mathbf p} ),  Z_{\mathbf p, *}^{1k} \right>_{L^2}=&\varep^{5/2}m_1 \left(\mathscr C(\mrp_0)-\frac{c_0}{\Theta_0}\mathrm p_0\right) 
+O\left(\varep^{5/2}|\mrp_0-\mrp_0^*|\|\hat{\mbp}\|_{\mbV_2^2}\right)\\
&+O\left(\varep^{7/2} \|\hat\mbp\|_{\mbV_2^2}\|\hat\mbp\|_{\mbV_3^2}, \varep^{9/2} \|\hat{\mbp}\|_{\mbV_4^2}\right)
\end{aligned}
\eeqs
where $\mathscr C(\mrp_0)$ is a smooth function of $\mrp_0$ and has the form 
\begin{equation*}
\begin{aligned}
\mathscr C(\mrp_0):=&-\frac{2\pi \Theta_0 m_0}{m_1^2}  \left(\sigma_1^*-\sigma_0 \right)+ 2\pi \Theta_0 \varep   \bigg[\frac{1}{2R_0^2(1+\mrp_0)^2}-\alpha (\sigma_1^*)\bigg]\\
&+\varep \mcC_1(\mrp_0) \mrp_0+\varep \mcC_2(\mrp_0)(\mrp_0-\mrp_0^*)+\varep^2\mcC_3(\mrp_0).
\end{aligned}
\end{equation*}
Here $\mcC_k(\mrp_0)$ for $k=1,2,3$ are smooth functions of $\mrp_0$ with uniform bounds independent of $\varep$ whose form is not material to the subsequent analysis.
\item for $k\geq 1$ and  $k\in \Sigma_1$, 
\beqs
\begin{aligned}
\qquad \left<\Pi_0 \mathrm F(\Phi_{\mathbf p} ), Z_{\mathbf p, *}^{1k} \right>_{L^2}=& -\varep^{5/2}m_1 c_k(\mrp_0)\mathrm p_k\bm 1_{\{k\geq 3\}}
+O(\varep^{7/2} |\mrp_0-\mrp_0^*|\|\hat\mbp\|_{\mbV_2^2}, \varep^{5/2}|\mrp_0-\mrp_0^*|\|\hat\mbp\|_{\mbV_2^2}^2 )e_k\\
&+O(\varep^{7/2}\|\hat\mbp\|_{\mbV_2^2} \|\hat\mbp\|_{\mbV_4^2}, \varep^{9/2}\|\hat\mbp\|_{\mbV_4^2})e_k.
\end{aligned}
\eeqs
Here $c_k$ depends on  time only through $\mrp_0$ and has the form 
\begin{equation*}
\begin{aligned}
c_k(\mrp_0):=&\frac{\beta_k^2-1}{R_0^2}\Bigg[c_0 R_0(\mrp_0-\mrp_0^*) +O(\varep |\mrp_0-\mrp_0^*|^2)  +\varep \left(- \frac{m_0}{m_1^2} \sigma_2^*+\frac{2\beta_k^2-3}{2R_0^2(1+\mrp_0 )^2} +\alpha (\sigma_1^*)  \right)  \Bigg].
\end{aligned}
\end{equation*}
\end{enumerate}
\end{lemma}
\begin{proof} Adding and subtracting the far-field value of the residual and using the definition of $\Pi_0$ and decomposition of $Z_{\mbp,*}^{1k}$,  we break the projection in dominant and remainder terms 
\begin{equation}\label{decomp-PiF-Z*1}
\left<\Pi_0\mathrm F(\Phi_{\mbp} ), Z_{\mbp, *}^{1k}\right>_{L^2}=\mathcal I_k+\mathscr R_{k}, 
\end{equation}
in which the dominant term $\mathcal I_k$ and remainder $\msR_k:=\mathscr R_{k, 1}+ \mathscr R_{k,2}$ are given by
\begin{equation}\label{def-cI-sR}
\begin{aligned}
\mathcal I_k&:=\int_\Omega \left(\mathrm F(\Phi_{\mbp} )-\mrF_m^\infty \right)Z_{\mbp}^{1k}\dd   x,\\
\mathscr R_{k,1}&:=\varep \int_\Omega \left(\mathrm F(\Phi_{\mbp} )- \mrF_m^\infty \right)(\tilde \varphi_{1,k} \tilde \Theta_k +\tilde \varphi_{2,k} \varep\tilde \Theta_k')\dd   x,\\
\mathscr R_{k,2}&:= \frac{1}{|\Omega|}\int_\Omega \left(\mathrm F(\Phi_{\mbp} )-\mrF_m^\infty\right)\dd x\int_\Omega Z_{\mbp, *}^{1k}\dd x.
\end{aligned}
\end{equation}
To streamline the presentation, the approximation of remainder terms $\mathscr R_{k,0}, \msR_{k,1}$ is given in the appendix.  
To approximate the dominant term $\mcI_k$ we 
replace $Z_{\mbp}^{1k}$ by its definition \eqref{def-Zp},  
and replace $\psi_1(z_\mbp)$ with $\phi_0'(z_\mbp)/m_1$. Rewriting $\mcI_k$ in the local coordinates we obtain
 \beqs
 \mcI_k= \frac{\varep^{1/2}}{m_1}\int_{\mbR_{2\ell}}\int_{\msI_\mbp} (\mrF(\Phi_\mbp) -\mrF_m^\infty) \phi_0'  \tilde \Theta_k  (1-\varep z_\mbp \kappa_\mbp) \dd \tilde s_\mbp \dd z_\mbp.
 \eeqs
Using the expansion of $\mrF(\Phi_\mbp)$ from Lemma \ref{lem-def-Phi-p}, we decompose 
$ \mcI_k=\sum_{j=1}^3\mcI_{k,j}$, where
\beqs
\begin{aligned}
\mcI_{k,1} & =\frac{\varep^{5/2}}{m_1}  \int_{\mbR_{2\ell}} \int_{\msI_\mbp} \left(\mrF_2-\mrF_2^\infty\right) \phi_0' \tilde \Theta_k(1-\varep z_\mbp\kappa_\mbp) \dd \tilde s_\mbp \dd z_\mbp, \\
\mcI_{k,2} & =\frac{\varep^{7/2}}{m_1}  \int_{\mbR_{2\ell}} \int_{\msI_\mbp} \left(\mrF_3-\mrF_3^\infty\right) \phi_0' \tilde \Theta_k(1-\varep z_\mbp\kappa_\mbp) \dd \tilde s_\mbp \dd z_\mbp,\\
\mcI_{k,3} & =\frac{\varep^{9/2}}{m_1}  \int_{\mbR_{2\ell}} \int_{\msI_\mbp} \left(\mrF_{\geq 4}-\mrF_{\geq 4}^\infty +e^{-\ell \nu/\varep}\varep^{-4}\mrF_e\right) \phi_0' \tilde \Theta_k(1-\varep z_\mbp\kappa_\mbp) \dd \tilde s_\mbp \dd z_\mbp.
\end{aligned}
\eeqs
We deal with these terms one by one. First, from Lemma \ref{lem-def-Phi-p} $\mrF_2$ is localized with far-field $\mrF_2^\infty=0$, while both $\phi_0'$  and $\mrF_2$  have 
odd parity in $z_\mbp$, we deduce
\beqs
\mcI_{k,1} = \frac{\varep^{5/2}}{m_1} \int_{\mbR_{2\ell}} \int_{\msI_\mbp}  \mrF_2 \phi_0' \tilde \Theta_k \dd \tilde s_\mbp\dd z_\mbp.
\eeqs
Changing the integrating order and using \eqref{est-proj-rF2,3} yields the reduction
\beq\label{est-cI-k,1}
\mcI_{k,1}= \frac{\varep^{5/2}m_0}{m_1}(\sigma_1^*-\sigma) \msA_k(\mbp),
\eeq
where  $\msA_k$  was introduced in \eqref{def-sAB}. For $\mcI_{k,2}$, $\mrF_3$ is given \eqref{F-234} with projection \eqref{est-proj-rF2,3}, we have 
\beqs
\begin{aligned}
\mcI_{k,2} &= \varep^{7/2} m_1 \msB_k + \frac{\varep^{9/2}}{m_1} \int_{\mbR_{2\ell}}\int_{\msI_\mbp} (\mrF_3-\mrF_3^\infty) \phi_0' \tilde \Theta_k z_\mbp \kappa_\mbp \dd \tilde s_\mbp \dd z_\mbp\\
&= \varep^{7/2} m_1 \msB_k + \frac{\varep^{9/2}}{m_1} \int_{\mbR_{2\ell}}\int_{\msI_\mbp} (\phi_0'\Delta_{s_\mbp} \kappa_\mbp +f_3-f_3^\infty) \phi_0'  \tilde \Theta_k z_\mbp \kappa_\mbp \dd \tilde s_\mbp \dd z_\mbp
\end{aligned}
\eeqs
for $f_3=f_3(z_\mbp, \bm \gamma_\mbp'')$. Since $|\phi_0'|^2 z_\mbp$ has odd parity it does not contribute to the integral, and we deduce
\beq\label{est-cI-k,2-1}
\mcI_{k,2} = \varep^{7/2} m_1 \msB_k + \frac{\varep^{9/2}}{m_1} \int_{\mbR_{2\ell}}\int_{\msI_\mbp} (f_3-f_3^\infty) \phi_0' \tilde \Theta_k z_\mbp \kappa_\mbp \dd \tilde s_\mbp \dd z_\mbp.
\eeq 
Addressing the integral, we change the order of integration order and integrate in $z_\mbp$, leaving  a  function $h=h(\bm \gamma_\mbp'')$ that enjoys the properties of Notation\,\ref{Notation-h}  since both $f_3$ and $\kappa_\mbp$ depend only on second and lower derivatives of $\bm \gamma_\mbp$.
With this notation we deduce
\beqs
 \frac{\varep^{9/2}}{m_1} \int_{\mbR_{2\ell}}\int_{\msI_\mbp} (f_3-f_3^\infty) \phi_0' \tilde \Theta_k z_\mbp \kappa_\mbp \dd \tilde s_\mbp \dd z_\mbp =\frac{\varep^{9/2}}{m_1} \int_{\msI_\mbp} h(\bm \gamma_\mbp'') \tilde \Theta_k \dd \tilde s_\mbp.
\eeqs 
Applying Lemma \ref{lem-est-h-proj} to bound the right-hand side of this identity and returning the estimate to $\mcI_{k,2}$ in \eqref{est-cI-k,2-1} yields the asymptotic form
\beq\label{est-cI-k,2}
\mcI_{k,2}=\varep^{7/2}m_1 \mathscr B_k +\varep^{9/2} C(\mrp_0)\delta_{k0} +O(\varep^{9/2}\|\hat\mbp\|_{\mbV_2^2})e_k.
\eeq  
Similar estimates applied to $\mcI_{k,3}$ yield the expansion 
\beq
\label{est-cI-k,3}
\mcI_{k,3}= \varep^{9/2}C(\mrp_0) \delta_{k0}+O(\varep^{9/2} \|\hat\mbp \|_{\mbV_4^2})e_k.
\eeq
Combining  the estimates \eqref{est-cI-k,1},  \eqref{est-cI-k,2}, and \eqref{est-cI-k,3} we find the dominant expansion
\beq
\mcI_{k}=\frac{\varep^{5/2}m_0}{m_1}(\sigma_1^*-\sigma) \msA_k(\mbp) +\varep^{7/2}m_1 \msB_k(\mbp) +\varep^{9/2}C(\mrp_0) \delta_{k0}+O(\varep^{9/2} \|\hat\mbp \|_{\mbV_4^2})e_k.
\eeq
Rearranging and including the error estimates on  $\msR_{k,1}, \msR_{k,2}$ from Lemma \ref{lem-sR12}, we reduce \eqref{decomp-PiF-Z*1} to
\beq\label{est-proj-F-Z1-1}
\begin{aligned}
\left<\Pi_0 \mrF(\Phi_\mbp), Z_{\mbp, *}^{1k}\right>_{L^2}=&\frac{\varep^{5/2}m_0}{m_1}(\sigma_1^*-\sigma) \left( \msA_k(\mbp) +\varep  C_1(\mrp_0)\delta_{k0} +O(\varep\|\hat\mbp\|_{\mbV_2^2}) \right) \\
&+\varep^{7/2}m_1 \msB_k(\mbp)+\varep^{9/2}C_2(\mrp_0) \delta_{k0}+O(\varep^{9/2} \|\hat\mbp \|_{\mbV_4^2})e_k.
\end{aligned}
\eeq
where $C_1, C_2$ are smooth functions of $\mrp_0$. 
To reduce these expressions to their final form we consider the cases $k=0$ and $k\neq 0$ separately. For $k=0$, using the form of $\msA_0,\msB_0$ presented in Corollary \ref{cor-sAB} of the appendix,  we rewrite the expansion above as
 \beqs
\begin{aligned}
\left<\Pi_0 \mrF(\Phi_\mbp), Z_{\mbp, *}^{10}\right>_{L^2}=& \frac{\varep^{5/2}m_0}{m_1} (2\pi \Theta_0+\varep C_1(\mrp_0))(\sigma -\sigma_1^*)\\
&+\varep^{7/2}m_1 \pi\Theta_0\left(\frac{1}{R_0^2(1+\mrp_0)^2}-\alpha\right)+\varep^{9/2}C(\mrp_0)\\
&+ O(\varep^{9/2}\|\hat\mbp\|_{\mbV_4^2}, \varep^{7/2}\|\hat\mbp\|_{\mbV_2^2}\|\hat\mbp\|_{\mbV_3^2},  \varep^{5/2}|\sigma-\sigma^*|\|\hat\mbp\|_{\mbV_2^2}).
\end{aligned}
\eeqs
Here we expanded $\sigma_1^*=\sigma^*+O(\varep)$ to simplify the error term.  Since the coefficient  $\alpha=\alpha(\sigma)$ it smooth it affords the expansion
\beq\label{est-alpha}
\begin{aligned}
\alpha(\sigma)=&\alpha(\sigma_1^*) + O(\varep, |\sigma-\sigma^*|).
\end{aligned}
\eeq 
Using the first expansion of Corollary \ref{cor-def-hatla} and deducing that $|\sigma-\sigma^*|\lesssim |\mrp_0-\mrp_0^*|+\varep^2\|\hat\mbp\|_{\mbV_2^2}$ from the second expansion,  we expand
\beq\label{est-proj-F-Z1-k=0}
\begin{aligned}
\left<\Pi_0 \mrF(\Phi_\mbp), Z_{\mbp, *}^{10}\right>_{L^2}=\varep^{5/2} m_1\left( \mathscr C (\mrp_0) -\frac{c_0}{\Theta_0}\mrp_0\right)\h{100pt} \\
+ O(\varep^{9/2}\|\hat\mbp\|_{\mbV_4^2}, \varep^{7/2}\|\hat\mbp\|_{\mbV_2^2}\|\hat\mbp\|_{\mbV_3^2}, \varep^{5/2}|\mrp_0-\mrp_0^*|\|\hat\mbp\|_{\mbV_2^2}),
\end{aligned}
\eeq
where $\mathscr C=\mathscr C(\mrp_0)$ is as in part (1) of Lemma\,\ref{lem-proj-F}.
For the case $k\neq 0$,  we replace $\sigma_1^*$ with $(\sigma^*-\varep\sigma_{\geq 2}^*)$, 
and reduce \eqref{est-proj-F-Z1-1} to
\beqs
\begin{aligned}
\left<\Pi_0 \mrF(\Phi_\mbp), Z_{\mbp, *}^{1k}\right>_{L^2}= & \frac{\varep^{5/2} m_0}{m_1} (\sigma^*-\sigma)  \msA_k +\varep^{7/2} \left(\msB_k-\frac{m_0}{m_1^2}\sigma_{\geq 2}^* \msA_k \right)\\
&+ O(\varep^{7/2} |\sigma^*-\sigma|\|\hat\mbp\|_{\mbV_2^2}, \varep^{9/2}\|\hat\mbp\|_{\mbV_4^2})e_k
\end{aligned}
\eeqs
Using this expansion, together with the expansions of $\msA_k, \msB_k$ from Corollary \ref{cor-sAB}, and the bound $|\sigma-\sigma^*|\lesssim |\mrp_0-\mrp_0^*|+\varep^2 \|\hat\mbp\|_{\mbV_2^2}$  and the expansion of $\alpha$ from \eqref{est-alpha} yields the result (2) of Lemma\,\ref{lem-proj-F} with $c_k=c_k(\mrp_0)$ as defined therein. 
\end{proof}
Collecting the results on the meander projection of $\p_t\Phi_\mbp$ and of the residual $\mrF(\Phi_\mbp)$ we can bound the Galerkin projections of the difference between the normal velocity of $\Phi_\mbp$ and the curvature induced velocity given in \eqref{normal-velocity}.
\begin{cor}\label{cor-residual-proj}
There exists a unit vector $\bm e\in l^2(\mbR^{N_1})$ such that
the Galerkin projections of the normal and curvature induced velocity,\eqref{normal-velocity}, satisfy
\beq
\label{Vn-Vp-resid}
\int_{\msI_\mbp} \left(\p_t\bm \gamma_\mbp \cdot \mbn_\mbp - V_\mbp\right) \tilde \Theta_k \dd \tilde s_\mbp = 
-\frac{\varep^{1/2}}{m_1} \left< \Rem[v^\bot], Z_{\mbp,*}^{1k}\right>_{L^2} + O(\varep^{4}|\sigma_1^*-\sigma|,\varep^{5}, \varep^{5}\|\hat\mbp\|_{\mbV^2_4},\varep\|\dot\mbp\|_{l^2})e_k,
\eeq
for $k\in\Sigma_1.$
\end{cor}
\begin{proof}
From \eqref{est-proj-F-Z1-1} we have 
\beqs
\begin{aligned}
\left<\Pi_0\mrF(\Phi_\mbp), Z_{\mbp,*}^{1k}\right>_{L^2}=& \frac{\varep^{5/2}m_0}{m_1} (\sigma_1^*-\sigma)\msA_k(\mbp) +\varep^{7/2} m_1\msB_k(\mbp)+O(\varep^{7/2} (\sigma_1^*-\sigma), \varep^{9/2})\delta_{k0} \\
&+O(\varep^{7/2}\|\hat\mbp\|_{\mbV_2^2}, \varep^{9/2}\|\hat\mbp\|_{\mbV_4^2})e_k
\end{aligned}
\eeqs 
with $\msA_k,\msB_k$ defined in \eqref{def-sAB}.
The result follows by combining this with \eqref{NV-LB-proj}, multiplying by $-\varep^{1/2}/m_1$,
and using the formulation \eqref{eq-v-Phi} of the gradient flow to replace $\p_t\Phi_\mbp+\Pi_0\mrF(\Phi_\mbp)$ with $\Rem[v^\bot].$
\end{proof}
The meandering parameter $\mrp_0=\mrp_0(t)$ controls the length of the interface $\Gamma_\mbp$. Its equilibrium value $\mrp_0^*$ is determined by
\beq\label{def-p0*-original}
\left<\Pi_0\mrF(\Phi_\mbp), Z_{\mbp, *}^{1k}\right>_{L^2} \Big|_{\mbp=\mbp^*}=0 
\eeq 
for $\mbp^*=(\mrp_0^*, \mrp_1^*, \mrp_2^*, \bm 0)$. 
\begin{lemma}\label{lem-p0*}
Suppose $|\sigma_0(\Gamma_0, M_0)-\sigma_1^*|\lesssim 1$, then there exists $\varep_0$ small enough such that for each $\varep\in(0,\varep_0)$ $\mrp_0^*$ is well defined through \eqref{def-p0*-original} and admits the approximation  $\mrp_0^*=\mrp_{0,0}^*+\varep \mrp_{0,1}^*+O(\varep^2)$ with
\begin{equation}\label{def-p*0}
\begin{aligned}
\mathrm p_{0,0}^*&=-\frac{m_0 }{c_0 R_0  m_1^2} \left(\sigma_1^*-\sigma_0(\Gamma_0, M_0) \right);\\
\mathrm p_{0,1}^*&=\frac{1}{c_0 R_0} \left(\frac{1}{2R_0^2(1+\mrp_{0,0}^*)^2}-\alpha (\sigma_1^*)
\right) +\frac{\Theta_0}{c_0}\mathcal C_1(\mrp_{0,0}^*)\mrp_{0,0}^*.
\end{aligned}
\end{equation}
\end{lemma}
\begin{proof}From statement (1) of lemma \ref{lem-proj-F} and \eqref{def-p0*-original}, $\mrp_0^*$ solves
\begin{equation}\label{def-p0*}
c_0\mathrm  p_0/\Theta_0=\mathscr C(\mrp_0).
\end{equation} 
Here $\mathscr C$ defined in part (1) of Lemma\,\ref{lem-proj-F} can be decomposed as $\mathscr C(\mrp_0)=\mathscr C_0+ \varep\mathscr C_1(\mrp_0)$  with
\begin{equation}
\begin{aligned}
\mathscr C_0&:=-\frac{m_0}{m_1^2} 2\pi \Theta_0  \left(\sigma_1^*-\sigma_0(\Gamma_0,M_0)\right),\\
\mathscr C_1(\mrp_0)&:=2\pi \Theta_0 \left(\frac{1}{2R_0^2(1+\mrp_0)^2}-\alpha (\sigma_1^*) \right)+\mcC_1(\mrp_0)\mrp_0+ \mcC_2(\mrp_0)(\mrp_0-\mrp_0^*)+\varep \mcC_3(\mrp_0),
\end{aligned}
\end{equation}
and the base bulk density $\sigma_0$ defined in \eqref{def-c0}.

 The system is linear in  $\mrp_0$ for $\varep=0$, and hence has a unique solution. The smooth continuation of this unique solution, for $\varep\in (0,\varep_0)$ with $\varep_0$ small enough, is a simple application of the implicit function theorem since $\mathscr C_1$ is smooth in $\mrp_0.$ Noting $\Theta_0=1/\sqrt{2\pi R_0}$. A regular perturbation expansion $\mathrm p_0^*=\mathrm p_{0,0}^*+\varep\mathrm p_{0, 1}^{*}+O(\varep^2)$, shows that
\eqref{def-p*0}.
\end{proof}

From Lemma \ref{lem-sigma},  the equilibrium bulk density $\sigma^*=\sigma(\mbp^*)$ with $\mbp^*=(\mrp_0^*, \mrp_1^*, \mrp_2^*, \bm 0)$ depends only on the radial meander parameter $\mrp_0^*$ through
\begin{equation}
\sigma^*=\sigma_0
-\frac{c_0 R_0 m_1^2}{m_0} \mrp_0^* +\varep \mrp_0^*\mcC(\mrp_0^*),
\end{equation} 
With approximation of $\mrp_0^*$ given in Lemma \ref{lem-p0*} , we derive $\sigma^*=\sigma_1^*+\varep\sigma_2^*+O(\varep^2)$ with
\begin{equation}\label{def-hatla*2}
\sigma_2^*=-\frac{c_0\Theta_0 m_1^2}{m_0} \mathrm p_{0,1}^*=\frac{m_1^2}{m_0}\left(\frac{-1}{2R_0^2(1+\mathrm p_{0,0}^*)^2}+\alpha (\sigma_1^*) + \mrp_{0,0}^*\mathcal C(\mrp_{0,0}^*)\right).
\end{equation}
These relations give the map from the system mass to a unique, up to translation, equilibrium profile with parameters $(\sigma,\mrp_0, \hat\mbp)= (\sigma^*, \mrp_0^*, \bm 0)$. 

Returning to \eqref{def-p0*} we may expanding
$\mathscr C(\mathrm p_0)$ around $\mathrm p_0^*$ to write the residual in the form  
\begin{equation}
\mathscr C (\mathrm p_0)-\frac{c_0}{\Theta_0}\mathrm p_0= \frac{c_0}{\Theta_0}(\mathrm p_0^*-\mathrm p_0)+O(\varep)(\mrp_0^*-\mrp_0).
  \end{equation}
 This allows us to reformulate the projection of the residual $\Pi_0\mathrm F(\Phi_{\mbp} )$ onto $Z_{\mathbf p, *}^{10}$, given in part (1) of  Lemma \ref{lem-proj-F} in terms of $\mrp_0-\mrp_0^*$ and $\|\hat\mbp\|_{\mbV_2^2}$,
\begin{equation}\label{proj-F2-0}
\begin{aligned}
\int_\Omega \Pi_0\mathrm F(\Phi_{\mathbf p})Z_{\mathbf p, *}^{10}\dd   x=& \varep^{5/2} \frac{m_1 c_0}{\Theta_0} ( \mathrm p_0^*-\mathrm p_0) +O(\varep^{7/2}|\mrp_0-\mrp_0^*|)\\
&+
O\left(\varep^{5/2} |\mrp_0-\mrp_0^*|\|\hat\mbp\|_{\mbV_2^2},\varep^{7/2}\|\hat{\mathbf p}\|_{\mbV_2^2}\|\hat\mbp\|_{\mbV_3^2}, \varep^{9/2}\|\hat{\mbp}\|_{\mbV_4^2}\right).
\end{aligned}
\end{equation}
Using \eqref{def-hatla*2} to eliminate $\sigma_2^*$  and using the a priori assumption $|\mrp_0(t)|\ll 1$ for all $t\in \mathbb R^+$, we rewrite  $c_k$  as,  
\begin{equation}\label{def-ck-1}
c_k(\mrp_0)= \frac{c_0   }{R_0} (\mbD_{kk}-1)(\mrp_0-\mrp_0^*)   +\varep \frac{(\mbD_{kk}-1)^2}{R_0^4 }  +O(\varep|\mrp_0|\mbD_{kk}),  
\end{equation}
where $\mbD$ is the diagonal matrix defined in \eqref{def-mbD}
that induces a norm on $\mbV_{2}^2$.

We estimate weighted norms of the projections by absorbing the factors of $\beta_k$ into higher $\mbV_j^2$ norms of $\mbp$, which impacts the statement of the error terms.
\begin{cor}\label{cor-proj-F-w} Let $k\geq 3$, the residual projected to the weighted space  obey 
\beqs
\begin{aligned}
\int_{\Omega}\Pi_0 \mathrm F(\Phi_{\mathbf p} ) \beta_k Z_{\mathbf p, *}^{1k} \dd   x= -\varep^{5/2}m_1 c_k\beta_k \mathrm p_k + O(\varep^{5/2}|\mrp_0-\mrp_0^*|\|\hat\mbp\|_{\mbV_2^2}\|\hat\mbp\|_{\mbV_3^2} )e_k\\
+O(\varep^{7/2} |\mrp_0-\mrp_0^*|\|\hat\mbp\|_{\mbV_3^2}, \varep^{7/2}\|\hat\mbp\|_{\mbV_3^2}\|\hat\mbp\|_{\mbV_4^2},\varep^{7/2}\|\hat\mbp\|_{\mbV_2^2} \|\hat\mbp\|_{\mbV_5^2}, \varep^{9/2}\|\hat\mbp\|_{\mbV_5^2})e_k
\end{aligned}
\eeqs
\end{cor}
\begin{proof}
Multiplying \eqref{decomp-PiF-Z*1} by $\beta_k$ we have
\beqs
\left<\Pi_0 \mrF(\Phi_\mbp), Z_{\mbp, *}^{1k}\right>_{L^2}=\beta_k \mcI_k +\beta_k \msR_k,
\eeqs
in which $\beta_k \msR_k=\beta_k \msR_{k,1} +\beta_k \msR_{k,2}$ are given in Lemma \ref{lem-sR12-w}. We focus on $\beta_k\mcI_k$ and its expansion $\beta_k\mcI_{k,1}+\beta_k\mcI_{k,2} +\beta_k \mcI_{k,3}$. Combining with \eqref{est-cI-k,1} and \eqref{est-cI-k,2-1}, we have 
\beqs
\begin{aligned}
\beta_k \mcI_k= &\frac{\varep^{5/2} m_0}{m_1}(\sigma_1^*-\sigma)\beta_k \msA_k(\mbp) +\varep^{7/2} m_1\beta_k\msB_k(\mbp) \\
&+ \frac{\varep^{9/2}}{m_1} \int_{\mbR_{2\ell}} \int_{\msI_\mbp} (f_3-f_3^\infty) \phi_0' \beta_k \tilde  \Theta_k z_\mbp \kappa_\mbp \dd\tilde s_\mbp \dd z_\mbp +\beta_k \mcI_{k,3}.
\end{aligned}
\eeqs
Performing the integration in $z_\mbp$ and recalling the form of $f_3=f_3(z_\mbp, \bm \gamma_\mbp'')$, 
Lemma \ref{lem-h-proj-w} and \eqref{est-h-gamma''-de} with $l=1$
allow us to rewrite and estimate the integral on the second line of the equality above as
\beqs
\varep^{9/2} \int_{\msI_\mbp} h(\bm \gamma_\mbp'') \beta_k\tilde \Theta_k \dd \tilde s_\mbp=O(\varep^{9/2}\|\hat\mbp\|_{\mbV_3^2})e_k
\eeqs
Estimation  of $\beta_k \mcI_{k,3}$ can be arrived at by similar arguments that use the form of $\mrF_{\geq 4}$ and $\mrF_e$  which collectively contribute an error of order $O(\varep^{9/2}\|\hat\mbp\|_{\mbV_5^2})$. 
The remainder of the reduction follows the lines of the proof of Lemma \ref{lem-proj-F}, using Lemmas \ref{lem-sAB-w} and \ref{lem-sR12-w} from the appendix. We  omit the details. 
\end{proof}

\subsection{Projection of the orthogonal remainder $\Rem[v^\bot]$}
We estimate the projection and weighted projection of the orthogonal remainder $\Rem[v^\bot]$ in the flow \eqref{eq-v-Phi} induced by the orthogonal perturbation $v^\bot$ that lifts the solution $u$ off of the bilayer manifold $\cM_b$. This requires an estimate of impact of the flow on the meander basis functions. 
By the definition \eqref{eq-def-Z*} of $Z_{\mbp,*}^{ij}$, Lemma \ref{Notation-e_i,j}, we have the bounds
\begin{equation}\label{est-L2-p-tZ*}
\|\p_t Z_{\mathbf p, *}^{I(j)j}\|_{L^2} \lesssim \varep^{-1}\|\dot{\mathbf p}\|_{l^2} 
, \qquad \forall j\in \Sigma.
\end{equation}
A full justification is provided in Lemma $6.5$ of \cite{CP-linear}. 
\begin{lemma}\label{lem-est-Rv} Under the assumption \eqref{A-0}, there exists a unit vector ${\bm e} = (e_k)_{k=0}^{N_1-1}$ such that the projection of the orthogonal remainder $\Rem[v^\bot]$, defined in \eqref{def-Rv}, to the meandering space satisfies the bound
 \beqs
 \left<\Rem[v^\bot], Z_{\mbp, *}^{1k}\right>= O\left(\varep^{-3/2} \|v^\bot\|_{L^2}\|\dot\mbp\|_{l^2}, (\varep^{2} +\varep^2\|\hat\mbp\|_{\mbV_4^2}) \|v^\bot\|_{L^2}, \|\mathrm N(v^\bot)\|_{L^2} \right)e_k,
 \eeqs
 and the weighted estimate
 \beqs
 \left<\Rem[v^\bot], \beta_k Z_{\mbp, *}^{1k}\right>= O\left(\varep^{-5/2} \|v^\bot\|_{L^2} \|\dot\mbp\|, (\varep^{3/2} \|\hat\mbp\|_{\mbV_5^2} +\varep^{1/2})  \|v^\bot\|_{L^2}, \varep^{-1}\|\mathrm N(v^\bot)\|_{L^2} \right)e_k
 \eeqs
 for $k\geq 3$ and $k\in \Sigma_1$.
\end{lemma}
\begin{proof}
We break the orthogonal remainder into its three component terms, and rewrite the projection of $\p_t v^\bot$ as
\beqs
\left<\p_t v^\bot, Z_{\mbp, *}^{1k}\right>_{L^2} = \p_t \left<v^\bot, Z_{\mbp,*}^{1k}\right>_{L^2} + \left<v^\bot, \p_t  Z_{\mbp,*}^{1k}\right>_{L^2}.
\eeqs
The first item on the right-hand side is zero since $v^\bot$ is perpendicular to the meandering slow space $\mcZ_{ *}^1$; and the second term is bounded  through  H\"older's inequality and estimate \eqref{est-L2-p-tZ*}. Combining these, we deduce 
\beq
\left<\p_t v^\bot, Z_{\mbp, *}^{1k}\right>_{L^2} = O(\varep^{-1} \|\dot\mbp\|_{l^2}\|v^\bot\|_{L^2}),
\eeq
which together with the dimension $N_1\leq \varep^{-1}$ of $\mcZ_{*}^1$ and approximate orthogonality \eqref{ortho-Z*} we deduce
\beqs
\|\Pi_{\mcZ_*^1} \p_t v^\bot\|_{L^2}\lesssim \varep^{-1}N_1^{1/2}\|\dot\mbp\|_{l^2}\|v^\bot\|_{L^2}\lesssim \varep^{-3/2}\|\dot\mbp\|_{l^2}\|v^\bot\|_{L^2}.
\eeqs 

For the second term in the orthogonal remainder, $\Pi_0 \mbL_\mbp v^\bot$, we turn to Lemma \ref{thm-coupling est} to bound its projection. In fact, this yields the estimate
\beq
\left\|\Pi_{\mcZ_*^1}\Pi_0 \mbL_\mbp v^\bot\right\| \lesssim (\varep^2 + \varep^2 \|\hat\mbp\|_{\mbV_{4}^2}) \|v^\bot\|_{L^2}.
\eeq
The third term in the orthogonal residual is the nonlinearity $\mathrm N(v^\bot)$ which has the simple bound
\beq
\|\Pi_{Z_*^1} \Pi_0 \mathrm N(v^\bot)\|_{L^2}\lesssim \|\mathrm N(v^\bot)\|_{L^2}.
\eeq
The unweighted estimate of Lemma\,\ref{lem-est-Rv} follows. For the weighted estimate, the term $\beta_k$ is uniformly bounded by $\varep^{-1}$, so that unweighted estimate implies
 \beqs
 \left<\mathscr R[v^\bot], \beta_k Z_{\mbp, *}^{1k}\right>= O\left(\varep^{-2} \|v^\bot\|_{L^2} \|\dot\mbp\|, (\varep +\varep\|\hat\mbp\|_{\mbV_4^2})  \|v^\bot\|_{L^2}, \varep^{-1}\|\mathrm N(v^\bot)\|_{L^2} \right)e_k.
 \eeqs 
 From the definition of $\mbV_k^2$, and Young's inequality we have the embedding estimate
 \beqs
 \|\hat\mbp\|_{\mbV_4^2}\leq \|\hat\mbp\|_{\mbV_3^2}^{1/2}\|\hat\mbp\|_{\mbV_5^2}^{1/2}\lesssim \varep^{1/2}\|\hat\mbp\|_{\mbV_5^2} +\varep^{-1/2}\|\hat\mbp\|_{\mbV_3^2}.
 \eeqs
 The weighted estimate follows by the a priori bound on $\|\hat\mbp\|_{\mbV_3^2}$ for $\mbp\in \mathcal O_{2,\delta}$.
\end{proof}

\subsection{Dynamics of meandering parameters $\mbp$}
By combining the results of Lemmas \ref{lem-def-bT}, \ref{lem-proj-F}  and \ref{lem-est-Rv} we obtain a detailed reduction of the dynamics of the meandering parameters $\mbp$ induced by the gradient flow.  For simplicity of presentation, we introduce two time-dependent functions
\beq\label{def-E12} 
\begin{aligned}
E(t):=\varep + \|\hat\mbp\|_{\mbV_2^2} +|\mrp_0|+\varep^{-3/2}\|v^\bot\|_{L^2}; \qquad E_w(t):=\varep+  \|\hat\mbp\|_{\mbV_3^2} +|\mrp_0|+\varep^{-2}\|v^\bot\|_{L^2}.
\end{aligned}
\eeq 
Obviously, $E(t)\leq E_w(t)$ for any $t\in \mbR^+$. Our assumptions make these two quantities small in $L^\infty(\mbR^+)$, as required in Theorem \ref{thm-est-p} and Corollary \ref{cor-ODE-pw}, respectively.  

\begin{thm}\label{thm-est-p} 
In addition to the a priori assumptions \eqref{A-0} suppose $v^\bot$ satisfies $\|v^\bot\|_{L^2}\ll \varep^{3/2}$ such that $E(t)$ is suitably small uniformly in time, then the meandering parameters $\mathbf p=(\mrp_0, \mrp_{1,2}, \hat\mbp)$ evolve according to
\beq
\label{eq-hatp}
\begin{aligned}
&\dot\mrp_0=- \varep^3\frac{c_0}{R_0}(\mrp_0-\mrp_0^*)+d_0,\\
&\dot\mrp_k=O\left(\varep^3|\mrp_0-\mrp_0^*|\|\hat\mbp\|_{\mbV_3^2}\right)+ d_k \quad \hbox{for} \quad k=1,2,\\
&\dot{\hat\mbp}=-\varep^3  \left[ \frac{c_0 }{R_0}(\mbD + \mbU^T)(\mrp_0-\mrp_0^*)+\frac{\varep}{R_0^4} (\mbD-\mbI)^2 \right]\hat\mbp  +\hat d.
\end{aligned}
\eeq
 Here $\mbI$ is the $(N_1-3)$ by $(N_1-3)$ identity matrix, and $\mbU$ is defined in \eqref{def-bU}.
The vector $ d=(d_0,  d')$ with $d'=(d_1, d_2, \hat d)$ denotes error terms that satisfy
\beqs
\begin{aligned}
|d_0| \lesssim  &\; \varep^3E(t) |\mrp_0 -\mrp_0^*|   + \varep^4E(t)\|\hat\mbp\|_{\mbV_4^2} + \varep^{5/2} \|v^\bot\|_{L^2} +\varep^{1/2}\|\mathrm N(v^\bot)\|_{L^2};\\
\|d'\|_{l^2}\lesssim &\; \varep^3E(t) |\mrp_0-\mrp_0^*|\|\hat\mbp\|_{\mbV_2^2}  + \varep^4(\varep+\varep^{-2}\|v^\bot\|_{L^2})|\mrp_0-\mrp_0^*|+\varep^4E(t)\|\hat\mbp\|_{\mbV_4^2} \\
&+\varep^{5/2}\|v^\bot\|_{L^2} +\varep^{1/2} \|\mathrm N(v^\bot)\|_{L^2}.
\end{aligned}
\eeqs
Moreover the rate of change $\dot{\mathbf p}=(\dot\mrp_0, \dot\mrp_{1,2}, \dot{\hat\mbp})$ admits the upper bounds:
\beq\label{est-dotp-1}
\begin{aligned}
\|\dot\mbp\|_{l^2} 
\lesssim \varep^3 |\mrp_0-\mrp_0^*| +\varep^4 \|\hat\mbp\|_{\mbV_4^2} +|d_0|+\|d'\|_{l^2}.
\end{aligned}
\eeq
\end{thm}
\begin{proof} 
Projecting equation \eqref{eq-v-Phi} onto $Z_{\mathbf p, *}^{1k}$ in $L^2$, using the definition \eqref{def-bT} of $\mathbb T$ yields
\begin{equation}\label{ODE-p-1}
\begin{aligned}
\left<\mathbb T\dot{\mathbf p}, \bB_k\right>_{l^2}=&-\left<\Pi_0 \mathrm F(\Phi_{\mathbf p} ), Z_{\mathbf p, *}^{1k}\right>_{L^2}+\left<\mathscr R[v^\bot], Z_{\mathbf p, *}^{1k},\right>_{L^2},
\end{aligned}
\end{equation}
where $\Rem[v^\bot]$ defined in \eqref{def-Rv}  is the orthogonal remainder contributed by $v^\bot$. We first address the case $k=0$, multiplying the identity \eqref{ODE-p-1} with $k=0$ by  $-\varep^{1/2} \frac{\mu_0}{m_1} \cdot\frac{ \Theta_0 }{R_0}$ and applying Lemmas \ref{lem-def-bT} and identity \eqref{proj-F2-0} yields the ODE
\beq\label{ODE-p0-1}
\dot \mrp_0=\varep^3 \frac{ c_0 }{R_0} (\mrp_0-\mrp_0^*)+d_0.
\eeq 
The remainder $d_0$ satisfies the bound
\beqs
\begin{aligned}
|d_0|\lesssim & \varep^3\left( \varep + \|\hat\mbp\|_{\mbV_2^2} +|\mu_0-1| \right) |\mrp_0 -\mrp_0^*| + (\varep^2 +\|\hat\mbp\|_{\mbV_2^2})\|\dot\mbp\|_{l^2}  +\varep^4 \|\hat\mbp\|_{\mbV_2^2}\|\hat\mbp\|_{\mbV_3^2} +\varep^5\|\hat\mbp\|_{\mbV_4^2}\\
&+\varep^{1/2}\left|\left<\Rem[v^\bot], Z_{\mathbf p, *}^{1k}\right>_{L^2}\right|.
\end{aligned}
\eeqs
The quantity $\mu_0=1+O(|\mrp_0|,\varep)$ was introduced in \eqref{def-mu0} and  the projection of the orthogonal remainder was bounded in Lemma \ref{lem-est-Rv}. With these estimates and  introduction of $E(t)$ in \eqref{def-E12} we simplify the bound of $d_0$,
\beq\label{est-d0-1}
\begin{aligned}
|d_0| \lesssim  & \varep^3E(t) |\mrp_0 -\mrp_0^*| + E(t)\|\dot\mbp\|_{l^2}  
+\varep^4E(t)\|\hat\mbp\|_{\mbV_4^2}+ \varep^{5/2}  \|v^\bot\|_{L^2} +\varep^{1/2}\|\mathrm N(v^\bot)\|_{L^2}.
\end{aligned}
\eeq
For $k\geq 1$, we multiply identity \eqref{ODE-p-1} by $-\varep^{1/2}\mu_0/m_1$ and apply Lemmas \ref{lem-def-bT} and \ref{lem-proj-F} for $k\geq 1$ to deduce
\beq\label{ODE-pk-1}
\dot\mrp_k   = -\dot\mrp_0 \left( \mrp_k \bm 1_{\{k\geq 3\}} -\hat \mbp^T \mbU \bB_k \right) -\varep^3 \mu_0 c_k \mrp_k +d_k. 
\eeq
The remainder $d'=(d_k)_{k=1}^{N_1-1}$ can be bounded 
in $l^2$ by collecting those remainders in Lemmas \ref{lem-def-bT} and \ref{lem-proj-F} and applying Lemma \ref{lem-est-Rv},
\beq\label{est-d'-1}
\begin{aligned}
\|d'\|_{l^2}\lesssim &\varep^{3}E(t) |\mrp_0 -\mrp_0^*| \|\hat\mbp\|_{\mbV_2^2} +\varep^4E(t)\|\hat\mbp\|_{\mbV_4^2}  +\Big(\varep^2+\|\hat\mbp\|_{\mbV_2^2}^2  +\varep^{-1}\|v^\bot\|_{L^2}\Big)\|\dot{\mbp}\|_{l^2}\\
&+ \|\hat\mbp\|_{\mbV_2^2}\|\dot\mrp_{1,2}\|_{l^2} +\varep^{5/2}\|v^\bot\|_{L^2} +\varep^{1/2} \|\mathrm N(v^\bot)\|_{L^2}.
\end{aligned}
\eeq
Replacing $\dot\mrp_0$ on the right-hand side of \ref{ODE-pk-1} with the right-hand side of \eqref{ODE-p0-1}, using the expansion $\mu_0=1+O(|\mrp_0|,\varep)$, and replacing $c_k$ with \eqref{def-ck-1} yields 
\beq\label{ODE-pk-2}
\begin{aligned}
\dot\mrp_k   =& - \frac{c_0}{R_0}(\mrp_0 -\mrp_0^*) \left( \mrp_k \bm 1_{\{k\geq 3\}} -\hat \mbp^T \mbU \bB_k \right) -\varep^3  c_k \mrp_k +\tilde d_k\\
=& -\varep^3   \left[\frac{c_0}{R_0} (\mrp_0-\mrp_0^*)\mbD_{kk} +\frac{\varep}{R_0^4} (\mbD_{kk}-1)^2 \right]\mrp_k \bm 1_{\{k\geq 3\}}+\frac{ c_0}{R_0}(\mrp_0-\mrp_0^*) \hat\mbp \mbU\bB_k +\tilde d_k.
\end{aligned}
\eeq
We have introduced $\tilde d_k=d_k +O(\|\hat\mbp\|_{\mbV_2^2} d_0, \varep^3 (|\mrp_0|+\varep )|\mrp_0-\mrp_0^*|\|\hat\mbp\|_{\mbV_2^2}, \varep^4(|\mrp_0|+\varep)\|\hat\mbp\|_{\mbV_4^2}) e_k$. In the remainder of the proof we omit the title on $d_k$. From \eqref{est-d'-1} the revised definition of $d'=(d_k)_{k=1}^{N_1-1}$ enjoys the $l^2$-bound 
\beq\label{est-d'-2}
\begin{aligned}
\|d'\|_{l^2}\lesssim &\;\varep^3E(t) |\mrp_0-\mrp_0^*|\|\hat\mbp\|_{\mbV_2^2} +\|\hat\mbp\|_{\mbV_2^2} |d_0| +\varep^4E(t)\|\hat\mbp\|_{\mbV_4^2}  +\Big(\varep^2+\|\hat\mbp\|_{\mbV_2^2}^2  \\
&+\varep^{-1}\|v^\bot\|_{L^2}\Big)\|\dot{\mbp}\|_{l^2} + \|\hat\mbp\|_{\mbV_2^2}\|\dot\mrp_{1,2}\|_{l^2} +\varep^{5/2}\|v^\bot\|_{L^2} +\varep^{1/2} \|\mathrm N(v^\bot)\|_{L^2}.
\end{aligned}
\eeq
The bounds \eqref{est-d0-1} and \eqref{est-d'-2} are not yet closed since they depend on $\dot \mbp$, however  from \eqref{eq-hatp}  we derive
\begin{equation}\label{Lem59-pdot} 
\begin{aligned} |\dot{\mrp}_0|\lesssim  \varep^3|\mrp_0-\mrp_0^*|+|d_0|, \qquad  \|\dot {\hat \mbp} \|_{l^2}\lesssim  \varep^3 |\mrp_0-\mrp_0^*|\|\hat\mbp\|_{\mbV_2^2}+\varep^4\|\hat\mbp\|_{\mbV_4^2} +\|d'\|_{l^2}, \\|\dot\mrp_{1,2}|\lesssim \varep^3 |\mrp_0-\mrp_0^*|\|\hat\mbp\|_{\mbV_2^2}+\|d'\|_{l^2}.\h{1.4in}
\end{aligned}
\end{equation}
Then the estimates \eqref{est-dotp-1} follow directly. 
 Using this bound \eqref{est-dotp-1} and the bound on $\dot\mrp_{1,2}$ from \eqref{Lem59-pdot} on the right-hand side of \eqref{est-d0-1} and \eqref{est-d'-1},  and remarking that $\varep^2+\varep^{-1}\|v^\bot\|_{L^2}+\|\hat\mbp\|_{\mbV_2^2}\ll 1$, then some algebraic rearrangement  leads to the bounds
\beqs
\begin{aligned}
|d_0| \lesssim  &\; \varep^3  E(t)  |\mrp_0 -\mrp_0^*| + E(t)\|d'\|_{l^2}+ \varep^4 E(t)\|\hat\mbp\|_{\mbV_4^2} + \varep^{5/2} \|v^\bot\|_{L^2} +\varep^{1/2}\|\mathrm N(v^\bot)\|_{L^2};\\
\|d'\|_{l^2}\lesssim &\; \varep^3E(t) |\mrp_0-\mrp_0^*|\|\hat\mbp\|_{\mbV_2^2}  + (\varep^5+\varep^{2}\|v^\bot\|_{L^2})|\mrp_0-\mrp_0^*|+\varep^4E(t)\|\hat\mbp\|_{\mbV_4^2}   +  E(t) d_0\\
&+\varep^{5/2}\|v^\bot\|_{L^2} +\varep^{1/2} \|\mathrm N(v^\bot)\|_{L^2}.
\end{aligned}
\eeqs
Using the estimate on $\|d'\|_{l^2}$ to eliminate it from the right-hand side of the estimate on $|d_0|$ yields the final upper bound for $d_0$.  The final $l^2$-estimate for $d'$ follows from the estimate above and the a priori assumptions \eqref{A-0}. \end{proof}

\begin{cor} \label{cor-ODE-pw} If $\|v^\bot\|_{L^2}\ll \varep^2$ and assumption \eqref{A-0} holds so that $E_w(t)$ introduced in \eqref{def-E12} is suitably small uniformly in time, then the weighted meander evolution can be written in the form
\beqs
\mbD^{1/2} \dot{\hat\mbp }   =-\varep^3  \left[ \frac{c_0 }{R_0}(\mbD + \mbU^T)(\mrp_0-\mrp_0^*)+\frac{\varep}{R_0^4} (\mbD-\mbI)^2 \right]\mathbb D^{1/2}\hat\mbp  + \hat d_w, 
\eeqs
where the weighted remainder $\hat d_w=(d_{w,k})_{k=3}^{N_1-1}$ satisfies 
\beqs
\begin{aligned}
\| \hat d_w\|_{l^2} \lesssim & \varep^3E_w(t) |\mrp_0-\mrp_0^*|\|\hat\mbp\|_{\mbV_3^2} +\varep^4|\mrp_0-\mrp_0^*|  +E_w(t)|\hat\mbp\|_{\mbV_5^2}  +\varep \|v^\bot\|_{L^2} +\varep^{-1/2}\|\mathrm N(v^\bot)\|_{L^2}.
\end{aligned}
\eeqs
\end{cor}
\begin{proof}
Multiplying equation \eqref{ODE-p-1} by the weight $\beta_k$ we have
\beqs
\left<\mathbb T\dot{\mathbf p}, \beta_k \bB_k\right>_{l^2}=-\left<\Pi_0 \mathrm F(\Phi_{\mathbf p} ), \beta_k Z_{\mathbf p, *}^{1k}\right>_{L^2}-\left<\Rem[v^\bot], \beta_k Z_{\mathbf p, *}^{1k}\right>_{L^2}.
\eeqs 
Multiplying this result by $-\varep^{-1/2}\mu_0/m_1$ and  applying Corollaries \ref{cor-proj-Phi_t} and \ref{cor-proj-F-w}  yields
\beq\label{ODE-pk-w}
\beta_k \dot \mrp_k    =-\dot{\mrp}_0\left(\beta_k  \mrp_k -\beta_k  \hat\mbp^T \mbU \bB_k \right) -\varep^3\mu_0 c_k \beta_k \mrp_k  +  d_{w,k},
\eeq 
where  from Lemma \ref{lem-est-Rv} the weighted remainder $\hat d_w=(d_{w,k})$ can be bounded as
\beqs
\begin{aligned}
\|\hat d_w\|_{l^2} \lesssim  & \varep^3\left(\|\hat\mbp\|_{\mbV_2^2} +\varep \right) |\mrp_0-\mrp_0^*|\|\hat\mbp\|_{\mbV_3^2} +\varep^4 E_w(t)\|\hat\mbp\|_{\mbV_5^2}  + \|\hat\mbp\|_{\mbV_2^2} \|\dot\mrp_{1,2}\|_{l^2} \\
&+(\varep +\|\hat\mbp\|_{\mbV_2^2}\|\hat\mbp\|_{\mbV_3^2}+\varep^{-2}\|v^\bot\|_{L^2})\|\dot\mbp\|_{l^2}  +\|\hat\mbp\|_{\mbV_2^2}^2 \|\dot \mbp\|_{\mbV_1^2} \\
&+\varep  \|v^\bot\|_{L^2} +\varep^{-1/2}\|\mathrm N(v^\bot)\|_{L^2}.
\end{aligned}
\eeqs
As in the unweighted case, we use the first equation of \eqref{eq-hatp} to substitute for $\dot\mrp_0$ in the right-hand side of \eqref{ODE-pk-w}, replace $c_k$ with its definition \eqref{def-ck-1}, and recall that $\mu_0=1+O(\varep, |\mrp_0|)$. These manipulations yield the equation  
\beqs
\beta_k \dot\mrp_k =-\varep^3  \left[\frac{c_0}{R_0} (\mrp_0-\mrp_0^*)\mbD_{kk} +\frac{\varep}{R_0^4} (\mbD_{kk}-1)^2 \right]\beta_k\mrp_k \bm 1_{\{k\geq 3\}}+\frac{ c_0}{R_0}(\mrp_0-\mrp_0^*) \beta_k\hat\mbp \mbU\bB_k +\tilde d_{w,k}, 
\eeqs
where from Lemma \ref{cor-bU} the revised remainder takes the form 
$$\tilde d_{w,k}=d_{w,k} +O(\|\hat\mbp\|_{\mbV_2^2}d_0, \varep^3E_w |\mrp_0-\mrp_0^*|\|\hat\mbp\|_{\mbV_3^2}, \varep^4E_w\|\hat\mbp\|_{\mbV_5^2}).$$ 
We drop the tilde, and the weighted evolution presented in the Lemma follows. The revised from of the remainder $\hat d_w$ satisfies the bound
\beq\label{est-d-w-2}
\begin{aligned}
\|\hat d_w\|_{l^2} \lesssim  & \varep^3E_w(t) |\mrp_0-\mrp_0^*|\|\hat\mbp\|_{\mbV_3^2} +\|\hat\mbp\|_{\mbV_2^2}|d_0|   +\varep^4E_w(t)\|\hat\mbp\|_{\mbV_5^2}+ \|\hat\mbp\|_{\mbV_2^2} \|\dot\mrp_{1,2}\|_{l^2} \\
&+(\varep +\|\hat\mbp\|_{\mbV_2^2}\|\hat\mbp\|_{\mbV_3^2}+\varep^{-2}\|v^\bot\|_{L^2})\|\dot\mbp\|_{l^2} +\|\hat\mbp\|_{\mbV_2^2}^2 \|\dot \mbp\|_{\mbV_1^2} +\varep \|v^\bot\|_{L^2} +\varep^{-1/2}\|\mathrm N(v^\bot)\|_{L^2}.
\end{aligned}
\eeq
Since $\dot\mbp=(\dot\mrp_0, \dot\mrp_{1,2}, \dot\hat\mbp)$, from the evolution system in the Lemma, the definition of $\mbV_1^2$, and the  bounds on $\dot \mrp_0, \dot\mrp_{1,2}$ from Theorem \ref{thm-est-p}  we find
\beqs
\begin{aligned}
\|\dot\mbp\|_{\mbV_1^2} & \lesssim  |\dot\mrp_0| +|\dot\mrp_{1,2}|+ \|\mbD^{1/2}\dot{\hat\mbp}\|_{l^2}\\
&\lesssim  \varep^3|\mrp_0-\mrp_0^*| +|d_0|+\|d'\|_{l^2} +\varep^4\|\mbp\|_{\mbV_5^2}+\|\hat d_w\|_{l^2}.
\end{aligned}
\eeqs
Using this estimate and the bounds on $d_0, \|d'\|_{l^2}, \|\dot\mbp\|_{l^2}$  from Theorem \ref{thm-est-p}  in the right-hand side of \eqref{est-d-w-2}, we obtain the desired bound on $\hat d_w$.

\end{proof}

\begin{lemma}\label{lem-mbD} The diagonal matrices $\mbD$ and $\mbD-\mbI$ are uniformly comparable as maps from $l^2$ to $l^2$, in particular, 
\beqs
\frac{1}{2}\|\mbD\|_{l_2^*} \leq \|\mbD-\mbI\|_{l_2^*}\leq \|\mbD\|_{l_2^*}.
\eeqs
\end{lemma}
\begin{proof}
Since $\mbD$ and $\mbD-\mbI$ are both diagonal, we only need compare their diagonal elements. Indeed, their $k$-th diagonal term are $\beta_k^2$ and $\beta_k^2-1$, for $\mbD$ and $\mbD-\mbI$ respectively.   The Lemma follows directly from the relationship
\beqs
\frac{1}{2}\beta_k^2 \leq \beta_k^2-1\leq \beta_k^2
\eeqs
since $\beta_k^2>2$ for $k\geq 3$, see \eqref{def-Theta-beta}. 
\end{proof}

\section{Nonlinear stability and normal velocity}\label{sec-main-thm}
We close the nonlinear estimates which establish the stability of the bilayer manifold defined around the circular equilibrium, and derive the normal velocity of the interfaces $\Gamma_\mbp$ that corresponds to the meander evolution. A key to the analysis is that the 
operator $\Pi_0\mbL_\mbp$ is uniformly coercive on the space $L^2$-orthogonal to the slow space $\mathcal Z_*(\mbp)$. However the slow space includes the pearling space $\mathcal Z_*^0(\mbp)$ on which the operator is weakly coercive.
This dichotomy motivates a further decomposition of the orthogonal perturbation $ v^\bot$ in pearling and fast modes as
\begin{equation}\label{decomp-v}
 v^\bot=Q(  x, t)+w(  x, t), 
 \quad w\in \mathcal Z_*^\bot(\mbp, \rho),
\end{equation}
where $Q=\Pi_{\mcZ_*^0}v^\bot\in \mathcal Z_{*}^0$ admits the form
\beqs
Q=\sum_{j\in \Sigma_0} \mrq_j Z_{\mbp, *}^{0j},
\eeqs 
for some $\mbq=\mbq(t)=(\mrq_i)_{i\in \Sigma_0}$. The decomposition is well posed by Lemma 4.10 of \cite{CP-linear} which follows from the almost-orthogonality \eqref{ortho-Z*}. Moreover, this Lemma establishes the bounds
\beqs
\|Q\|_{\Hin}+\|w\|_{\Hin}\lesssim \|v^\bot\|_{\Hin}; \qquad \|Q\|_{\Hin}\sim \|\mbq\|_{l^2}. 
\eeqs 
\subsection{Control  of $\mbq,w$}
The decomposition \eqref{decomp-v} of $ v^\bot$, allows us to write \eqref{eq-v-Phi} as an evolution for  $Q$ 
\begin{equation}\label{eq-Q}
\p_t Q+\Pi_0\mbL_{\mathbf p}  Q=-\p_t \Phi_{\mathbf p}  -\p_t w-\Pi_0 \mathrm F(\Phi_{\mathbf p} )-\Pi_0 \mbL_{\mathbf p}  w-\Pi_0 \mathrm N( v^\bot).
\end{equation}
Taking the $L^2$ projection onto $Q$ and onto $\p_t Q$  we obtain the $l^2$ estimates on the evolution of the pearling parameters $\mbq$.

\begin{lemma}\label{lem-est-q} Under  assumptions \eqref{A-0}, if the pearling stability condition $(\mathbf{PSC})$ \eqref{cond-P-stab} holds uniformly independent of $\varep,\rho$, then there exists $C>0$ independent of $\varep$ such that the pearling parameters $\mathbf q=(\mathrm q_k(t))_{k\in \Sigma_0}$ obey
\begin{equation*}
\begin{aligned}
&\p_t\|\mathbf q\|_{l^2}^2+C\varep\|\mathbf q\|_{l^2}^2 \lesssim \varep \rho^{-4}\|\mbL_\mbp w\|_{L^2}^2+\varep^{-1}\|\mathrm N(v^\bot)\|_{L^2}^2+\varep^8\|\hat{\mathbf p}\|_{\mbV_4^2}^2.
\end{aligned}
\end{equation*}
Moreover, the $l^2$-norm of the time derivative $\dot\mbq$ can be bounded by
\beqs
\|\dot\mbq\|_{l^2}^2\lesssim \|\mbq\|_{L^2}^2 + \varep^2\|w\|_{L^2}^2+\varep\|\dot\mbp\|_{l^2}^2+\|\mathrm N(v^\bot)\|_{L^2}^2+\varep^{9}\|\hat\mbp\|_{\mbV_4^2}^2.
\eeqs
\end{lemma}
\begin{proof}
The proof is a simplification of that of Lemma 5.4 of \cite{CP-linear} since the base interface $\Gamma_0$ is a circle, so the contribution from the geometric quantities of the base interface $\Gamma_0$ are zero. 
\end{proof}

The equation \eqref{eq-Q} can be rewritten as an equation of the  fast modes $w$, 
\begin{equation}\label{eq-w}
\p_t w+\Pi_0 \mbL_{\mathbf p}  w=-\p_t \Phi_{\mathbf p} -\p_t Q-\Pi_0 \mathrm F(\Phi_{\mathbf p} ) -\Pi_0\mbL_{\mathbf p}  Q-\Pi_0\mathrm N( v^\bot). 
\end{equation}
By taking $L^2$-inner product with $\mbL_\mbp w$, we develop two $H^2$ estimates of $w$ by dealing with the residual differently. These estimates have utility on  different time scales. 
\begin{lemma}\label{lem-est-w} 
Under the a priori assumptions \eqref{A-0}, the function $w\in\mathcal Z_*^\bot$, obeys
\beqs
\frac{\dd}{\dd t}\left<\mbL_\mbp w,w\right>_{L^2}+\frac{1}{2}\|\mbL_\mbp w\|_{L^2}^2 \lesssim \varep^{-1}\|\dot\mbp\|_{l^2}^2 + \varep^2\rho^{-4}\|\mbq\|_{l^2}^2 +\varep^5|\mrp_0-\mrp_0^*|^2 +\varep^7(1+\|\hat\mbp\|_{\mbV_4^2}^2)+\|\mathrm N(v^\bot)\|_{L^2}^2; 
\eeqs
The fast modes also obey the second estimate
\begin{equation}\label{est-w-1}
\begin{aligned}
\frac{\mrd }{\mrd  t}\left<\mbL_{\mathbf p}  w, w\right>_{L^2} +\frac{1}{2}\|\mbL_{\mathbf p}  w\|_{L^2}^2\lesssim& \varep^{-1}\|\dot{\mathbf p}\|_{l^2}^2+\varep^2\rho^{-4} \|\mathbf q\|_{l^2}^2+\varep^5(|\mathrm p_0-\mathrm p_0^*|^2+\|\hat{\mathbf p}\|_{\mbV_2^2}^2)\\
&+\varep^7\|\hat{\mathbf p}\|_{\mbV_4^2}^2+\|\mathrm N( v^\bot)\|_{L^2}^2.
\end{aligned}
\end{equation}
\end{lemma}
\begin{proof}
The first estimate is derived in Lemma 5.3 of \cite{CP-linear}, we address the second estimate. Taking the $L^2$-inner product of \eqref{eq-w} with 
$\mbL_\mbp w$  we estimate each term as in the first case, except for the residual $\mrF(\Phi_\mbp)$. This yields the bound
\beqs
\begin{aligned}
\frac{\dd}{\dd t} \left<\mbL_\mbp w, w\right>_{L^2} + \|\mbL_\mbp w\|_{L^2}^2\lesssim & \left(\varep^{1/2}\|\dot\mbp\|_{l^2}+\varep\rho^{-2}(\|\mbq\|_{l^2} +\|\dot\mbq\|_{l^2}) +\|\mathrm N\|_{L^2}\right)\|\mbL_\mbp w\|_{L^2}\\
&- 2\left<\Pi_0 \mrF(\Phi_\mbp), \mbL_\mbp w\right>_{L^2}.
\end{aligned}
\eeqs
Applying H\"older's and Young's inequality this reduces to the estimate
\beqs
\begin{aligned}
\frac{\dd}{\dd t} \left<\mbL_\mbp w, w\right>_{L^2} + \|\mbL_\mbp w\|_{L^2}^2\lesssim & \varep \|\dot\mbp\|_{l^2}^2 +\varep^2\rho^{-4}(\|\mbq\|_{l^2}^2 +\|\dot\mbq\|_{l^2}^2) +\|\mathrm N\|_{L^2}^2+\|\Pi_0 \mrF(\Phi_\mbp)\|_{L^2}^2.
\end{aligned}
\eeqs
The  second estimate on $w$ follows from the $L^2$-bound on the residual $\Pi_0\mrF(\Phi_\mbp)$ given in Lemma\, \ref{lem-est-L2-F},  and the $l^2$-bound on $\dot\mbq$ in Lemma \,\ref{lem-est-q}.
\end{proof}


The Lemmas \ref{lem-est-w}, \ref{lem-est-q} and Theorem \ref{thm-est-p}, incorporate $L^2$-bounds of the nonlinear term $\mathrm N(v^\bot)$. This quantity, and the $L^\infty$ norm of the orthogonal perturbation are bounded in terms of the fast and the pearling modes in Lemma 5.9 of \cite{CP-linear}, which we quote below for completeness.

\begin{lemma}\label{lem-est-N} If $\|v^\bot\|_{L^\infty(\Omega)}$ is bounded independent of $\varep$, then
\begin{equation}
\label{eq-Nest}
\|\mathrm N( v^\bot)\|_{L^2}\lesssim\varep^{-1}\Big(\rho^{-2}\left<\mbL_{\mathbf p}  w, w\right>_{L^2} +\|\mathbf q(t)\|_{l^2}^2\Big),
\end{equation}
Moreover, if $v^\bot=w+Q$ as in \eqref{decomp-v} then it admits the upper bound
\begin{equation*}
\| v^\bot\|_{L^\infty}\lesssim \varep^{-1}\left(\rho^{-1}\left<\mbL_{\mathbf p}  w, w\right>_{L^2}^{1/2} +\|\mathbf q(t)\|_{l^2}\right).
\end{equation*}
\end{lemma}
The following result, Theorem 5.13 of \cite{CP-linear}, establishes that the bilayer manifold around a general admissible pair $(\Gamma_0, M_0)$, see \eqref{adm-pair},
is stable up to its boundary. Specifically, orbits that start within a thin tubular neighborhood of $\cM_b$ say within a comparable tubular neighborhood until the meander parameters $\mbp$ reach the boundary of the domain $\cO_\delta.$

\begin{thm} {\bf [Theorem 5.13 \cite{CP-linear}]} \label{thm-main-1}
 Consider the  mass preserving flow \eqref{eq-FCH-L2-p}-\eqref{eq-FCH-L2} subject to periodic boundary condition and initial data $u_0\in \mcV_{\varep^{5/2}}(\cM_b(\Gamma_0, M_0), \cO_\delta)$. Let  $(\Gamma_0, M_0)$ be a admissible pair that satisfies  
\beqs
|\sigma_0(|\Gamma_0|, M_0) -\sigma_1^*| \lesssim  \delta, 
\eeqs  
where $\sigma_1^*=\sigma_1^*(\eta_1, \eta_2)$, introduced in \eqref{def-hatla1*}, satisfies  the  pearling stability condition $(\mathbf{PSC}_*)$, \eqref{PSC*}. 
  Then for $\delta$ small enough, there exists $\varep_0$ small enough such that for all $\varep\in(0,\varep_0)$  the solution $u$ lies in the projection valid domain $\mathcal U(\mathcal M_b)$ so long as $\mbp\in \mathcal O_{2,\delta}$.  And the  following statement holds on the residence interval $[0,T]$. 
 \begin{enumerate}
 \item  The solution  of the mass preserving flow \eqref{eq-FCH-L2} can be decomposed as 
 \beq\label{decomp-u-MT}
 u(x)=\Phi_\mbp(\bm x; \sigma) +v^\bot(x), \qquad v^\bot(x)=Q(x)+w(x)
 \eeq
 where $\Phi_\mbp\in \cM_b(\Gamma_0, M_0;\rho)$ 
 and  $Q$ is the projection of $v^\bot$ to the pearling slow space, i.e. $Q=\Pi_{\mcZ_*^0}v^\bot \in \mcZ_*^0$. 
 \item Moreover,  the orthogonal perturbation  $v^\bot$ remains in $\mcV_{C\varep^{5/2}\rho^{-2}}$ for some positive constant $C$, that is 
\beq\label{est-p-w}
\|v^\bot\|_{\Hin}\lesssim \|w\|_{\Hin}(t)+\|\mbq\|_{l^2}(t)\leq C\varep^{5/2}\rho^{-2}, \qquad \forall t\in [0,T].
\eeq
\end{enumerate}
\end{thm}

\subsection{Nonlinear stability}
For a circular base interface $\Gamma_0$, we extend the result of Theorem\,\ref{thm-main-1} to show that the time of residence $T$ of orbits $u=u(t)$ that start in the thin tubular neighborhood is infinite, and that the orbits converge to a translate of the circular equilibrium. We fix $\varep_0$ and the spectral parameter $\rho$ small enough, and choose $\delta$ be a small independent of $\varep\in(0, \varep_0).$ In this sub-section we impose additional a priori assumptions on the velocity of the meandering parameters $\mbp$ and the perturbation $v^\bot$,
\beq\label{A-basic}
\mbp\in \cO_{2,\delta}, \qquad \|\dot\mbp\|_{l^2}\lesssim \varep^3,   \qquad \varep^{-2} \|v^\bot\|_{L^2} \leq \delta.
\eeq
These assumptions subsume the in \eqref{A-0}, of Theorem \ref{thm-est-p}, and of Corollary \ref{cor-ODE-pw}, for $\delta$ small enough. In particular, the auxiliary functions introduced in \eqref{def-E12} satisfy
\beq\label{est-E12}
E(t)\lesssim \delta, \qquad E_w(t)\lesssim \delta. 
\eeq 

\begin{lemma}\label{lem-est-p-2}
Under the a prior assumptions \eqref{A-basic}  with $\delta$ small enough, for $c_0$ given in \eqref{def-c0} 
the meander parameters evolving according to \eqref{eq-hatp} satisfy
\beqs
\frac{\dd}{\dd t}  |\mrp_0-\mrp_0^*|^2  + \varep^3 \frac{c_0}{2R_0} |\mrp_0-\mrp_0^*|^2 \lesssim   \varep^5\delta^2 \|\hat \mbp\|_{\mbV_4^2}^2 +\varep^2 \|v^\bot\|_{L^2}^2  +\varep^{-2}\|\mathrm N(v^\bot)\|_{L^2}^2;
\eeqs
and
\beqs
\frac{\dd}{\dd t} \|\hat\mbp\|_{\mbV_2^2}^2  + \frac{\varep^4}{32 R_0^2} \|\hat\mbp\|_{\mbV_4^2}^2 \lesssim  \varep^2 |\mrp_0-\mrp_0^*|^2 \|\hat\mbp\|_{\mbV_2^2}^2+\varep \|v^\bot\|_{L^2}^2 +\varep^{-3}\|\mathrm N(v^\bot)\|_{L^2}^2.
\eeqs 
Moreover, the time derivative  of $\mbp$ has the following $l^2$ bound
\beqs
\|\dot\mbp\|_{l^2}\lesssim \varep^3 |\mrp_0-\mrp_0^*|+\varep^4 \|\hat \mbp\|_{\mbV_4^2}+\varep^{5/2}\|v^\bot\|_{L^2}+\varep^{1/2}\|\mathrm N(v^\bot)\|_{L^2}.
\eeqs

\end{lemma}
\begin{proof}
Multiplying equation for $\mrp_0$ from Theorem \ref{thm-est-p} by $2(\mrp_0-\mrp_0^*)$ and applying Young's inequality we deduce the bound
\beqs
 \frac{\dd}{\dd t} |\mrp_0-\mrp_0^*|^2  +\varep^3 \frac{c_0}{R_0}|\mrp_0-\mrp_0^*|^2   \lesssim \varep^{-3} d_0^2.
\eeqs
Using the bound on the remainder $d_0$ from Theorem \ref{thm-est-p} with $E(t)$ bounded by \eqref{est-E12}, the inequality of $\mrp_0$ follows 
provided with $\delta$ small enough.  

For the first estimate on $\hat\mbp$,  we note that $\|\hat\mbp\|_{\mbV_2^2}=\|\mbD\hat\mbp\|_{l^2}$ and take the inner product of the evolution equation for $\hat\mbp$ in Theorem \ref{thm-est-p} with $2\mbD^2\hat\mbp$. This yields the equality
\beq\label{est-p-V22-1}
\frac{\dd}{\dd t} \|\hat\mbp\|_{\mbV_2^2}^2 =  -\varep^3 \frac{2c_0}{R_0}(\mrp_0-\mrp_0^*) \left<(\mbD+\mbU^T) \hat\mbp, \mbD^2\hat\mbp \right>_{l^2} -\frac{2\varep^4}{R_0^2} \left<(\mbD-\mbI)^2 \hat\mbp, \mbD^2\hat\mbp\right>_{l^2} +2\left<\hat d, \mbD^2 \hat\mbp\right>_{l^2}.
\eeq
By H\"older's inequality and the bound $\|\mbU^T\|_{l_2^*}\lesssim \|\mbD\|_{l_2^*}$ from Corollary\,\ref{cor-bU}, the first term on the right-hand side  of \eqref{est-p-V22-1} can be bounded from above by 
\beqs
\begin{aligned}
-\varep^3  \frac{2 c_0}{R_0} (\mrp_0-\mrp_0^*) \left<(\mbD+\mbU^T)  \hat\mbp, \mbD^{2} \hat\mbp \right>_{l^2} & \lesssim \varep^3|\mrp_0-\mrp_0^*| \|\mbD \hat\mbp\|_{l^2} \|\mbD^{2} \hat\mbp\|_{l^2} \\
&\leq  C \varep^2 \|\hat\mbp\|_{\mbV_3^2}^2 |\mrp_0-\mrp_0^*|^2  +\frac{\varep^4}{32R_0^2} \|\hat\mbp\|_{\mbV_5^2}^2.
\end{aligned}
\eeqs
for a constant $C$ independent of $\varep\in(0,\varep_0)$.
From Lemma \ref{lem-mbD}, the second term on the right-hand side is negative and can be bounded from above by
\beqs
 -\frac{2\varep^4}{R_0^2} \left<(\mbD-\mbI)^2 \hat\mbp, \mbD^2\hat\mbp\right>_{l^2} \leq   -\frac{2\varep^4}{R_0^2} \|(\mbD-\mbI)^{2} \hat\mbp\|_{l^2}^2 \leq  -\frac{\varep^4}{8 R_0^2} \|\hat\mbp\|_{\mbV_4^2}^2. 
\eeqs 
Then employing H\"older and Young's inequality to bound the  third item on the right hand side of \eqref{est-p-V22-1} implies
\beqs
\frac{\dd}{\dd t} \|\hat\mbp\|_{\mbV_2^2}^2  + \frac{\varep^4}{16R_0^2} \|\hat\mbp\|_{\mbV_4^2}^2 \lesssim  \varep^2 |\mrp_0-\mrp_0^*|^2 \|\hat\mbp\|_{\mbV_2^2}^2+ \varep^{-4}\|\hat d\|_{l^2}^2.
\eeqs
It remains  to obtain an bound on $\|\hat d\|_{l^2}$. In fact, $\|\hat d\|_{l^2}\leq \|d'\|_{l^2}$ and the latter is bounded in Theorem\, \ref{thm-est-p}. We note that $E(t)\lesssim \delta$ and $\varep^{-2}\|v^\bot\|_{L^2}\leq \delta$.  Using the bound on $\|d'\|_{l^2}$ from Theorem\,\ref{thm-est-p}, we have
\beqs 
\varep^{-4}\|\hat d\|_{l^2}^2 \lesssim \varep^2 \delta^2 |\mrp_0-\mrp_0^*|^2 \|\hat\mbp\|_{\mbV_2^2}^2+  \varep^4|\mrp_0-\mrp_0^*|^2+\varep^4\delta^2 \|\hat\mbp\|_{\mbV_4^2}^2 +\varep \|v^\bot\|_{L^2} +\varep^{-3}\|\mathrm N(v^\bot)\|_{L^2}^2.
\eeqs 
Absorbing these terms involving   $\|\hat\mbp\|_{\mbV_4^2}$ and
$\|\hat\mbp\|_{\mbV_2^2}$ for $\delta$ suitably small independent of $\varep$, we obtain the first estimate on $\hat\mbp$.

Finally, the $l^2$-bound on $\dot\mbp$ follows from  the bound on $\|\dot\mbp\|_{l^2}$ and the estimates on $d=(d_0, d')$ in Theorem \ref{thm-est-p}.
\end{proof} 

We require bounds on the evolution of weighted norms of $\mbp.$
\begin{lemma}\label{lem-est-p-V32} Under the a priori assumptions of \eqref{A-basic} with $\delta$ small enough, 
there exists a strictly positive constant $c_*$ independent of $\varep\in(0,\varep_0), \rho, \delta$ such that
\beqs
 \frac{\dd}{\dd t}\|\hat\mbp\|_{\mbV_3^2}^2 +
 c_*\varep^4\|\hat\mbp\|_{\mbV_5^2}^2 \lesssim \varep^2|\mrp_0-\mrp_0^*|^2 \|\hat\mbp\|_{\mbV_3^2}^2 +\varep^4 |\mrp_0-\mrp_0^*|^2+ \varep^{-2} \|v^\bot\|_{L^2}^2 +\varep^{-5}\|\mathrm N(v^\bot)\|_{L^2}^2.
\eeqs
\end{lemma}

\begin{proof}
 Since $\|\hat\mbp\|_{\mbV_3^2}=\|\mbD^{3/2}\hat\mbp\|_{l^2}$ we take the $l^2$-inner product of the weighted evolution equation in Corollary \ref{cor-ODE-pw}  with $2\mbD^{5/2}\hat\mbp$. This yields the equality
\beq\label{est-p-V32-1}
\begin{aligned}
\frac{\dd}{\dd t} \|\hat\mbp\|_{\mbV_3^2}^2 =& -\varep^3  \frac{2c_0}{R_0} (\mrp_0-\mrp_0^*) \left<(\mbD+\mbU^T) \mbD^{1/2} \hat\mbp, \mbD^{5/2} \hat\mbp \right>_{l^2}-\frac{2\varep^4 }{R_0^2}  \left<(\mbD-\mbI)^2 \mbD^{1/2} \hat\mbp, \mbD^{5/2} \hat\mbp\right>_{l^2}\\
& +2\left<\hat d_w, \mbD^{5/2} \hat\mbp\right>_{l^2}.
\end{aligned}
\eeq
By H\"older's inequality and the bound $\|\mbU^T\|_{l_2^*}\lesssim \|\mbD\|_{l_2^*}$ from Corollary\,\ref{cor-bU}, the first term on the right-hand side  of \eqref{est-p-V32-1} can be bounded from above by 
\beqs
\begin{aligned}
-\varep^3  \frac{c_0}{R_0} (\mrp_0-\mrp_0^*) \left<(\mbD+\mbU^T) \mbD^{1/2} \hat\mbp, \mbD^{5/2} \hat\mbp \right>_{l^2} & \lesssim \varep^3|\mrp_0-\mrp_0^*| \|\mbD^{3/2} \hat\mbp\|_{l^2} \|\mbD^{5/2} \hat\mbp\|_{l^2} \\
&\leq  C \varep^2 \|\hat\mbp\|_{\mbV_3^2}^2 |\mrp_0-\mrp_0^*|^2  +\frac{\varep^4}{64R_0^2} \|\hat\mbp\|_{\mbV_5^2}^2.
\end{aligned}
\eeqs
for a constant $C$ independent of $\varep\in(0,\varep_0)$. The second inequality above follows from an application of Young's inequality and the definition of $\mathbb V_3^2, \mathbb V_5^2$ in terms of $\mbD$. 
To address the second term on the right-hand side of \eqref{est-p-V32-1} we recall that the diagonal matrices satisfy $\frac{1}{2}\mbD\leq \mbD-\mbI\leq \mbD$ and Lemma\,\ref{lem-mbD} so that
\beqs
\begin{aligned}
-\frac{2\varep^4 }{R_0^2}  \left<(\mbD-\mbI)^2 \mbD^{1/2} \hat\mbp, \mbD^{5/2} \hat\mbp\right>_{l^2} \leq  -\frac{2\varep^4 }{R_0^2} \|(\mbD-\mbI)^{5/2} \hat\mbp\|_{l^2}^2  \leq -\frac{\varep^4 }{16R_0^2} \|\mbD^{5/2} \hat\mbp\|_{l^2}^2\leq -\frac{\varep^4 }{16 R_0^2} \| \hat\mbp\|_{\mbV_5^2}^2.
\end{aligned}
\eeqs
To bound the third term on the right-hand side of \eqref{est-p-V32-1} we applying H\"older's and Young's inequalities, which yield $C>0$ such that 
\beqs
\left<\hat d_w, \mbD^{5/2} \hat\mbp\right>_{l^2}\leq  \frac{\varep^4}{64R_0^2} \|\hat\mbp\|_{\mbV_5^2}^2 +C\varep^{-4}\|\hat d_w\|_{L^2}^2.
\eeqs
Using the $l^2$-bound of $\hat d_w$ from Corollary \ref{cor-ODE-pw} and  returning the three estimates above to \eqref{est-p-V32-1}, and taking $c_*=\frac{1}{32 R_0^2}$ completes the proof provided with $\varep_0$ and $\delta$ small enough. 
\end{proof}

To close our analysis we must reconcile the a priori bounds with the estimates we have derived. This requires that we restrict the initial energy in the system as measured by the difference between the initial and equilibrium curve length, $|\mrp_0-\mrp_0^*|,$ and the weighted $\mbV_{2}^2$- and $\mbV_3^2$-norms of $\hat\mbp$ that control the deviation of the initial curve from circularity. Correspondingly we introduce the parameter set 
\beq
\label{def:cO-circ}
\cO^\circ_{m, \delta}:=  \left\{ \mbp\in \mbR^{N_1} \; \big| \; \varep^{-1/2} |\mrp_0-\mrp_0^*|+  \|\hat\mbp\|_{\mbV_3^2}<m \delta\right\}. 
\eeq 
Our analysis requires $\cO_{m,\delta}^\circ\subset  \cO_{m,\delta}$, see \eqref{def-cO-p}. This is established under the assumption on the initial data.

\begin{lemma}
Fix $\delta>0$, and let $\varep_0>0$ be sufficiently small. Then for any admissible pair $(\Gamma_0, M_0)$ with base bulk density $\sigma_0$, given in \eqref{def-c0}, that is sufficiently close to the equilibrium bulk density $\sigma_1^*$, defined in \eqref{def-hatla1*},  satisfying
\beq\label{A-sigma0*}
|\sigma_0(\Gamma_0,M_0)-\sigma_1^*|\leq  \frac{c_0m_1^2 R_0}{2m_0} \delta,
\eeq 
we have $\cO_{m,\delta}^\circ\subset  \cO_{m,\delta}$ for $m\geq 1$. 
\end{lemma}
\begin{proof}
Pick any $\mbp\in \cO^\circ_{m,\delta}$, we are going to prove $\mbp\in \cO_{m,\delta}$ for $m\geq 1$. By the definition of $\cO_{m,\delta}$ in \eqref{def-cO-p}, it suffices to show $|\mrp_0|\leq \delta$. From Lemma \ref{lem-sigma}  and the triangle inequality we have
\beqs
\begin{aligned}
|\mrp_0|& 
\leq  \frac{m_0}{c_0m_1^2R_0}|\sigma(\mbp) -\sigma_1^*| +\frac{m_0}{c_0m_1^2R_0}|\sigma_1^*-\sigma_0| +C\varep.
\end{aligned}
\eeqs
Recalling $\sigma^*=\sigma_1^*+O(\varep)$ , the bound above, together with Corollary \ref{cor-def-hatla}  and $|\mrp_0-\mrp_0^*|\leq m\varep^{1/2}$ for $\mbp\in \cO_{m,\delta}$ yields
\beqs
|\mrp_0| 
\leq |\mrp_0-\mrp_0^*| +\frac{\delta}{2} + C\varep< \delta, 
\eeqs 
for $\varep\in(0,\varep_0)$ and $\varep_0$ sufficiently small with respect to $\delta.$
\end{proof}

Our main result establishes the asymptotic stability of the nearly circular bilayer manifold. We assume that  $(\Gamma_0, M_0)$ is a admissible pair, that the base interface $\Gamma_0$ is circle of radius $R_0>0$, and the pair satisfy  \eqref{A-sigma0*} for which  $\sigma_1^*=\sigma_1^*(\eta_1, \eta_2)$, introduced in \eqref{def-hatla1*}, satisfies the  pearling stability condition $(\mathbf{PSC}_*)$, \eqref{PSC*}.

\begin{thm}
\label{thm-main-2}
Consider the mass preserving gradient flow \eqref{eq-FCH-L2-p}-\eqref{eq-FCH-L2} subject to periodic boundary conditions.  Let  $(\Gamma_0, M_0)$ be a admissible pair that satisfies \eqref{A-sigma0*}, and $\sigma_1^*=\sigma_1^*(\eta_1, \eta_2)$ given in \eqref{def-hatla1*} satisfies the  pearling stability condition $(\mathbf{PSC}_*)$, \eqref{PSC*}.
Then there exists $\delta$, $\rho$, and $\varep_0$ small enough such that for all $\varep\in(0,\varep_0)$,  the solution $u$ of the mass preserving gradient flow arising from initial data $u_0\in \mcV_{\varep^{5/2}}(\cM_b(\Gamma_0, M_0;\rho), \cO^\circ_\delta)$ defined in \eqref{def-tubular}, remains in a slightly bigger set for all $t\in [0, \infty)$. Indeed 
the solution  admits the decomposition \eqref{decomp-u-MT}
and there are constants $C,c>0$ independent of $\delta, \rho,$ $\varep\in(0,\varep_0),$ and choice of initial data
for which the orthogonal perturbation $v^\bot$ satisfies
\beq\label{est-vbot}
\|v^\bot\|_{\Hin}\leq C\varep^{5/2}\rho^{2}, \qquad \|v^\bot\|_{\Hin} \leq C e^{-c\varep^4 t},
\eeq
and projected meander parameters relax to an equilibrium value $\mbp^*=(\mrp_0^*, \mrp_{1,2}^*, \bm 0)$ 
according to
\beq
\label{p-decay}
|\mrp_0-\mrp_0^*|^2 +\varep \|\hat\mbp\|_{\mbV_2^2}^2 \leq 4 \varep\delta^2 e^{-c \varep^4 t}. 
\eeq
The equilibrium curve length parameter $\mrp_0^*$ admits the approximation 
\beq\label{app-p0*}
\mrp_0^*=-\frac{m_0}{c_0R_0 m_1^2} (\sigma_1^*-\sigma_0) +O(\varep), \quad \sigma_1^*=-\frac{\eta_1+\eta_2}{2m_0}m_1^2, \quad  \sigma_0= \frac{M_0-2\pi R_0m_0}{B_2^\infty|\Omega|}+O(\varep^2), 
\eeq
where  the positive constants $c_0, m_0, m_1$ are defined in \eqref{def-c0}, \eqref{def-m0}, $\eta_1, \eta_2$ are system parameters. 
For all $k\leq 4$ we have the temporal $L^2$ bound 
\beq\label{p-L2-Vk}
\varep^4 \int_0^\infty e^{c\varep^4 t}\|\hat\mbp\|_{\mbV_k^2}^2 \dd t\leq 4\delta^2.
\eeq
The translation parameters $\mrp_1, \mrp_2$ remain within $O(\delta)$ of their initial value $\mrp_1(0), \mrp_2(0)$ and converge to
$\mrp_1^*, \mrp_2^*$ as $t\to\infty.$
\end{thm}
\begin{proof}  
Since $u_0\in \mcV_{\varep^{5/2}}(\cM_b, \cO_\delta^\circ)$, there exists $\mbp_0\in \cO_\delta^\circ$ and $v_0\in \Hin$ satisfying $\|v_0\|_{\Hin}\leq \varep^{5/2}$ such that $u_0=\Phi_{\mbp_0} + v_0$.  Lemma \,\ref{lem-Manifold-Projection} affords the decomposition  $u_0=\Phi_{\mbp(0)}+v_0^\bot$ where $\Phi_{\mbp(0)}=\Pi_{\cM_b}u_0$ and $v_0^\bot$ is the orthogonal perturbation. The distance from  $\mbp(0)$ to $\mbp_0\in \cO_\delta^\circ$ can be bounded by
\beqs
\|\mbp(0)-\mbp_0\|_{l^2}\lesssim \varep^3, \qquad \|\hat\mbp(0)-\hat \mbp_0\|_{\mbV_3^2}\lesssim N_1^{3}\|\hat\mbp(0)-\hat \mbp_0\|_{l^2}\lesssim \rho^{3/4},
\eeqs 
where we applied Lemma \,\ref{lem-est-V} 
and bounded $N_1$ from \eqref{est-N01}. Note that $\mbp_0\in\cO_\delta^\circ$. Hence for $\varep_0, \rho$ small enough depending on $\delta$, triangle inequality implies  initial meander parameters satisfy
\beq \label{MT-init-est}
\varep^{-1/2}|\mrp_0(0)-\mrp_0^*|+
\|\hat\mbp(0)\|_{\mbV_3^2}\leq \frac{3\delta}{2},
\eeq 
and there exists $T>0$ such that   
\begin{equation}\label{assump-A}
({\bf A})\qquad 
\begin{aligned}    
& \varep^{-1/2}|\mathrm p_0(t) -\mrp_0^* |+ \|\hat{\mbp}\|_{\mbV_3^2}(t) <  2 \delta, \quad \|\dot{\mbp}\|_{l^2} (t)< 2 \varep^3\qquad \forall t\in[0,T). 
\end{aligned}
\end{equation}
We show $T=\infty$ in the following. 

The equilibrium pearling stability condition $(\mathbf{PSC}_*)$ holds by assumption. Under the a priori assumption $(\mathbf A)$ the dynamic pearling stability condition $(\mathbf{PSC})$ \eqref{cond-P-stab} holds uniformly.  Indeed, since $\sigma^*=\sigma_1^*+O(\varep)$, $(\mathbf{PSC})$ holds if  $\sigma$ stays close to $\sigma^*$, which follows from Corollary \ref{cor-def-hatla} since 
\beq\label{main-est-sigma}
|\sigma -\sigma^*|\lesssim |\mrp_0-\mrp_0^*| +\varep^2\|\hat\mbp\|_{\mbV_2^2} \lesssim \varep^{1/2}\delta.
\eeq
Choosing $\varep_0$ small enough the dynamic pearling stability condition $(\mathbf{PSC})$ holds uniformly on $[0,T]$ and Theorem \ref{thm-main-1} applies, which in turn affords the following uniform bounds on $w, \mbq$
\beq\label{est-w-q-uniform}
 \left<\mbL_{\mathbf p} w, w\right>_{L^2}\lesssim \varep^5\rho^{- 2}, \qquad \|\mbq\|_{l^2}^2\lesssim\varep^{5}\rho^{-4}\qquad \forall t\in[0, T).
\eeq 
 Since $v^\bot=Q+w$ is an orthogonal decomposition and $\|Q\|_{\Hin}\sim \|\mbq\|_{l^2}$,  we may estimate $v^\bot$  and  the nonlinear term $\mathrm N(v^\bot)$ from Lemma \,\ref{lem-est-N},
\beq\label{est-v-1}
\begin{aligned}
\|v^\bot\|_{\Hin}^2 \lesssim   \rho^{-2}\left<\mbL_\mbp w,w\right>_{L^2} +\|\mbq\|_{l^2}^2, \qquad 
\|\mathrm N(v^\bot)\|_{L^2}^2&\lesssim \varep^{-2} \left(\rho^{-2} \left<\mbL_\mbp w,w\right>_{L^2} +\|\mbq\|_{l^2}\right)^2.
\end{aligned}
\eeq
The $\Hin$ bound of $v^\bot$ in \eqref{est-vbot} follows from the coercivity of $\mbL_\mbp$. 
Furthermore,  we rewrite these bounds in \eqref{est-v-1} as
\beq\label{est-v}
\begin{aligned}
\|v^\bot\|_{\Hin}^2 \lesssim \rho^{-4} \|\mbL_\mbp w\|_{L^2}^2 +\|\mbq\|_{l^2}^2 , \qquad 
\|\mathrm N(v^\bot)\|_{L^2}^2\lesssim \varep^{3}\rho^{-8}  \|\mbL_\mbp w\|_{L^2}^2 +\varep^3\rho^{-4} \|\mbq\|_{l^2}^2.
\end{aligned}
\eeq

Combining the  estimates above  and \eqref{est-w-q-uniform} with Lemma \ref{lem-est-p-2} yields an estimate of $\|\dot\mbp\|_{l^2}$. In particular,  for $\varep_0$ small enough 
\beq\label{est-dotp-l2-1} 
\|\dot\mbp\|_{l^2}\leq \varep^3 \qquad \forall t\in[0,T).
\eeq 
 It remains to bound the length residual $|\mrp_0-\mrp_0^*|$ and $\|\hat\mbp\|_{\mbV_3^2}$ in order to verify the apriori assumptions  $(\mathbf A)$.
From Lemma \ref{lem-est-p-2} and \eqref{est-v} obtain the useful bound of $\|\dot\mbp\|_{l^2}$  
\beq\label{est-dotp-l2} 
\|\dot\mbp\|_{l^2}^2\lesssim \varep^6 |\mrp_0-\mrp_0^*|^2 +\varep^8 \|\hat\mbp\|_{\mbV_4^2}^2+ \varep^4 \rho^{-8} \|\mbL_\mbp w\|_{L^2}^2 +\varep^4 \rho^{-4} \|\mbq\|_{l^2}^2.
\eeq

\medskip
\noindent
{\it Step 1: Uniform estimates of $|\mrp_0-\mrp_0^*|$ and $\hat \mbp$ in $\mbV_3^2$. }  
 We introduce a mixed, weighted energy: 
 \beqs
 \mathcal E_1(t):= \|\hat\mbp\|_{\mbV_3^2}^2 + \varep^{-1}|\mrp_0-\mrp_0^*|^2 +\varep^{-2}\rho^{-10} \left<\mbL_\mbp w, w\right>_{L^2}+\varep^{-3} \rho^{-5}\|\mbq\|_{l^2}^2. 
 \eeqs
 Combining the first estimate on $w$ from Lemma \,\ref{lem-est-w}, $\mbq$ from Lemma \,\ref{lem-est-q}, and the $\mbV_3^2$-estimate of $\hat\mbp$ from Lemma \,\ref{lem-est-p-V32}, and the $\|\dot \mbp\|_{l^2}$ bound \eqref{est-dotp-l2} we find a revised positive constant $c_*$ independent of $\varep, \rho, \delta$ for which the $\mcE_1$-dissipation inequality
 \beqs
 \begin{aligned}
 \frac{\dd }{\dd t} \mcE_1(t) &+c_*\left(\varep^4 \|\hat\mbp\|_{\mbV_5^2}^2 +\varep^2 |\mrp_0-\mrp_0^*|^2  +\varep^{-2} \rho^{-10}\|\mbL_\mbp w\|_{L^2}^2 +\varep^{-2} \rho^{-5}\|\mbq\|_{l^2}^2 \right) \\
 &\lesssim  
 G(t)+\varep^5\rho^{-10}+\varep^{-2}\|v^\bot\|_{L^2}^2 +\varep^{-5} \|\mathrm N(v^\bot)\|_{L^2}^2
 \end{aligned}
 \eeqs 
 holds. Here we have introduced 
 \beqs 
 \begin{aligned}
 G(t):=&\varep^{-3}\rho^{-10}\|\dot \mbp\|_{l^2}^2 +\varep^4(\delta^2+\varep\rho^{-10}+\varep^2\rho^{-2})\|\hat\mbp\|_{\mbV_4^2}+ \varep^2(\delta^2+\varep\rho^{-10})|\mrp_0-\mrp_0^*|^2\\
 &+\varep^{-2}\rho^{-10}(\rho+\varep \rho^{-8})\|\mbL_\mbp w\|_{L^2}^2 + (\rho^{-13}+\varep \rho^{-14})\|\mbq\|_{l^2}^2. 
 \end{aligned}
 \eeqs 
 For $\delta, \rho$ small enough, and $\varep\in(0,\varep_0)$ with $\varep_0 $  small enough, depending on $\rho,\delta$, the function $G(t)$ can be absorbed into the positive term on the left-hand side of the $\mcE_1$-dissipation inequality. The bounds on the orthogonal perturbation $v^\bot$ and nonlinear terms $\mathrm N(v^\bot)$ in \eqref{est-v}  and estimate of $w, \mbq$ in \eqref{est-w-q-uniform} imply
 \beqs
 \frac{\dd }{\dd t} \mcE_1(t) +\frac{c_*}{2}\left(\varep^4 \|\hat\mbp\|_{\mbV_5^2}^2 +\varep^2 |\mrp_0-\mrp_0^*|^2  +\varep^{-2} \rho^{-10}\|\mbL_\mbp w\|_{L^2}^2 +\varep^{-2} \rho^{-5}\|\mbq\|_{l^2}^2 \right) \lesssim  \varep^5\rho^{-10}.
 \eeqs  
Since  $\|\hat\mbp\|_{\mbV_3^2}\leq \|\hat\mbp\|_{\mbV_5^2}$, we deduce that  $\mcE_1$-dissipation inequality reduces to the simple form
\beqs
 \frac{\dd }{\dd t} \mcE_1(t) +\frac{c_*}{2}\varep^4 \mcE_1(t) \lesssim  \varep^5\rho^{-10}.
\eeqs
 Multiplying by $e^{c_*\varep^4 t/2} $, integrating,
 and applying the initial bound \eqref{MT-init-est}, we derive the uniform bound
 \beqs
 \mcE_1(t)\leq e^{-c_*\varep^4 t/2} \mcE_1(0) +C \varep\rho^{-10}.
 \eeqs
The right hand side is strictly less than $ 4\delta^2$ for $\varep_0$ small enough depending on $\rho, \delta$. From the definition of $\mcE_1(t)$, the first assumption in $(\mathbf A)$ holds uniformly on $[0, T)$, which combined with the $l^2$-bound of $\dot\mbp$ in \eqref{est-dotp-l2-1} yields $T=\infty$. 
 
 \medskip
 \noindent
 {\it Step 2. Decay estimates.  } To obtain a decay estimate, we introduce a second mixed weighted energy
\beq
\label{E2-def}
\mcE_2(t)=\|\hat\mbp\|_{\mbV_2^2}^2 +\varep^{-1} |\mrp_0-\mrp_0^*|^2 +\rho^{-10} \left< \mbL_\mbp w, w\right>_{L^2}  + \varep^{-1} \rho^{-5} \|\mbq\|_{l^2}^2.
\eeq
From the definition of $\mcE_2$, the estimates \eqref{MT-init-est} and \eqref{est-w-q-uniform} yield the initial bound
\beq\label{IC-mcE2}
\mcE_2(0)< 4\delta^2.
\eeq
Combining the second estimate on $w$ from Lemma \,\ref{lem-est-w}, the $\mbq$-estimate from Lemma \,\ref{lem-est-q}, and the $\mbV_2^2$-estimate on $\hat\mbp$ from Lemma \,\ref{lem-est-p-2}, we find a revised positive constant $c_*$ independent of $\varep, \rho, \delta$ such that 
 \beq
 \begin{aligned}
 \frac{\dd }{\dd t}\mcE_2(t)&+ c_*  \left(\varep^4 \|\hat\mbp\|_{\mbV_4^2}^2+ \varep^2 |\mrp_0-\mrp_0^*|^2 +\rho^{-10}\|\mbL_\mbp w\|_{L^2}^2 +\rho^{-5}\|\mbq\|_{l^2}^2 \right)\\
 & \lesssim  \varep^{-1}\rho^{-10} \|\dot\mbp\|_{l^2}^2+  \varep \|v^\bot\|_{L^2}^2 +\varep^{-3}\|\mathrm N(v^\bot)\|_{L^2}^2.
 \end{aligned}
 \eeq
 The remainder of {\it Step 2} follows the approach of {\it Step 1}. We employ  the upper bound on $\dot\mbp$ from \eqref{est-dotp-l2} and the estimates on $v^\bot, \mathrm N(v^\bot)$ from  \eqref{est-v} to eliminate these terms from the right-hand side so long as $\varep_0\ll \rho\ll 1$. 
 We deduce that
 \beqs
 \frac{\dd}{\dd t} \mcE_2(t) +\frac{c_*}{2}  \left(\varep^4 \|\hat\mbp\|_{\mbV_4^2}^2+ \varep^2 |\mrp_0-\mrp_0^*|^2 +\rho^{-10}\|\mbL_\mbp w\|_{L^2}^2 +\rho^{-5}\|\mbq\|_{l^2}^2 \right) \leq 0, 
 \eeqs
Since $\|\hat\mbp\|_{\mbV_2^2}\leq \|\hat\mbp\|_{\mbV_4^2}$, the coercivity of $\mbL_\mbp$ from \eqref{coer-bLp} allows us to bound the positive term on the left-hand side from below,
\beqs
\frac{1}{2}\left(\varep^4 \|\hat\mbp\|_{\mbV_4^2}^2+ \varep^2 |\mrp_0-\mrp_0^*|^2 +\rho^{-10}\|\mbL_\mbp w\|_{L^2}^2 +\rho^{-5}\|\mbq\|_{l^2}^2\right)\geq \frac{\varep^4}{2}\mcE_2(t).  
\eeqs
This yields the $\mcE_2$-dissipation inequality
\beqs
 \frac{\dd}{\dd t} \mcE_2(t) +\frac{c_*\varep^4}{4}  \mcE_2(t) +\frac{c_*}{4}  \left(\varep^4 \|\hat\mbp\|_{\mbV_4^2}^2+ \varep^2 |\mrp_0-\mrp_0^*|^2 +\rho^{-10}\|\mbL_\mbp w\|_{L^2}^2 +\rho^{-5}\|\mbq\|_{l^2}^2 \right) \leq 0.
\eeqs
Multiplying both sides by $e^{c_*\varep^4 t/4} $ and integrating with respect to time from $0$ to $t$ yields
\beqs
e^{c_*\varep^4t/4}\mcE_2(t) + \int_0^t  e^{c_*\varep^4\tau/4}\left(\varep^4 \|\hat\mbp\|_{\mbV_4^2}^2+ \varep^2 |\mrp_0-\mrp_0^*|^2 +\rho^{-10}\|\mbL_\mbp w\|_{L^2}^2 +\rho^{-5}\|\mbq\|_{l^2}^2 \right) \dd \tau \leq  \mcE_2(0),
\eeqs
from which we deduce the asymptotic decay of the $\mcE_2$ on the $\varep^{-4}$ time-scale,
\beq\label{decay-est}
\mcE_2(t)  \leq e^{-c_*\varep^4t/2} \mcE_2(0)\qquad \forall t\in [0,\infty).  
\eeq
The decay estimates in \eqref{p-decay} and \eqref{decay-est} follow.  Moreover  the relaxation of the weighted norms is controlled by the initial energy, 
\beq\label{total-growth}
\int_0^\infty e^{c_*\varep^4\tau/4}\left(\varep^4 \|\hat\mbp\|_{\mbV_4^2}^2+ \varep^2 |\mrp_0-\mrp_0^*|^2 +\rho^{-10}\|\mbL_\mbp w\|_{L^2}^2 +\rho^{-5}\|\mbq\|_{l^2}^2 \right) \dd \tau \leq  \mcE_2(0). 
\eeq
\medskip
\noindent
The bound \eqref{IC-mcE2} on $\mcE_2(0)$ yields the temporal  estimate \eqref{p-L2-Vk} for $k=4,$ which implies them for $k<4$. 

\medskip
\noindent 
{\it Step 3. Relaxation of the translation parameters.}
The decay from {\it step 2} shows that the translational parameters $\mrp_{1,2}$ converge an equilibrium point $\mrp_{1,2}^*$ close to the initial value. In fact  for $k=1,2$ and any $t_1\leq t_2$ on $[0, \infty)$
\beqs
\begin{aligned}
|\mrp_{k}(t_2) -\mrp_{k}(t_1)| &\leq \int_{t_1}^{t_2} |\dot\mrp_k(\tau)|\dd \tau\leq \int_{t_1}^{t_2}\|\dot \mbp\|_{l^2} \dd \tau, 
\end{aligned}
\eeqs
which combined with \eqref{est-dotp-l2} implies
\beqs
\begin{aligned}
|\mrp_{k}(t_2) -\mrp_{k}(t_1)| 
&\lesssim \int_{t_1}^{t_2} \left(\varep^3|\mrp_0-\mrp_0^*| +\varep^4 \|\hat\mbp\|_{\mbV_4^2} + \varep^2\rho^{-4}\|\mbL_\mbp w\|_{L^2} +\varep^2 \rho^{-2}\|\mbq\|_{l^2}\right)\dd \tau.
\end{aligned}
\eeqs
We use the weighted norm relaxation estimate \eqref{total-growth} and H\"older's inequality to bound the right-hand side. The integral of the $|\mrp_0-\mrp_0^*|$-term satisfies 
\beqs
\begin{aligned}
\int_{t_1}^{t_2} \varep^3 |\mrp_0-\mrp_0^*|\dd \tau &\leq \varep^3 \left(\int_{t_1}^{t_2} e^{c_*\varep^4\tau/4}|\mrp_0-\mrp_0^*|^2 \dd \tau \right)^{1/2} \left(\int_{t_1}^{t_2} e^{-c_*\varep^4 \tau/4}\dd \tau\right)^{1/2},\\
&\lesssim  \left(\int_0^\infty e^{c_*\varep^4\tau/4}\varep^2 |\mrp_0-\mrp_0^*|^2 \dd \tau \right)^{1/2} e^{-c_*\varep^4t_1/8}\lesssim e^{-c_*\varep^4t_1/8} \mathcal E_2^{1/2}(0).
\end{aligned}
\eeqs
The other terms have similar or better bounds and  from \eqref{IC-mcE2} we have
\beqs 
|\mrp_k(t_2)-\mrp_k(t_1)|\lesssim e^{-c_*\varep^4t_1/8} \delta. 
\eeqs 
We deduce that $\mrp_{1,2}(t)$ converges to some unique equilibrium value $\mrp_{1,2}^*$ as time tends to $\infty$. In particular, taking $t_1=0, t_2=t$ yields
\beqs
|\mrp_k(t)-\mrp_k(0)|\lesssim \delta, \qquad \forall t\in [0,\infty). 
\eeqs
We conclude that  $\mrp_{1,2}(t)$ stays in a $C\delta$-neighborhood of initial data for some positive constant independent of $\varep,\rho,\delta$.  The proof is complete.

\end{proof}

 The projection of $u=u(t)$ onto the bilayer manifold defines the meander parameters $\mbp=\mbp(t)$ and induces a normal velocity on $\Gamma_\mbp$ defined through \eqref{def-Vp}.
 As shown in \eqref{Vn-Vp-resid} of Corollary\,\ref{cor-residual-proj} given an interface $\Gamma_\mbp$, the curvature induced velocity $V_\mbp$, defined in \eqref{normal-velocity},  drives the actual  normal velocity $\p_t\bm\gamma_\mbp\cdot\mbn_\mbp$. The following result quantifies this relation for all $\mbp\in\cO_{2,\delta}^\circ$.
 
 \begin{prop}
\label{cor-Vp} There exists a $C>0$, independent of $\rho$, $\delta$ and $\varep_0$ such that for all  $\mbp\in\cO_{2,\delta}^\circ$ and all $\varep\in(0,\varep_0)$ the 
normal velocities satisfy the error bound
 \beq 
 \label{Vn-Vp-est}
 \|\p_t\bm\gamma_\mbp\cdot\mbn_\mbp - V_\mbp\|_{L^2(\msI_\mbp)} \leq C \left(\varep^{9/2}\rho^{-4}+\varep^4\delta\|\hat\mbp\|_{\mbV_4^2} \right).
 \eeq
  \end{prop}
 \begin{proof}
 From the triangle inequality we have the relation
 \beq\label{NV-est-0}
 \begin{aligned}
 \|\p_t\bm\gamma_\mbp\cdot\mbn_\mbp - V_\mbp\|_{L^2(\msI_\mbp)} \leq &
 \left \|\Pi_{G_1}\left(\p_t\bm\gamma_\mbp\cdot\mbn_\mbp - V_\mbp\right)\right\|_{L^2(\msI_\mbp)} + \left \| \Pi^\bot_{G_1} \p_t\bm\gamma_\mbp\cdot\mbn_\mbp\right\|_{L^2(\msI_\mbp)} \\
 &+\left \|\Pi^\bot_{G_1} V_\mbp\right \|_{L^2(\msI_\mbp)},
\end{aligned}
\eeq
where $\Pi_{G_1}^\bot=\mathrm I-\Pi_{G_1}$ is the complement to the Galerkin  projection onto $G_1\subset L^2(\msI_\mbp)$, the space spanned by the first $N_1$ Laplace-Beltrami eigenmodes. 
For the first term, from \eqref{Vn-Vp-resid} of Corollary\,\ref{cor-residual-proj} and Lemma\,\ref{lem-est-Rv} we find that
\beqs
\begin{aligned}
\left \|\Pi_{G_1}\left(\p_t\bm\gamma_\mbp\cdot\mbn_\mbp - V_\mbp\right)\right\|_{L^2(\msI_\mbp)} \lesssim  &\;
\varep^{-1} \|v^\bot\|_{L^2} \|\dot\mbp\|_{l^2}+ (\varep^{5/2} +\varep^{5/2}\|\hat\mbp\|_{\mbV_4^2}) \|v^\bot\|_{L^2}+\varep^{1/2} \|\mathrm N(v^\bot)\|_{L^2}\\
&\quad +  \varep^{4}|\sigma_1^*-\sigma|+\varep^5+\varep^{5}\|\hat\mbp\|_{\mbV^2_4}+\varep\|\dot\mbp\|_{l^2}.
\end{aligned}
\eeqs
From estimates  \eqref{est-w-q-uniform}-\eqref{est-v-1} and $\|v^\bot\|_{L^2}\leq \|v^\bot\|_{\Hin}$ from the definition of $H^2$ inner norm in \eqref{def-Hin}, we bound $L^2$-norm of $v^\bot$ and $\mathrm N(v^\bot)$ as
\beqs
\|v^\bot\|_{L^2}\lesssim \varep^{5/2}\rho^{-2}, \qquad \|\mathrm N(v^\bot)\|_{L^2}\lesssim \varep^{4}\rho^{-4}.
\eeqs 
With these estimates, \eqref{main-est-sigma} with \eqref{p-decay}, and $|\sigma^*-\sigma_1^*|\leq \varep$, for $\varep_0$ small enough depending on $\rho$ we reduce the bound on the first term of the right hand side of \eqref{NV-est-0} to \beq\label{NV-est-1}
\left \|\Pi_{G_1}\left(\p_t\bm\gamma_\mbp\cdot\mbn_\mbp - V_\mbp\right)\right\|_{L^2(\msI_\mbp)} \lesssim \varep^{9/2}\rho^{-4} 
+\varep^5\rho^{-2}\|\hat\mbp\|_{\mbV^2_4}+\varep\|\dot\mbp\|_{l^2}.
\eeq
Bounding the last two terms in \eqref{NV-est-0} requires a standard Sobolev estimate of the $L^2$ projection to the high frequency space $G_1^\bot$ in terms of $H^1$-norm of a function. In fact for any function $f=f(s_\mbp)\in H^1(\msI_\mbp)$, 
\beq\label{Sobolev-Gonebot} 
\|\Pi^\bot_{G_1} f\|_{L^2(\msI_\mbp)} \leq \beta_{N_1}^{-1} \|f\|_{H^1(\msI_\mbp)}\lesssim \varep\rho^{-1/4} \|f\|_{H^1(\msI_\mbp)}.
\eeq
Here we used $\beta_{N_1}\sim N_1$ with $N_1$ bounded from below by  \eqref{est-N01}.
Applying the Sobolev estimate \eqref{Sobolev-Gonebot} twice to $\p_t\bm \gamma_\mbp \cdot\mbn_\mbp$ and utilizing the identity \eqref{GM-timede-proj} yields  the bound
\beqs
\|\Pi_{G_1}^\bot \p_t\bm \gamma_\mbp \cdot \mbn_\mbp\|_{L^2(\msI_\mbp)}\lesssim \varep^2\rho^{-1/2} \Bigl\|\sum_{j\in \Sigma_1}\dot\mrp_j\xi_j\Bigr\|_{H^2(\msI_\mbp),}
\eeqs 
where $(\xi_j)$, with its components given in Lemma \,\ref{lem-change-of-coord}, is  bounded in $H^2(\msI_\mbp)$-norm, independent of $\varep$ for all $\mbp\in \cO_{2,\delta}$. Hence  for $\varep\in(0, \varep_0)$ and $\varep_0\ll \rho$ we deduce that
\beq\label{NV-est-2}
\|\Pi_{G_1}^\bot \p_t\bm \gamma_\mbp \cdot \mbn_\mbp\|_{L^2(\msI_\mbp)} 
\lesssim \varep\|\dot\mbp\|_{l^2}. 
\eeq 
To bound the last term in \eqref{NV-est-0},  we  deduce from  the definition of $V_\mbp$ in \eqref{normal-velocity} and  triangle inequality that
\beqs
\|\Pi_{G_1}^\bot V_\mbp\|_{L^2(\msI_\mbp)}\lesssim \varep^3|\sigma-\sigma_1^*| \|\Pi_{G_1}^\bot\kappa_\mbp\|_{L^2(\msI_\mbp)} + \varep^{4}\|\Pi_{G_1}^\bot\Delta_{s_\mbp}\kappa_\mbp\|_{L^2(\msI_\mbp)} +\varep^{4}\|\Pi_{G_1}^\bot\kappa_\mbp^3\|_{L^2(\msI_\mbp)}.
\eeqs 
The curvature $\kappa_\mbp$  does not lie in the $N_1$-dimensional  linear space $G_1$. In fact, it admits the expansion \eqref{def-kappa-p,0}, for which the first two terms $\kappa_{\mbp,0}, \mathcal Q_1\in G_1$, and hence $\Pi_{G_1}^\bot\kappa_\mbp=\Pi_{G_1}^\bot\mcQ_2$. Applying \eqref{Sobolev-Gonebot}, the estimate \eqref{main-est-sigma}, and the $H^1(\msI_\mbp)$-bound of $\mathcal Q_2$ from Lemma \,\ref{lem-Gamma-p} we find
\beqs
\varep^3|\sigma-\sigma_1^*|\|\Pi_{G_1}^\bot\kappa_\mbp\|_{L^2(\msI_\mbp)} \lesssim \varep^{4}\rho^{-1/4}|\sigma-\sigma_1^*|\|\mcQ_2\|_{H^1(\msI_\mbp)} \lesssim \varep^{9/2}\rho^{-1/4} 
\eeqs 
for all $\mbp\in \cO_{2,\delta}^\circ$. From the Sobolev embedding $L^\infty(\msI_\mbp)\hook H^1(\msI_\mbp)$, \eqref{def-kappa-p,0},  the estimate 
$\|\mcQ_1\|_{H^1(\msI_\mbp)}\lesssim \|\hat\mbp\|_{\mbV_3^2}\lesssim 1$, and the $H^1(\msI_\mbp)$-bound of $\mcQ_2$ from Lemma \,\ref{lem-Gamma-p} we derive
\beqs 
\varep^{4}\|\Pi_{G_1}^\bot\kappa_\mbp^3\|_{L^2(\msI_\mbp)}\lesssim \varep^{5}\rho^{-1/4}. 
\eeqs 
Similarly, since $G_1$ is invariant under $\Delta_{s_\mbp}$ and $\Pi_{G_1}^\bot\Delta_{s_\mbp}\kappa_\mbp=\Pi_{G_1}^\bot\Delta_{s_\mbp}\mcQ_2$, applying the $H^2(\msI_\mbp)$-bound of $\mcQ_2$ from Lemma \,\ref{lem-Gamma-p} implies
\beqs
\varep^4 \|\Pi_{G_1}^\bot\Delta_{s_\mbp}\kappa_\mbp\|_{L^2(\msI_\mbp)}  \lesssim  \varep^4 \|\mathcal Q_2\|_{H^2(\msI_\mbp)}\lesssim  \varep^4 \|\hat\mbp\|_{\mbV_2^2}\|\hat\mbp\|_{\mbV_4^2}. 
\eeqs 
Combining the above three estimates yields
\beq\label{NV-est-3}
\|\Pi_{G_1}^\bot V_\mbp\|_{L^2(\msI_\mbp)}\lesssim \varep^{9/2}\rho^{-1/4}
+ \varep^{4}\|\hat\mbp\|_{\mbV_2^2}\|\hat\mbp\|_{\mbV_4^2}.
\eeq 
Now returning the estimates \eqref{NV-est-1} and \eqref{NV-est-2}-\eqref{NV-est-3} to \eqref{NV-est-0} we obtain
\beqs
\|\p_t\bm\gamma_\mbp\cdot\mbn_\mbp - V_\mbp\|_{L^2(\msI_\mbp)} \lesssim \varep^{9/2}\rho^{-4}
+\varep^4\|\hat\mbp\|_{\mbV_2^2}\|\hat\mbp\|_{\mbV_4^2} +\varep\|\dot\mbp\|_{l^2}. 
\eeqs 
The $l^2$-bound of $\dot\mbp$ with $\mbp\in \cO_{2,\delta}^\circ$ from \eqref{est-dotp-l2} completes the estimation.
\end{proof}
\begin{remark}
The dominant source of error in comparing the exact and formal normal velocity arises from
the truncation errror in the Galerkin projection of the surface diffusion term $\Delta_{s_\mbp} \kappa_\mbp.$  For a general $\mbp\in\cO_{2,\delta}^\circ$,  the largest term in $V_\mbp$ is generically the surface diffusion term which scales like $\varep^4\|\Delta_s \kappa_\mbp\|_{L^2(\msI_\mbp)}\sim \delta \varep^3$, while its Galerkin residual $\varep^4\delta\|\hat\mbp\|_{\mbV^2_4}\lesssim \varep^{3} \delta \|\hat\mbp\|_{\mbV^2_3}\lesssim \delta^2\varep^3$ is smaller. The  $L^2(\msI_\mbp)$ norm of the other terms in $V_\mbp$ typically scale like $\varep^{7/2}$ or $\varep^4.$
\end{remark}

 \section*{Acknowledgement}  The second author thanks the NSF DMS for their support through grant 1813203. Both authors thank Gurgen Hayrapetyan for sharing preliminary results on this problem for the weak FCH that arose out of his thesis.
 
 
\section{Appendix}
We present technical results whose proof was deferred from the main presentation.
\subsection{Results on the projection of the normal velocity}
\begin{lemma}\label{lem-kappa-Theta} Under  the assumption \eqref{A-00},  the curvature of $\Gamma_\mbp$ admits the following projection identities:
\beqs
\begin{aligned}
&\int_{\msI_\mbp}\kappa_{\mbp}\tilde \Theta_k\dd \tilde s_{\mbp}=-2\pi \Theta_0 \delta_{k0} -\frac{\beta_k^2-1}{R_0^2} \mrp_k\bm 1_{\{k\geq 3\}}+O\left(\|\hat{\mbp}\|_{\mbV_2^2}^2\right)e_k,\\[3pt]
&\int_{\msI_\mbp}\kappa_{\mbp}^3\tilde \Theta_k\dd \tilde s_{\mbp}=-\frac{2\pi \Theta_0}{R_0^2(1+\mrp_0)^2}\delta_{k0}-\frac{3(\beta_k^2-1)}{R_0^4(1+\mrp_0)^2}\mrp_k\bm 1_{\{k\geq 3\}} + O\left(\|\hat\mbp\|_{\mbV_2^2}\|\hat\mbp\|_{\mbV_3^2}\right)e_k,\\[3pt]
&\int_{\msI_\mbp}\Delta_{s_{\mbp}}\kappa_{\mbp}\tilde \Theta_k\dd \tilde s_{\mbp}= \frac{(\beta_k^2-1)\beta_k^2}{R_0^4(1+\mrp_0)^2}\mrp_k\bm 1_{\{k\geq 3\}}+O\left(\|\hat{\mbp}\|_{\mbV_2^2}\|\hat\mbp\|_{\mbV_4^2}\right)e_k\bm 1_{\{k\geq 1\}},
\end{aligned}
\eeqs
for $k\in\Sigma_1.$ Here $\bm e=(e_j)_{j=0}^{N_1-1}$ denotes a possibly different unit vector in each line.
\end{lemma}
\begin{proof} The curvature admits the expansion as in Lemma \,\ref{lem-Gamma-p},  and the quadratic term $\mathcal Q_2$ contributes 
\begin{equation}\label{cQ-1-Theta-k}
\int_{\msI_\mbp}\mathcal Q_2\tilde \Theta_k\dd \tilde s_{\mbp}=O\left(\|\hat{\mbp}\|_{\mbV_2^2}^2\right)e_k,
\end{equation}
while from the orthogonality \eqref{ortho-tTheta} the projection of the linear term takes the form
\begin{equation}\label{cQ-2-Theta-k}
\int_{\msI_\mbp}\mathcal Q_1\tilde \Theta_k\dd \tilde s_{\mbp}=\frac{(1-\beta_k^2)\mrp_k}{R_0^2}\bm 1_{\{k\geq 3\}}.
\end{equation}
The leading order term $\kappa_{\mbp,0}$ contributes
\begin{equation}\label{kappa-0-Theta-k}
\int_{\msI_\mbp}\kappa_{\mbp,0} \tilde \Theta_k\dd \tilde s_{\mbp}=-2\pi \Theta_0\delta_{k0}.
\end{equation}
Combining \eqref{def-kappa-p,0} with the identities  \eqref{cQ-1-Theta-k}-\eqref{kappa-0-Theta-k} yields the first result of the Lemma. 
For the $\kappa_\mbp^3$ projection, we expand
\begin{equation}
\kappa_{\mbp}^3=\kappa_{\mbp,0}^3+3\kappa_{\mathbf p,0}^2\mathcal Q_1+\tilde {\mathcal Q}_2,
\end{equation}
where $\tilde{\mathcal Q}_2$ denotes quadratic terms in $\mbp$ and satisfies
\beqs
|\tilde{\mathcal Q}_2|\lesssim |\mathcal Q_1|^3+|\mathcal Q_2|^3+|\mathcal Q_1|^2|\mathcal Q_2|+|\mathcal Q_1||\mathcal Q_2|^2
\eeqs
Since $|\mathcal Q_1|+|\mathcal Q_2|\lesssim \|\hat\mbp\|_{\mbV_2}\lesssim 1$, the assumption \eqref{A-00} and the estimates above imply
\beqs
\begin{aligned}
\|\tilde{\mathcal Q}_2\|_{L^2(\msI_\mbp)} &\lesssim \|\mathcal Q_1\|_{L^\infty}\|\mathcal Q_1\|_{L^2(\msI_\mbp)}+\|\mathcal Q_2\|_{L^2(\msI_\mbp)} \lesssim \|\hat\mbp\|_{\mbV_2}\|\hat\mbp\|_{\mbV_2^2}\lesssim \|\hat\mbp\|_{\mbV_3^2}\|\hat\mbp\|_{\mbV_2^2}.
\end{aligned}
\eeqs
Here we also used the embedding estimate in the perturbation space $\mbV$, see Lemma \,\ref{lem-est-V}.
Since $\kappa_{\mbp,0}$ is independent of $s_{\mbp}$, the second identity of Lemma follows from \eqref{cQ-2-Theta-k}, \eqref{kappa-0-Theta-k},
the definition of $\kappa_{\mbp,0}$ in \eqref{def-kappa-p,0}, and H\"older's inequality.  For the Laplace-Beltrami curvature projection, we integrate by parts and use \eqref{tTheta''}, to find
\beq 
\int_{\msI_\mbp} \Delta_{s_\mbp}\kappa_\mbp \tilde \Theta_k\dd \tilde s_\mbp=-\beta_{\mbp, k}^2\int_{\msI_{\mbp}}  \mathcal Q_1 \tilde \Theta_k\dd \tilde s_\mbp+\int_{\msI_\mbp} \Delta_{s_\mbp}\mathcal Q_2\, \tilde \Theta_k\dd \tilde s_\mbp.
\eeq
The first term on the right-hand side is dominant, and can be estimated by \eqref{cQ-2-Theta-k}. The second term on the right-hand side is higher order and
can be bounded by
\beq
\left|\int_{\msI_\mbp} \Delta_{s_\mbp}\mathcal Q_2\, \tilde \Theta_k\dd \tilde s_\mbp\right|\lesssim  \|\hat\mbp\|_{\mbV_2^2}\|\hat\mbp\|_{\mbV_4^2}e_k.
\eeq
The result follows.
\end{proof}

Recombining the identities in Lemma\,\ref{lem-kappa-Theta}, 
yields the following.
\begin{cor}\label{cor-sAB}For $\msA_k=\msA_k(\mbp), \msB_k=\msB_k(\mbp)$ defined in \eqref{def-sAB}, there exists a unit vector $(e_k)$  such that
\beqs
\begin{aligned}
\msA_k(\mbp)=& -2\pi \Theta_0 \delta_{k0} -\frac{\beta_k^2-1}{R_0^2} \mrp_k \bm 1_{\{k\geq 3\}} +O(\|\hat\mbp\|_{\mbV_2^2}^2)e_k,\\
\mathscr B_k(\mbp)=& 2\pi \Theta_0\left(\frac{1}{2R_0^2(1+\mrp_0)^2}- \alpha \right)   \delta_{k0} -\frac{\beta_k^2-1}{R_0^2(1+\mrp_0)^2}\left[ \frac{2\beta_k^2-3 }{2R_0^2} -\alpha (1+\mrp_0)^2\right]\mrp_k \mathbf 1_{k\geq 3}\\
&+O(\|\hat\mbp\|_{\mbV_2^2}\|\hat\mbp\|_{\mbV_4^2}) e_k \bm 1_{\{k\geq 1\}} +O(\|\hat\mbp\|_{\mbV_2^2}\|\hat\mbp\|_{\mbV_3^2}) e_k.
\end{aligned}
\eeqs
\end{cor}

\begin{lemma}\label{lem-est-h-proj}Let $h=h(\bm \gamma_\mbp'')$ enjoy the properties of Notation \ref{Notation-h}, then for j=0,1,2 
\beqs
\begin{aligned}
&\int \nabla_{s_\mbp}^jh(\bm \gamma_\mbp'') \tilde \Theta_k \dd\tilde s_\mbp = \mathcal C(\mrp_0)\delta_{k0} \delta_{j0}+ O(\|\hat\mbp\|_{\mbV_{2+j}^2}), \\
&\int \nabla_{s_\mbp}^jh(\bm \gamma_\mbp'') \varep \tilde \Theta_k' \dd\tilde s_\mbp=  O(\|\hat\mbp\|_{\mbV_{2+j}^2}).
\end{aligned}
\eeqs
\end{lemma}
\begin{proof}
From decomposition \eqref{decomp-h} the function $h$ can be rewritten as $h(\bm \gamma_{\mbp}'')= h(\bm \gamma_{\mbp,0}'')+(h(\bm \gamma_{\mbp}'')-h(\bm \gamma_{\mbp,0}''))$, 
and the integral of the leading order term $h(\bm \gamma_{\mbp,0}'')$ reduces to 
\begin{equation}\label{est-h-Theta-k-1}
\int_{{\msI_\mbp}} \nabla_{s_\mbp}^jh(\bm \gamma_{\mbp,0}'')\tilde \Theta_k\dd \tilde s_{\mbp}=\delta_{j0}h(\bm \gamma_{\mbp,0}'')\int_{{\msI_\mbp}}\Theta_k\dd \tilde s_{\mbp}=\mathcal C(\mrp_0)\delta_{k0}\delta_{j0},
\end{equation}
where the constant $\mathcal C(\mrp_0)$ depends only on $\mrp_0$.   Here we note $h(\bm \gamma_{\mbp,0}'')$ is independent of $s_\mbp$ by Lemma \ref{lem-h}.  Moreover,   we have the bound
\begin{equation}\label{est-h-Theta-k-2}
\begin{aligned}
\left|\int_{{\msI_\mbp}} \nabla_{s_\mbp}^j \left(h(\bm \gamma_{\mbp}'')- h(\bm \gamma_{\mbp,0}'')\right)\tilde \Theta_k\dd \tilde s_{\mbp}\right|&\lesssim \|\hat{\mbp}\|_{\mbV_2^2}e_k.
\end{aligned}
\end{equation}
The proof is complete.
\end{proof}

\begin{lemma}\label{lem-mass-F} Imposing assumptions \eqref{A-00}, then there exist smooth functions $C_k=C_k(\mrp_0)$ for $k=1,2$ such that
\begin{equation*}
\int_\Omega \left(\mathrm F(\Phi_{\mbp} ) -\mrF_m^\infty\right)\dd x=C_1(\mrp_0)\varep^4+C_2(\mrp_0)\varep^4(\sigma-\sigma_1^*)+ O\left(\varep^4\|\hat{\mbp}\|_{\mbV_2^2}, \varep^5\|\hat{\mbp}\|_{\mbV_4^2}\right).
\end{equation*}
\end{lemma}
\begin{proof}
This is a direct result of the form of $\mrF$ given in Lemma \ref{lem-def-Phi-p} and the  Lemma\,\ref{lem-est-h-proj}, details are omitted. 
\end{proof}

\begin{lemma}\label{lem-sR12} There exists a unit vector $\bm e=(e_k)_{k=0}^{N_1-1}$ such that the remainders defined in \eqref{def-cI-sR} satisfy
\beqs
\begin{aligned}
\msR_{k,1}(\mbp)& =\varep^{7/2}(\sigma_1^*-\sigma) \left(C_1(\mrp_0) \delta_{k0} +O(\|\hat\mbp\|_{\mbV_2^2})e_k\right) +\varep^{9/2} C_2(\mrp_0)+ O(\varep^{9/2}\|\hat\mbp\|_{\mbV_4^2}) e_k\\
\msR_{k,2}(\mbp)& =\varep^{11/2}\Big( C_1(\mrp_0)+  C_2(\mrp_0)(\sigma-\sigma_1^*) \Big)\delta_{k0}+ O\left(\varep^{11/2}\|\hat{\mbp}\|_{\mbV_2^2}\right)e_k
\end{aligned}
\eeqs
where $C_1$ and $C_2$ are smooth functions of $\mrp_0$.
\end{lemma}

\begin{proof}
For $\msR_{k,1}$ in \eqref{def-cI-sR},  we expand $\mrF$ as in Lemma \ref{lem-def-Phi-p} that is, $\mrF-\mrF_m^\infty= \varep^2\mrF_2 +\varep^3(\mrF_3-\mrF_3^\infty) +\varep^4(\mrF_4-\mrF_4^\infty) +e^{-\ell\varep/\nu}\mrF_e$ by noting $\mrF_2^\infty=0$. Integrating out $z_\mbp$
and using that the functions $\tilde \varphi_{1,k}=\varep^{-1/2}\varphi_{1,k}(z_\mbp, \bm \gamma_\mbp''), \tilde\varphi_{2,k}=\varep^{-1/2}\varphi_{2,k}(z_\mbp, \bm\gamma_\mbp'')$ from \eqref{eq-def-Z*} enjoy the properties of Notation \ref{Notation-h}, 
the leading order contribution from $\mrF_2$ takes the form
\beqs
\varep^{7/2} (\sigma_1^*-\sigma)\int_{\msI_\mbp}  \left( h_1(\bm \gamma_\mbp'') \tilde \Theta_k +h_2(\bm \gamma_\mbp'') \varep\tilde \Theta_k'\right) \dd\tilde s_\mbp.
\eeqs
Here we mention the dependence of   $\varphi_{1,k},\varphi_{2,k}$ on $s_\mbp$ are uniform in $k$ so that Lemma \ref{Notation-e_i,j}  applies. 
Applying Lemma \ref{lem-est-h-proj} with $j=0$ we see that this term provides the leading order contribution to $\msR_{k,1}$. 
From the form of $\mrF_3, \mrF_4$ in Lemma \ref{lem-def-Phi-p} and Lemma \ref{lem-est-h-proj} we find that the remaining terms can be  bounded by $\varep^{9/2}\|\hat\mbp\|_{\mbV_4^2}$. 
\medskip

To estimate $\msR_{k,2}$ we turn to the definition of $Z_{\mbp, *}^{1k}$ with $k\in \Sigma_1$ and use that $\psi_1=\phi_0'/m_1$ has odd parity to derive 
\begin{equation}\label{est-proj-F-2-1}
\begin{aligned}
\msR_{k,2}&=C\int_\Omega \left(\mathrm F(\Phi_{\mbp} )-\mrF_m^\infty\right)\dd x\left(\varep^{3/2}\int_{\msI_\mbp} \left(h_1(\bm \gamma_{\mbp}'' )\tilde \Theta_k+h_2(\bm \gamma_{\mbp}'' )\varep \tilde \Theta_k'\right)\dd \tilde s_{\mbp}\right).
\end{aligned}
\end{equation}
Applying Lemma \ref{lem-est-h-proj} to $h_1, h_2$,  the identity \eqref{est-proj-F-2-1}  reduces to
\begin{equation}\label{est-proj-I2-2}
\begin{aligned}
\msR_{k,2}& = \varep^{3/2}\left(C(\mrp_0)\delta_{k0}+O\|\hat{\mbp}\|_{\mbV_2^2})e_k\right)\int_\Omega \left(\mathrm F(\Phi_{\mbp} )-\mrF_m^\infty\right)\dd x,
\end{aligned}
\end{equation}
which, combined with Lemma \ref{lem-mass-F}, yields the revised functions $C_1(\mrp_0), C_2(\mrp_0)$  which appear in the statement of the Lemma.
\end{proof}

\subsection{Weighted estimates}
The proof of the weighted estimates in Corollaries\,\ref{cor-proj-Phi_t} and \ref{cor-proj-F-w} are based on the following Lemma which is primarily an integration by parts.

\begin{lemma}\label{lem-h-proj-w} Let $f=f(s_\mbp)$ be a function of $s_\mbp$. Then  if $f\in H^1(\msI_\mbp)$ there exists a unit vector $(e_k)\in l^2$ such that 
\beqs
\int_{\msI_\mbp} f(s_\mbp) \beta_k \tilde \Theta_k \dd\tilde s_\mbp=  O(\|\nabla_{s_\mbp} f\|_{L^2(\msI_\mbp)}) e_k, 
\eeqs
for $k\geq 3$. Moreover  if $f\in W^{1,\infty}(\msI_\mbp)$,  then there exists a matrix $\mbE$ bounded in $l^2_*$ such that
\beqs
\int_{\msI_\mbp} f(s_\mbp)\tilde  \Theta_j \beta_k \tilde \Theta_k \dd\tilde s_\mbp=  O(\|\nabla_{s_\mbp} f\|_{L^\infty}) \mbE_{kj} +O(\|f\|_{L^\infty}) \beta_j \mbE_{kj}.
\eeqs
\end{lemma}
\begin{proof}
We observe from \eqref{tTheta''} that if we denote 
\beq\label{def-stheta-k}
\theta_k(\tilde s_\mbp) := -(1+\mrp_0)R_0\beta_{\mbp, k}^{-1}  \tilde \Theta_k'(\tilde s_\mbp),
\eeq
then 
\beqs
\theta_k'(\tilde s_\mbp) =\beta_k \tilde\Theta_k(\tilde s_\mbp).
\eeqs
Hence  through integration by parts
\beqs
\int_{\msI_\mbp} f(s_\mbp) \beta_k \tilde \Theta_k \dd\tilde s_\mbp= - \int_{\msI_\mbp} \nabla_{s_\mbp} f(s_\mbp) \theta_k(\tilde s_\mbp)\dd \tilde s_\mbp.
\eeqs
In light of Lemma \ref{Notation-e_i,j}, we only need to show that $\{\theta_k, k\geq 3\}$ are orthogonal to each other.   Since
\beqs
\tilde \Theta_k' = \left\{ \begin{aligned}
& - \beta_{\mbp, k}\tilde \Theta_{k+1}  \qquad & \hbox{if $k$ is odd};\\
& \beta_{\mbp, k} \tilde \Theta_{k-1} \qquad &  \hbox{if $k$ is even}.
\end{aligned}\right.
\eeqs
The orthogonality of $\theta_k$ follows from its definition and orthogonality of $\tilde \Theta_k$ in \eqref{ortho-tTheta}. The first estimate follows from the identity $\beta_{\mbp, j}=\beta_j/(R_0(1+\mrp_0))$. The second estimate is derived from  similar arguments using Lemma \ref{Notation-e_i,j}.  
\end{proof}

\begin{lemma}\label{lem-sAB-w} For $k\geq 3$  and $k\in \Sigma_1$, we the quantities $\msA_k, \msB_k$ defined in \eqref{def-sAB} satisfy the weighted approximations,
\beqs
\begin{aligned}
\beta_k \msA_k(\mbp) & = -\frac{\beta_k^2-1}{R_0^2} \beta_k \mrp_k + O(\|\hat\mbp\|_{\mbV_2^2}\|\hat\mbp\|_{\mbV_3^2})e_k , \\
\beta_k \msB_k(\mbp)&= -\frac{\beta_k^2-1}{R_0^2(1+\mrp_0)^2} \left[\frac{2\beta_k^2-3}{2R_0^2} -\alpha (1+\mrp_0)^2\right]\beta_k \mrp_k
 +O(\|\hat\mbp\|_{\mbV_3^2}\|\hat\mbp\|_{\mbV_4^2}, \|\hat\mbp\|_{\mbV_2^2}\|\hat\mbp\|_{\mbV_5^2})e_k. 
\end{aligned}
\eeqs
\end{lemma}
\begin{proof}
For the first approximation, since  $\kappa_\mbp$ admits expansion \eqref{def-kappa-p,0} we can rewrite the definition  \eqref{def-sAB} of  $\msA_k $ as
\beqs
\beta_k\msA_k(\mbp) =  -\kappa_{\mbp,0} \int_{\msI_\mbp} \beta_k \tilde \Theta_k \dd \tilde s_\mbp +\beta_k \int_{\msI_\mbp}  \mathcal Q_1 \tilde \Theta_k \dd \tilde s_\mbp + \int_{\msI_\mbp}  \mathcal Q_2\beta_{k} \tilde \Theta_k \dd \tilde s_\mbp.
\eeqs 
The first term is zero since $\tilde \Theta_k$ no mass for $k\geq 3$; by \eqref{cQ-1-Theta-k} the second term equals
\beqs
\beta_k \int_{\msI_\mbp}  \mathcal Q_1 \tilde \Theta_k \dd \tilde s_\mbp = \frac{1-\beta_k^2}{R_0^2} \beta_k \mrp_k, 
\eeqs 
and with the aid of Lemma \ref{lem-h-proj-w} and the $H^1(\msI_\mbp)$ estimate  of $\mathcal Q_{2}$ from Lemma \ref{lem-Gamma-p} the third term is bounded as 
\beqs
 \int_{\msI_\mbp}   \mathcal Q_2\beta_{k} \tilde \Theta_k \dd \tilde s_\mbp =O(\|\hat\mbp\|_{\mbV_2^2}\|\hat\mbp\|_{\mbV_3^2}) e_k. 
\eeqs
The first estimate follows. The approximation of $\mathscr B_k(\mbp)$ is derived from similar arguments through the use of the higher-order estimates on quadratic term $\mathcal Q_2$ afforded by Lemma \ref{lem-Gamma-p}.
\end{proof}

\begin{lemma}\label{lem-sR12-w} For $k\geq 3$, the reminders defined in \eqref{def-cI-sR} satisfy the weighted estimates
\beqs
\begin{aligned}
\beta_k \msR_{k,1}&= O(\varep^{7/2} |\sigma_1^*-\sigma| \|\hat\mbp\|_{\mbV_3^2}, \varep^{9/2}\|\hat\mbp\|_{\mbV_5^2})e_k;\\
\beta_k \msR_{k,2}&= O(\varep^{11/2}\|\hat\mbp\|_{\mbV_3^2} )e_k.
\end{aligned}
\eeqs
\end{lemma}
\begin{proof}
This follows from arguments similar to those for Lemma \ref{lem-sR12}, using the weighted estimates from Lemma \ref{lem-h-proj-w}. The details are omitted.
\end{proof}

Lastly, we give a proof a Lemma \ref{lem-xi}. \begin{proof}[Proof of Lemma \ref{lem-xi}]
 We first deal with the unweighted approximation for the three cases $j=0$, $j=1,2$ and $j\geq 3$. 
 
\medskip
\noindent 
Case 1:  $j=0$. Using Lemma \ref{lem-change-of-coord} to replace $\xi_0(s_\mbp)$ in the integral we obtain
\beqs
-\int_{\msI_\mbp}  \xi_0(s_\mbp) \tilde \Theta_k   \dd \tilde s_\mbp=  \int_{\msI_\mbp} \left[\frac{1-(1+\mrp_0)\p_{\mrp_0}\ln A}{A}(R_0+\bar p) -\frac{\tilde s_\mbp\bar p' }{ A}\right] \mbn_0\cdot \mbn_\mbp \tilde \Theta_k \dd \tilde s_\mbp.
\eeqs
From  \eqref{est-bn-p0} and \eqref{est-A-|Gamma-p|} the normal projection $\mbn_\mbp\cdot \mbn_0$ and length normalization $A$ take the value one up to a quadratic correction.  This yields the leading order approximation
\beqs
\begin{aligned}
-\int_{\msI_\mbp}  \xi_0(s_\mbp) \tilde \Theta_k   \dd \tilde s_\mbp=  &\int_{\msI_\mbp} \left[\Big(1-(1+\mrp_0)\p_{\mrp_0}\ln A\Big) (R_0+\bar p) - \tilde s_\mbp\bar p' \right]  \tilde \Theta_k \dd \tilde s_\mbp 
+O(\|\hat\mbp\|_{\mbV_2^2}^2)\mathbb E_{k0}.
\end{aligned}
\eeqs
For $k\geq 1$ the eigenmode $\tilde \Theta_k$  has zero mass in $L^2(\msI_\mbp)$. Using the definition \eqref{def-barp} of $\bar p$, the orthogonality \eqref{ortho-tTheta}, and introducing $\mathbb U$ from \eqref{def-bU} we arrive at the expansion
\beqs
\begin{aligned}
-\int_{\msI_\mbp}  \xi_0(s_\mbp) \tilde \Theta_k   \dd \tilde s_\mbp&=  (1+\mrp_0) \Big(1-(1+\mrp_0)\p_{\mrp_0}\ln A\Big) 2\pi R_0^2 \tilde \Theta_0 \delta_{k0}- (1+\mrp_0)\hat \mbp^T\mathbb U \mathbf B_k\\
& +\Big(1-(1+\mrp_0)\p_{\mrp_0}\ln A\Big)(1+\mrp_0) \mrp_k \mathbf{1}_{\{k\geq 3\}}+O(\|\hat\mbp\|_{\mbV_2^2}^2)\mathbb E_{k0}.
\end{aligned}
\eeqs
 Since $|\p_{\mrp_0}\ln A|\lesssim \|\hat\mbp\|_{\mbV_2^2}$ from Lemma \ref{lem-Gamma-p},  and $\tilde \Theta_0= \Theta_0=1/\sqrt{2\pi R_0}$, we rewrite this expansion as
\beq
\begin{aligned}
-\int_{\msI_\mbp}  \xi_0(s_\mbp) \tilde \Theta_k   \dd \tilde s_\mbp&=  (1+\mrp_0) \Big(1+O(\|\hat\mbp\|_{\mbV_2^2})\Big) \frac{R_0}{\Theta_0}  \delta_{k0}- (1+\mrp_0) \hat \mbp^T\mathbb U \mathbf B_k\\
& + (1+\mrp_0) \mrp_k \mathbf{1}_{\{k\geq 3\}}+O(\|\hat\mbp\|_{\mbV_2^2}^2)\mathbb E_{k0}.
\end{aligned}
\eeq
The second approximation for $k\geq 1$ follows directly. The first estimate for $k=0$ follows by putting the item involving $\mathbb U$ into the error.

\medskip 
\noindent
Case 2: $j=1,2$.  Using Lemma \ref{lem-change-of-coord} to replace $\xi_j(s_\mbp)$ in \eqref{est-bT-2}, we have
\beqs
-\int \xi_j(s_\mbp) \tilde \Theta_k \dd \tilde s_\mbp= \int_{\msI_\mbp} \Theta_0 \bE_j \cdot \mbn_\mbp \tilde \Theta_k \dd \tilde s_\mbp
\eeqs
Since $\mbn_\mbp=\mbn_0 + O(\|\hat\mbp\|_{\mbV_1})$ by Lemma \ref{lem-Gamma-p} and applying identity  \eqref{proj-n0-E-basis} we find
\beqs
\begin{aligned}
-\int \xi_j(s_\mbp) \tilde \Theta_k \dd \tilde s_\mbp& = \int_{\msI_\mbp} \Theta_j  \tilde \Theta_k \dd \tilde s_\mbp +O(\|\hat\mbp\|_{\mbV_2^2}) \mbE_{kj},\\
&=(1+\mrp_0) \delta_{jk} + O(\|\hat\mbp\|_{\mbV_2^2}) \mbE_{kj}.
\end{aligned}
\eeqs
Here we used $\|\tilde \Theta_j -\Theta_j\|_{L^2(\msI_\mbp)}\lesssim \|\hat\mbp\|_{\mbV_2^2}$ by its definition in \eqref{def-barp}. \

\medskip 
\noindent
Case 3: $j\geq 3$. We follow the case $j=1$. Using Lemmas
 \ref{lem-change-of-coord} and \ref{lem-Phi_t-p}, we write  
\beqs
-\int_{\msI_\mbp}  \xi_j(s_\mbp) \tilde \Theta_k   \dd \tilde s_\mbp= \int_{\msI_\mbp} \left[\tilde \Theta_j -\frac{(1+\mrp_0) \p_{\mrp_j} \ln A}{A}(R_0+\bar p) \right]\mbn_0 \cdot \mbn_\mbp  \tilde \Theta_k \dd\tilde s_\mbp.
\eeqs
Approximating the normal projection $\mbn_0\cdot \mbn_\mbp$ with \eqref{est-bn-p0} and using the orthogonality of $\{\tilde \Theta_k\}$ in \eqref{ortho-tTheta} we derive
\beqs
\begin{aligned}
-\int_{\msI_\mbp}  \xi_j(s_\mbp) \tilde \Theta_k   \dd \tilde s_\mbp &= \int_{\msI_\mbp} \left[\tilde \Theta_j -\frac{(1+\mrp_0) \p_{\mrp_j} \ln A}{A}(R_0+\bar p) \right]  \tilde \Theta_k \dd\tilde s_\mbp +O(\|\hat\mbp\|_{\mbV_2^2}^2)\mbE_{kj}\\
&=(1+\mrp_0) \delta_{jk} +O(\|\hat\mbp\|_{\mbV_2^2}^2) \mbE_{kj}.
\end{aligned}
\eeqs
Here we also used the $l^2$ upper bound of $\nabla_\mbp A$ in Lemma \ref{lem-Gamma-p}. The last identity for the case $j\geq 3$ follows.

Now we deal with the weighted case for $k\geq 3$. In fact, similar derivation shows
\beqs
\begin{aligned}
\int_{\msI_\mbp}  \xi_j(s_\mbp) \beta_k\tilde \Theta_k   \dd \tilde s_\mbp
&=(1+\mrp_0) \beta_k \delta_{kj}-(1+\mrp_0)\hat\mbp^T\mbU \mathbf B_k \delta_{j0}+\int_{\msI_\mbp} \mathrm{R}[\xi_j] \beta_k \tilde \Theta_k \dd \tilde s_\mbp,
\end{aligned}
\eeqs
where the remainder term $\mathrm{R}[\xi_j]$ is given by
\beqs
\mathrm{R}[\xi_j]=\left\{\begin{aligned} 
 & - \frac{(1+\mrp_0) \p_{\mrp_0} \ln A}{A}(R_0+\bar p) \mbn_0 \cdot \mbn_\mbp-R_0-\tilde s_\mbp\bar p'  \left(\frac{\mbn_0\cdot \mbn_\mbp}{A} -1\right) \qquad &j=0; \\
&\Theta_0\left(\bE_j\cdot\mbn_\mbp -\tilde \Theta_j\right) \qquad &j=1,2;\\
&\tilde \Theta_j \mbn_0\cdot  (\mbn_\mbp-\mbn_0) - \frac{(1+\mrp_0) \p_{\mrp_j} \ln A}{A}(R_0+\bar p) \mbn_0 \cdot \mbn_\mbp, \qquad  &j\geq 3.
\end{aligned}\right.
\eeqs
We observe $\mathrm R[\xi_j]$  involves only the zero and first derivatives: $\bar p, \bar p'$. Note that $\mbn_\mbp\cdot\mbn_0$ and $A$ is one up to a quadratic error by Lemma \ref{lem-Gamma-p}. The contribution from the remainder is estimated through Lemma \ref{lem-h-proj-w} from the appendix. 
\end{proof}

\bibliographystyle{amsplain}

\end{document}